\documentclass[a4paper,10pt]{article}
\usepackage[utf8]{inputenc}
\usepackage{amssymb,amsmath}
\usepackage{array}
\usepackage{latexsym,mathtools,amsthm}
\usepackage{fancyvrb}
\usepackage{paralist}
\usepackage{booktabs}  
\usepackage[colorlinks,citecolor=blue]{hyperref}

\title{Harmful structures and Killing spinors on unimodular Lie groups}
\author{Diego Conti, Federico A. Rossi and Romeo Segnan Dalmasso}

\newtheorem{theorem}{Theorem}[section]
\newtheorem{lemma}[theorem]{Lemma}
\newtheorem{corollary}[theorem]{Corollary}
\newtheorem{proposition}[theorem]{Proposition}
\theoremstyle{definition}

\newtheorem{example}[theorem]{Example}
\theoremstyle{remark}
\newtheorem{remark}[theorem]{Remark}

\newcommand{\abs}[1]{\left\vert#1\right\vert}
\newcommand{\R}{\mathbb{R}}
\DeclareMathOperator{\im}{im}         
\newcommand{\lie}[1]{\mathfrak{#1}}     
\newcommand{\g}{\lie{g}}
\newcommand{\Z}{\mathbb{Z}}
\newcommand{\C}{\mathbb{C}}
\newcommand{\p}[1]{\mathbb{P}\left(#1\right)}
\newcommand{\CP}{\mathbb{CP}}
\newcommand{\hook}{\lrcorner\,}
\newcommand{\Spin}{\mathrm{Spin}}
\newcommand{\SU}{\mathrm{SU}}

\newcommand{\Cl}{\mathrm{Cl}}
\newcommand{\so}{\mathfrak{so}}
\newcommand{\su}{\mathfrak{su}}
\newcommand{\spin}{\mathfrak{spin}}
\newcommand{\SL}{\mathrm{SL}}
\newcommand{\id}{\operatorname{Id}}   
\newcommand{\Sl}{\lie{sl}}
\newcommand{\Span}[1]{\operatorname{Span}\left\{#1\right\}}
\DeclareMathOperator{\Hom}{Hom}
\DeclareMathOperator{\diag}{diag}
\DeclareMathOperator{\ad}{ad}

\DeclareMathOperator{\Ad}{Ad}
\DeclareMathOperator{\Tr}{tr}
\DeclareMathOperator{\rep}{Re}
\DeclareMathOperator{\imp}{Im}

\newcolumntype{C}{>{$}c<{$}}
\newcolumntype{L}{>{$}l<{$}}
\newcolumntype{R}{>{$}r<{$}}

\usepackage[backend=biber,maxnames=10]{biblatex}
\begin{document}
\VerbatimFootnotes
\maketitle
\begin{abstract}
A pseudo-Riemannian Einstein manifold with a Killing spinor and Killing constant $\lambda$ induces on its nondegenerate hypersurfaces a pair of spinors $\phi,\psi$  and a symmetric tensor $A$, corresponding to the second fundamental form. Viewed as an intrinsic object, $(\phi,\psi,A,\lambda)$ is known as a harmful structure; this notion generalizes nearly hypo and nearly half-flat structures to arbitrary dimension and signature. We show that when $A$ is a multiple of the identity the harmful structure is determined by a Killing spinor. We characterize left-invariant harmful structures on Lie groups in terms of Clifford multiplication by some special elements induced by the structure constants and metric. This enables us to classify left-invariant harmful structures  on unimodular metric Lie groups of definite or Lorentzian signature and dimension $\leq 4$, under the assumption that the symmetric tensor $A$ is diagonalizable over $\R$.   These pseudo-Riemannian Lie groups are principal orbits of cohomogeneity one Einstein metrics of Riemannian, Lorentzian or anti-Lorentzian signature with a Killing spinor.
\end{abstract}

\renewcommand{\thefootnote}{\fnsymbol{footnote}}
\footnotetext{\emph{MSC class 2020}: \emph{Primary} 53C25; \emph{Secondary} 53C50, 53C27, 53C30, 22E60.
}
\footnotetext{\emph{Keywords}: Killing spinor, Einstein metric, pseudo-Riemannian metric, left-invariant metric.}
\renewcommand{\thefootnote}{\arabic{footnote}}

\section*{Introduction}

A Killing spinor on an $n$-dimensional pseudo-Riemannian spin manifold $(M,g)$ is a non-zero section $\Psi$ of the spinor bundle $\Sigma M$ such that for some $\lambda\in\C$
\[
\nabla_X\Psi=\lambda X\cdot\Psi
\]
for any $X\in\Gamma(TM)$, where $\nabla$ is the connection on the spinor bundle induced by the Levi-Civita connection and the dot denotes Clifford multiplication. This condition constrains the scalar curvature $\mathrm{scal}_g$ to be constant and more precisely $\mathrm{scal}_g=4n(n-1)\lambda^2$, thus $\lambda$ is necessarily real or purely imaginary.

Killing spinors are of interest both in physics and mathematics. In physics they have been studied in relation to general relativity since~\cite{Walker1970OnSpacetimes} and later on in supergravity (see~\cite{Duff1986Kaluza-KleinSupergravity}). In the context of left-invariant metrics on Lie groups, indefinite metrics  in  dimension $4$ admitting parallel spinors (i.e., Killing with $\lambda=0$)  have been investigated in the Lorentzian case in~\cite{Cacciatori:2007vn} and in neutral signature in~\cite{dietmarMasato2015}, both in the Levi-Civita case and for more general connections with skew-symmetric torsion.

In mathematics,  Killing spinors arise in different contexts.  On compact Riemannian manifolds, Killing spinors correspond to eigenvectors for the Dirac operator which realize the lowest possible eigenvalue (see~\cite{friedrich1980}). Furthermore, the existence of a Killing spinor imposes strict conditions on the geometry of the manifold; for instance, a Riemannian manifold endowed with a Killing spinor is necessarily Einstein, and Ricci-flat when the spinor is parallel. This is not always the case in general signature;   an example of a non-Einstein Lorentz manifold endowed with a Killing spinor appears in~\cite{Bohle2003KillingManifolds}.

Parallel and Killing spinors have been studied also in relation to the holonomy of the metric both in the Riemannian and  pseudo-Riemannian setting. When the metric is definite, the existence of a parallel spinor characterizes metrics with special holonomy (see~\cite{wangParallelSpinors}); a similar characterization for complete manifolds with a non-parallel Killing spinor, using in particular $G$-structures, has been obtained in~\cite{Bar:RealKillingSpinors} for $\lambda$ real and~\cite{Baum:CompleteRiemannian} for $\lambda$ imaginary.

In the indefinite setting, the situation is not as clear-cut since indecomposability and irreducibility of pseudo-Riemannian manifolds are not equivalent. Nonetheless, an analogous study as in the Riemannian setting, mostly focused on the Lorentzian signature, has been undertaken in  \cite{baumLorentzian2000,baum2008codazzi,Leistner:Classification,Bohle2003KillingManifolds}; see also the more recent \cite{Alekseevsky_Cortes_Leistner_2023}, which covers all signatures.

A classical method to construct metrics with a Killing spinor is by considering metrics of cohomogeneity one; the principal orbits are then hypersurfaces, and inherit an invariant geometric structure from the ambient. This motivates the  study of the intrinsic geometry induced on an orientable nondegenerate hypersurface $M\subset Z$ by the existence of a parallel or Killing spinor on $Z$. Up to changing the sign of the metric on $Z$, one can assume that the normal direction is spacelike. It was shown in \cite{Bar_Gauduchon_Moroianu_2005} that the restriction of a parallel spinor to $M$  satisfies
\begin{equation}
 \label{eqn:generalizedKilling}
\nabla_X\Psi=\frac12A(X)\cdot\Psi,
\end{equation}
where $A$ is the Weingarten tensor and $\cdot$ denotes Clifford multiplication on $M$; in general, a spinor satisfying~\eqref{eqn:generalizedKilling} for some symmetric $(1,1)$ tensor $A$ is called a \emph{generalized Killing spinor}. Homogeneous Riemannian metrics with a generalized Killing spinor have been studied in \cite{AGRICOLA2023102014,moroianuSemmelmann2014Einstein,Moroianu_Semmelmann_2014,Conti_Salamon_2007,Conti:SU3}; a classification of $3$-dimensional Lie groups admitting a left-invariant Riemannian metric and generalized Killing spinor has been recently obtained in~\cite{Artacho_2025}. An analogous classification for negative-definite metrics (corresponding to $Z$ of signature $(1,3)$) has been obtained in~\cite{Murcia_Shahbazi_2022} using the language of polyforms introduced in~\cite{Cortes_Lazaroiu_Shahbazi_2021}.

By contrast, it was shown in~\cite{conti_segnan2024} that
a Killing spinor on $Z$ induces a pair $(\phi,\psi)$ of spinors on the hypersurface $M$, which satisfy
\begin{equation}\label{eqn:harmfulsystem}
\begin{cases}
	\nabla_X\psi=\frac12A(X)\cdot\psi+\lambda X\cdot\phi\\
	\nabla_X\phi=\lambda X\cdot\psi-\frac12A(X)\cdot\phi,
\end{cases}\end{equation}
where again $A$ is the Weingarten operator and $\lambda$ is the Killing constant; when $\dim M=n$ is even, $\psi$ and $\phi$ are related by $\phi=i^{1-q-\frac n2}\omega\cdot\psi$, where $\omega$ is the volume element and the signature of the metric on $M$ is $(p,q)$. If one additionally assumes that $Z$ is Einstein, it follows from~\cite{Koiso_1981} that
\begin{equation}
\label{eqn:koiso}
d\Tr A+\delta A=0.
\end{equation}
A quadruple $(\phi,\psi,A,\lambda)$ satisfying~\eqref{eqn:harmfulsystem} and~\eqref{eqn:koiso} (and the compatibility condition between $\phi$ and $\psi$ in even dimensions) is called a \emph{real harmful structure}.

It was proved in \cite{Ammann_Moroianu_Moroianu_2013,conti_segnan2024} that, in the real analytic setting, both constructions can be inverted: if \eqref{eqn:koiso} holds, a generalized Killing spinor extends to a manifold with a parallel spinor, and a real harmful structure always extends to a manifold with a Killing spinor. In definite signature, \eqref{eqn:koiso} follows from either \eqref{eqn:generalizedKilling} or \eqref{eqn:harmfulsystem}.

In specific instances, the geometry of a hypersurface in a manifold with a Killing spinor can be studied in terms of $G$-structures; in particular, hypersurfaces inside a nearly-K\"ahler or nearly parallel $\mathrm{G}_2$ manifold has an induced structure called  respectively nearly hypo or nearly half-flat (\cite{Fernandez_Ivanov_Munoz_Ugarte_2006,FoscoloHaskins,Singhal_2024,Conti:Embedding}), and
hypersurfaces inside a globally hyperbolic Lorentzian four-manifold have an induced $\{e\}$-structure, called a Killing-Cauchy pair (\cite{Murcia_Shahbazi_2023}). Harmful structures  generalize these geometries, without restrictions on the dimension, signature, or stabilizer of the spinor.

\smallskip
This paper studies definite or Lorentzian metrics on Lie groups  admitting a harmful structure or a Killing spinor; we will assume that metric and spinors are left-invariant, enabling us to work at the Lie algebra level, that the symmetric tensor $A$ is diagonalizable over $\R$, which is a non-trivial condition if the metric is indefinite, and that the Lie algebra is unimodular. We show that harmful structures with $A$ a multiple of the identity are determined by a Killing spinor, and vice versa a Killing spinor determines a family of harmful structures of that form (Proposition~\ref{prop:AmultipleIdentity}, Proposition~\ref{prop:PhiPsiLinDip}). In dimension $3$,  every harmful structure with $A$ diagonalizable on a  unimodular Lie algebra arises in this way (Corollary~\ref{cor:harmful3dim}); in fact, we show that a unimodular Lie algebra only admits a Killing spinor if it is abelian or simple, i.e. one of $\R^3$, $\su(2)$ or $\Sl(2,\R)$, endowed with a constant curvature metric (Theorem~\ref{thm:killing3dim}). In dimension $4$, we show that a Killing spinor on a unimodular Lie algebra is necessarily parallel (Proposition~\ref{prop:definiteAmultipleodentity}, Theorem~\ref{thm:killing31}). The existence of a parallel spinor forces the metric to be flat in Riemannian signature, leading to a straightforward characterization of the Lie algebras; we obtain a similar classification in Lorentzian signature, which also includes non-flat metrics with nilpotent holonomy $\Hom(\R^2,\R)$ (Corollary~\ref{cor:4dimParallel}). In contrast with the $3$-dimensional case, harmful structures of signature $(4,0)$ and $(3,1)$ are not necessarily induced by a Killing spinor;  we obtain a classification in Theorem~\ref{main:generaltheorem4}. We note that solving an appropriate ODE using these metrics with initial data would determine (possibly incomplete) cohomogeneity one Einstein metrics of Riemannian, Lorentzian or anti-Lorentzian signature admitting a Killing spinor.

We classify harmful structures up to Lie algebra automorphisms and the action of $\Spin(p,q)$. The key observation is that, relative to a fixed orthonormal basis, the operators $-\epsilon_j e_j\cdot\nabla_{e_j}$  (whose sum is minus the Dirac operator) are represented by elements $L_j$ of the Clifford algebra, and these elements determine the structure constants of the Lie algebra.  In dimension $3$, each $L_j$ is a multiple of the volume form. In signatures $(4,0)$ and $(3,1)$, given a harmful structure, if the orthonormal basis is chosen to be a basis of eigenvectors of $A$, the elements $L_j$ are determined uniquely; moreover, the harmful structure determines two vectors, denoted $*\xi$ and $*M$ in the paper, related respectively to the Dirac current of $\psi$ and the Dirac operator of the metric. Since there is a relation between $*M$, $*\xi$ and the eigenspaces of $A$, we are able to choose the orthonormal basis of eigenvectors to reduce the possibilities for the pair $(\psi_+,{*}M)$, which turns out to  determine $\phi,\psi$ completely. This allows us to classify harmful structures up to the action of $\Spin(p,q)$ and Lie algebra automorphisms, the latter taken into account by the fact that structure constants are computed relative to a fixed orthonormal basis.

The paper is structured as follows. In Section~\ref{sec:Clifford} we  fix notations, recall the formulae for Clifford multiplication in dimensions $3,4$, and prove basic lemmas describing orbits of the spinor action in low dimension. In Section~\ref{sec:harmfulAspects} we study harmful structures on manifolds, without assuming invariance, and describe some symmetries of~\eqref{eqn:harmfulsystem} that will be used in the classification. In Section~\ref{section:harmfulOnMetricLie} we turn to the invariant setting, proving some technical results involving the multivectors $L_j$, $M$ and $\xi$. In Section~\ref{sec:ClassificationDim3} we classify harmful structures on unimodular Lie algebras of dimension $3$. In Section~\ref{sec:general4dim} we specialize the results of Section~\ref{section:harmfulOnMetricLie} to signatures $(4,0)$ and $(3,1)$ and prove a relation between $*M$, $*\xi$ and the eigenspaces of $A$. We carry out the  classification of $4$-dimensional harmful structures with $A$ a multiple of the identity in Section~\ref{sec:Amultipleofidentity}; we treat the general case in Section~\ref{sec:40} for Riemannian signature and Section~\ref{sec:31} in Lorentzian signature.

\subsection*{Acknowledgments}
The authors are partially supported by the PRIN project n. 2022MWPMAB ``Interactions between Geometric Structures and Function Theories''. The authors also acknowledge Gruppo Nazionale per le Strutture Algebriche,
Geometriche e le loro Applicazioni (GNSAGA) of Istituto Nazionale di Alta
Matematica (INdAM). D. Conti acknowledges  the MIUR Excellence Department Project CUP I57G22000700001 awarded to the Department of Mathematics, University of Pisa.

\section{The Clifford algebra and spinor representation}
\label{sec:Clifford}
In this section we introduce explicit formulae for Clifford multiplication in dimensions three and four, and prove some basic lemmas concerning orbits in low dimensional spinor representations.

Let $e_1,\dotsc, e_n$ be the standard basis of $\R^n$, let $e^1,\dotsc, e^n$ be the dual basis, and fix a scalar product
\begin{equation}
 \label{eqn:generaldiagonalmetric}
 g=\epsilon_1e^1\otimes e^1+\dots + \epsilon_n e^n\otimes e^n,
\end{equation}
of signature $(p,q)$, where the $\epsilon_j$ are $\pm1$. We will write $\R^{p,q}$ instead of $\R^n$ to imply that such a metric has been fixed. We will always consider the orientation of $\R^n$, or $\R^{p,q}$, for which the basis $e_1,\dotsc,e_n$ is positively oriented.

Let $\Cl(p,q)$ denote the real Clifford algebra of $\R^{p,q}$. By definition, the spinor representation $\Sigma$ is an irreducible complex representation of $\Cl(p,q)$. We can identify $\R^{p,q}$ with a vector subspace of $\Cl(p,q)$, so that its action on $\Sigma$ restricts to Clifford multiplication, i.e. a bilinear map
\begin{equation*}
 \label{eqn:clifford}
 \R^{p,q}\otimes\Sigma\to\Sigma, \quad (v,\psi)\mapsto v\cdot\psi.
\end{equation*}
The symbol $\cdot$ will also be used for multiplication in the Clifford algebra. The group $\Spin(p,q)$ is the subgroup of $\Cl(p,q)$ generated by products of two unit vectors in $\R^{p,q}$. It acts on $\R^{p,q}$ by conjugation in $\Cl(p,q)$ and its  Lie algebra is naturally identified with
\[\spin(p,q)=\Span{e_j\cdot e_k\mid j<k}\subset\Cl(p,q);\]
at the Lie algebra level, the representation of $\spin(p,q)$ on $\R^{p,q}$ is given by
\[(e_j\cdot e_k)v =e_j\cdot e_k\cdot v - v\cdot e_j\cdot e_k.\]
Let $\omega=e_1\cdot\ldots\cdot e_n$ denote the volume element. If $n$ is even, $\omega$ acts on $\Sigma$ with two eigenspaces $\Sigma_\pm$, which are irreducible representations of $\Spin(p,q)$; explicitly,
\begin{equation}
\label{eqn:evenvolumeacts}
\omega\cdot u=	i^{q+n(n+1)/2}(u_+-u_-),
\end{equation}
where $u_\pm$ represents the component of $u$ in $\Sigma_\pm$. If $n$ is odd,
\begin{equation}
\label{eqn:oddvolumeacts}
\omega\cdot u= i^{q+n(n+1)/2}u
\end{equation}
The choices of sign in~\eqref{eqn:evenvolumeacts} and~\eqref{eqn:oddvolumeacts} are a matter of convention.

We will need explicit formulae for  Clifford multiplication relative to a basis $\{u_0,\dotsc, u_m\}$ of $\Sigma$. In Appendix~\ref{appendix:A} we will give general formulae, which have been obtained from~\cite{Baum_Kath_1999} with minimal variations; here, we will only consider the cases of dimensions three and four, which are those used in the paper. In three dimensions, without loss of generality we we will rewrite~\eqref{eqn:generaldiagonalmetric} in the form
\begin{equation}
\label{eqn:standard3metric}
g=\epsilon (e^1\otimes e^1+e^2\otimes e^2)+\epsilon_3e^3\otimes e^3.
\end{equation}
Define $\tau,\tau_3$ to equal $1$ or $i$ so that $\tau^2=\epsilon$ and $\tau_3^2=\epsilon_3$. Applying the general recipe of Appendix~\ref{appendix:A}, Clifford multiplication is given by
\begin{equation}\label{eqn:cliff3dim}
	\begin{aligned}
		e_1\cdot u_{0}&=i\tau u_{1},&
		e_1\cdot u_{1}&=i\tau u_{0},\\
		e_2\cdot u_{0}&=\tau u_{1},&
		e_2\cdot u_{1}&=- \tau u_{0},\\
		e_3\cdot u_{0}&=i\tau_3u_{0},&
		e_3\cdot u_{1}&=-i\tau_3u_{1}.
	\end{aligned}
\end{equation}

In Section~\ref{sec:ClassificationDim3}, we will need the following.
\begin{lemma}\label{lemma:orbitseta3dim}
Let $\Sigma$ be the spinor representation of $\Spin(p,3-p)$. Up to rescaling and acting by $\Spin(p,3-p)$, a nonzero element $\psi$ of $\Sigma$ can be assumed to satisfy:
\begin{enumerate}
    \item $\psi=u_0$ if the metric is definite;
    \item $\psi=u_0$ or $\psi=u_0+u_1$ if the metric is indefinite.
\end{enumerate}
\end{lemma}
\begin{proof}
If $g$ is definite, then $\Spin(p,3-p)\cong\SU(2)$ and $\Sigma=\C^2$; the action on $\p{\Sigma}=\CP^1$ is transitive. Thus, we may assume $\psi=u_0$.

If $g$ is indefinite, then $\Spin(p,q)_0\cong\SL(2,\R)$ which acts on
\begin{equation}\label{eqn:sigma21}
\Sigma=\R^2\oplus\R^2=\operatorname{Span}_\R(u_0,iu_1)\oplus\operatorname{Span}_\R(iu_0,-u_1).
\end{equation}
There are two orbits in $\p{\Sigma}$, corresponding to the spinors $u_0$ and $u_0+u_1$. Indeed, identify a spinor as a pair $(\xi,\eta)\in\R^2\oplus\R^2$ as in \eqref{eqn:sigma21}. If $\xi,\eta$ are linearly dependent, the pair $(\xi,\eta)$ is in the same orbit as a vector $(ke_1,he_1)$, i.e. the spinor is a multiple of $u_0$. If they are linearly independent, $\SL(2,\R)$ can take the basis $\{\xi,\eta\}$ to the basis $\{e_1,he_2\}$; rescaling the spinor we can assume $h=\pm1$, and then acting by $e_1\cdot e_2$ we can assume that $h=-1$.
\end{proof}

In four dimensions, taking a positive definite metric of the form \eqref{eqn:generaldiagonalmetric} leads to the following Clifford multiplication table:
\begin{equation}
 \label{eqn:cliff40}
\begin{aligned}
{e}_1\cdot u_{0}&=i  u_{1}&
{e}_1\cdot u_{1}&=i  u_{0}&
{e}_1\cdot u_{2}&=i  u_{3}&
{e}_1\cdot u_{3}&=i  u_{2}\\
{e}_2\cdot u_{0}&=u_{1}&
{e}_2\cdot u_{1}&=- u_{0}&
{e}_2\cdot u_{2}&=u_{3}&
{e}_2\cdot u_{3}&=- u_{2}\\
{e}_3\cdot u_{0}&=-i  u_{2}&
{e}_3\cdot u_{1}&=i  u_{3}&
{e}_3\cdot u_{2}&=-i  u_{0}&
{e}_3\cdot u_{3}&=i  u_{1}\\
{e}_4\cdot u_{0}&=-  u_{2}&
{e}_4\cdot u_{1}&=  u_{3}&
{e}_4\cdot u_{2}&=  u_{0}&
{e}_4\cdot u_{3}&=-  u_{1}
\end{aligned}
\end{equation}
The other signatures are obtained by adding a factor of $i$ to time-like basis elements. For instance, for the metric
\begin{equation}
 \label{eqn:standard31metric}
g=e^1\otimes e^1+e^2\otimes e^2+e^3\otimes e^3-e^4\otimes e^4,
\end{equation}
the last line of \eqref{eqn:cliff40} should be replaced by
\begin{equation}
\label{eqn:cliff31}
{e}_4\cdot u_{0}=-i  u_{2}\quad
{e}_4\cdot u_{1}=i  u_{3}\quad
{e}_4\cdot u_{2}=i  u_{0}\quad
{e}_4\cdot u_{3}=-i  u_{1}.
\end{equation}

We will not use the exact analogue of Lemma~\ref{lemma:orbitseta3dim} in dimension four. Indeed, it turns out that the spinors $\psi=\psi_++\psi_-$ that we will need to consider satisfy a condition of the form $\psi_-=v\cdot\psi_+$ for some vector $v$. Thus, what will be relevant is the classification of pairs $(v,\psi_+)$ up to $\Spin(p,4-p)$ action. In definite signature, this is trivial. In the Lorentzian setting, we obtain the following:
\begin{lemma}\label{lemma:orbitsPsiM}
Let $v\in\R^{3,1}$ and $\eta\in\Sigma$ with $\eta_+\ne0$; fix a basis so that the metric takes the form \eqref{eqn:standard31metric}. Up to $\Spin(3,1)$-symmetry and rescaling of the spinor, we can assume $\eta_+=u_0$ and
\[v\in\{0,\ me_1,\ e_3-e_4,\ e_3+e_4,\ m e_3,\ m e_4\mid m>0\}.
\]
\end{lemma}
\begin{proof}
Since $\Spin(3,1)_0=\SL(2,\C)$ acts  transitively on $\mathbb{P}(\Sigma_{3,1}^+)=\CP^1$, we can assume that $\eta_+=u_0$. The stabilizer of $[u_0]$ in $\p{\Sigma_{3,1}^+}$ has Lie algebra
\[\Span{e_1\cdot e_2, e_3\cdot e_4, e_1\cdot e_3-e_1\cdot e_4, e_2\cdot e_3- e_2\cdot e_4}.\]

It is clear that the torus generated by the first two elements acts on $\R^{3,1}$ by rotations in the plane $\Span{e_1,e_2}$ and hyperbolic rotations in the plane $\Span{e_3,e_4}$. To compute the action of the last two elements, we consider the basis $\{e_1,e_2,e_3-e_4,\frac12(e_3+e_4)\}$, so that
\begin{gather*}
	[e_1\cdot (e_3 -e_4), x_1e_1+x_2e_2+x_3(e_3-e_4)+\frac12x_4(e_3+e_4)]= 2x_1(e_3-e_4)-2x_4e_1,\\
	[e_2\cdot (e_3 -e_4), x_1e_1+x_2e_2+x_3(e_3-e_4)+\frac12x_4(e_3+e_4)]=2x_2(e_3-e_4)-2x_4e_2.
\end{gather*}
The corresponding one-parameter subgroups in $\Spin(3,1)$ act as
\begin{equation}
	\label{eqn:lemma:31:oneparametersubgroups}
	\begin{gathered}
		\exp \begin{pmatrix}
			0   & 0 & 0 & -2 x \\
			0   & 0 & 0 & 0    \\
			2 x & 0 & 0 & 0    \\
			0   & 0 & 0 & 0
		\end{pmatrix}
		=
		\begin{pmatrix}
			1   & 0 & 0 & -2 x   \\
			0   & 1 & 0 & 0      \\
			2 x & 0 & 1 & -2 x^2 \\
			0   & 0 & 0 & 1
		\end{pmatrix}
		\\
		\exp \begin{pmatrix}
			0 & 0   & 0 & 0    \\
			0 & 0   & 0 & -2 x \\
			0 & 2 x & 0 & 0    \\
			0 & 0   & 0 & 0
		\end{pmatrix}
		=
		\begin{pmatrix}
			1 & 0   & 0 & 0      \\
			0 & 1   & 0 & -2 x   \\
			0 & 2 x & 1 & -2 x^2 \\
			0 & 0   & 0 & 1
		\end{pmatrix}
	\end{gathered}
\end{equation}

We compute the orbits for the action of the stabilizer of $[u_0]$ on $\R^{3,1}$. Obviously, the zero vector forms an orbit.

If $0\neq v\in\Span{e_3,e_4}$, we can act by hyperbolic rotations, whose orbits are hyperbola or half-lines from the origin. Additionally,  the volume element $\omega\in\Spin({3,1})$ stabilizes $[u_0]$, since $u_0$ is chiral, and acts on a vector $v\in\R^{3,1}$ as $\omega\cdot v\cdot \omega^{-1}=-v$. Thus, depending on causality of $v$, we can assume it to be $e_3\pm e_4$ if light-like, $me_3$ if space-like, or $m e_4$ if time-like, where $m>0$.

If $v=x_1e_1+x_2e_2+x_3(e_3-e_4)+\frac12x_4(e_3+e_4)$, with $x_4\neq0$, then we can use the one-parameter subgroups \eqref{eqn:lemma:31:oneparametersubgroups} to obtain an element of $\Span{e_3,e_4}$ and proceed as before.

If $v=x_1e_1+x_2e_2+x_3(e_3-e_4)$ with $x_1$ or $x_2$ nonzero, then we can  use the one-parameter subgroups \eqref{eqn:lemma:31:oneparametersubgroups} to obtain an element of $\Span{e_1,e_2}$. We can then act by a rotation; thus, we can assume $v=me_1$, $m>0$ and the statement follows.
\end{proof}
We conclude this section by introducing more algebraic language. It will be convenient to introduce the Hodge star operator on multivectors as in the Riemannian setting by setting
\[\alpha\wedge *\beta =g(\alpha,\beta)\omega,\]
with $\omega$ denoting the volume element, so that if the metric takes the form \eqref{eqn:generaldiagonalmetric} we have
\[*(e_1\wedge \dotsb\wedge e_k) = \epsilon_1\dotsm \epsilon_k e_{k+1}\wedge\dotsb \wedge e_n.\]
In particular, for multivectors of degree $n-1$ we obtain
\begin{equation}
\label{eqn:starproduct}
e_1\wedge \dots \wedge e_{n-1}\cdot \psi =(-1)^{p}*(e_1\wedge \dots \wedge e_{n-1})\cdot \omega\cdot \psi.
\end{equation}
By linearity, we can replace $e_1\wedge \dotsc \wedge e_{n-1}$ with an arbitrary multivector $\beta$ of degree $n-1$.

In particular, if $n$ is even and $\psi_\pm$ has chirality $\pm1$, then \eqref{eqn:evenvolumeacts} gives
\begin{equation}\label{eqn:starproductOnChiral}
\beta\cdot \psi_\pm = \pm i^{2p+q+n(n+1)/2}{*}\beta\cdot \psi_\pm=\pm i^{-q+n(n+1)/2} {*}\beta\cdot\psi_\pm, \quad \deg\beta=n-1.
\end{equation}

If $\gamma$ is in $\Lambda^p(\R^n)^*$ and $v_1,\dotsc, v_k$ in $\R^n$, we will write
$(v_{1}\wedge \dotsc \wedge v_{k})\hook \gamma$ for the element of $\Lambda^{p-k}(\R^n)^*$ given by
\[((v_{1}\wedge \dotsc \wedge v_{k})\hook \gamma) (w_{k+1},\dotsc, w_p)= \gamma(v_1,\dotsc, v_k,w_{k+1},\dotsc, w_p);\]
the analogous interior product $\Lambda^k (\R^n)^*\otimes\Lambda^p \R^n\to\Lambda^{p-k}\R^n$ will be denoted by the same symbol.

When a metric is fixed, we will consider the standard musical isomorphisms $\flat$ and $\sharp$, extended to multivectors and forms by the following conventions:
\[(v_1\wedge \dotsm \wedge v_k)^\flat=v_1^\flat\wedge \dotsc \wedge v_k^\flat, \quad
(\alpha_1\wedge\dotsm \wedge \alpha_k)^\sharp=\alpha_1^\sharp\wedge \dots \wedge \alpha_k^\sharp.\]
When applying the operator $\hook$,  we will implicitly compose with musical isomorphisms to the left operand when appropriate, and write e.g. $(v_1\wedge \dotsc \wedge v_k)\hook \beta$ instead of $(v_1^\flat\wedge \dotsc \wedge v_k^\flat)\hook \beta $ for $v_1,\dotsc, v_k$ in $\R^n$ and $\beta$ in  $\Lambda^p \R^n$.

Given  $X$ in $\R^n$ and $\beta$ in $\Lambda^p \R^n$, one can easily check that
\begin{equation}
 \label{eqn:hookandwedgeadual}
X\wedge (*\beta) = (-1)^{p+1} {*}(X\hook \beta).
 \end{equation}
Finally, given vector fields $X$ and $Y$, we will occasionally view $X\wedge Y$ as an element of the Clifford algebra, defined as $\frac12(X\cdot Y-Y\cdot X)$. This is consistent with the fact that a multivector $X\wedge Y$ acts on spinors as multiplication by $X\cdot Y$ when $X,Y$ are orthogonal.

\section{The harmful aspect of spin geometry}\label{sec:harmfulAspects}
In this section we recall the notion of a harmful structure $(\phi,\psi,A,\lambda)$ introduced in \cite{conti_segnan2024} and discuss some special cases, as well as the extent to which different choices of $\phi,\psi,A$ and $\lambda$ may correspond to the same geometric situation.

We will work on an $n$-dimensional spin manifold $M$ with a metric of signature $(p,q)$; the bundle of spinors is a complex vector bundle with fibre $\Sigma$ and structure group $\Spin(p,q)$ that will be denoted by $\Sigma M$, and a spinor is a section of $\Sigma M$. If the dimension is even, $\Sigma M$ splits as $\Sigma_+M\oplus\Sigma_-M$; sections of $\Sigma_+M$ and $\Sigma_-M$ are called chiral spinors, respectively of positive and negative chirality.

\begin{remark}
 \label{remark:reverseorientation}
If  we change the orientation the bundle of spinors remains the same, but since the volume form changes its sign, this has the effect of
interchanging $\Sigma_+$ and $\Sigma_-$ in even dimensions. If the dimension is odd, a negatively oriented basis $e_1,\dotsc, e_n$ determines a positively oriented basis $-e_1,\dotsc, -e_n$. Therefore, we can canonically identify the spinor bundles relative to opposite orientations by considering the Clifford multiplication  map
\[e_j\widehat{\cdot} u = -e_j\cdot u;\]
then $\cdot$ and $\widehat\cdot$ satisfy \eqref{eqn:oddvolumeacts} relative to opposite orientations.
\end{remark}

Let $\nabla$ be the Levi-Civita connection; in an orthonormal frame, the Levi-Civita connection form $\theta$ is defined by $\nabla e_j=\theta_{kj}\otimes e_k$, and the connection on a spinor $\psi=\sum a^hu_h$ is given by
\begin{equation}\label{eqn:ConnectionOnSpinorBasis}
\nabla a^hu_h = da^h\otimes u_h+ \frac12 \sum_{k<j}a^h\theta_{kj}\otimes (e^j)^\sharp \cdot e_k\cdot u_h.
\end{equation}
To generalize to arbitrary frames, define
\[\theta^\sharp = \sum_{k,j}\theta_{kj}\otimes (e^j)^\sharp\otimes e_k\in\Omega^1(M,\Lambda^2TM);\]
the fact that $\theta^\sharp(X)$ is skew-symmetric and can be viewed as a section of $\Lambda^2TM$ follows from the fact that  $\theta(X)$ lies in $\so(p,q)$. Then~\eqref{eqn:ConnectionOnSpinorBasis} becomes
\begin{equation}\label{eqn:ConnectionOnSpinor}
\nabla a^hu_h = da^h\otimes u_h+ \frac14a^h\theta^\sharp\cdot  u_h.
\end{equation}

A \emph{Killing spinor} is a spinor $\psi$ satisfying
\[\nabla_X \psi = wX\cdot\psi,\]
where $w$ is a constant, either real or purely imaginary. A \emph{generalized Killing spinor} is a spinor $\psi$ such that
\[\nabla_X \psi = \frac12A(X)\cdot\psi,\]
where $A$ is a real symmetric tensor of type $(1,1)$.

\begin{remark}
\label{remark:changesign}
The bundle of spinors of $(M,g)$ and $(M,-g)$ can be identified in the following way.  We can make an identification
\begin{equation}
\label{eqn:RpqRqp}
 \R^{q,p}\cong(-1)^q i\R^{p,q}\subset \Cl(p,q)\otimes\C.
 \end{equation}
Since $\Spin(p,q)$ is generated by products $v_1\cdot v_2$ of two unit vectors in $\R^{p,q}\subset\Cl(p,q)$, $\Spin(q,p)$ is generated by the corresponding products $-v_1\cdot v_2$. Even though $\Spin(p,q)$ and $\Spin(q,p)$ are different groups, the spin representations are the same because they are obtained by restricting irreducible representations of $\mathbb{C}l(p+q)$. The coefficient in \eqref{eqn:RpqRqp} has been chosen  so that, for $n$ odd, the volume element of $\R^{q,p}$ acts on $\Sigma$ as multiplication by
\[(-1)^q i^{n} i^{q+n(n+1)/2}=i^{p+n(n+1)/2},\]
consistently with~\eqref{eqn:oddvolumeacts}; see also \cite[Theorem 9.43]{HarveySpecter}.

Thus, a spinor on $(M,g)$ can be interpreted as a spinor on $(M,-g)$, but the Clifford multiplication is changed as above by a factor of $(-1)^qi$. The Levi-Civita connection form $\theta$ is the same for $g$ and $-g$, but $\theta^\sharp$ changes its sign upon reversing the metric; this compensates the fact that  Clifford multiplication by an element of $\Lambda^2TM$ also changes sign, and therefore \eqref{eqn:ConnectionOnSpinorBasis} shows that the covariant derivative of a spinor is the same relative to the two metrics. In particular, real Killing spinors become imaginary Killing spinors and vice versa.
\end{remark}

In~\cite{conti_segnan2024} we introduced the notion of harmful structure. Given an $n$-dimensional pseudo-Riemannian spin manifold $(M,g)$ of signature $(p,q)$, a \textit{real harmful structure} on $(M,g)$ is a quadruple $(\phi,\psi,A,\lambda)$, where $\phi,\psi$ are nowhere vanishing spinors, $A$ is a symmetric $(1,1)$-tensor and $\lambda$ a complex constant such that
\begin{equation}
\label{eqn:systemSpace}
\begin{cases}
	\nabla_X\psi=\frac12A(X)\cdot\psi+\lambda X\cdot\phi\\
	\nabla_X\phi=\lambda X\cdot\psi-\frac12 A(X)\cdot\phi\\
	d\Tr A+\delta A=0
\end{cases}
\end{equation}
If $n$ is even, we further require that $\phi=i^{\frac{2-3q-p}2}\omega\cdot\psi$.

The motivation for studying this type of structure is that harmful structures appear naturally on nondegenerate hypersurfaces of manifolds with a Killing spinor, with the symmetric tensor $A$ corresponding to the Weingarten operator. Explicitly, let $Z$ be an $(n+1)$-dimensional pseudo-Riemannian spin manifold of signature $(p+1,q)$ carrying a Killing spinor $\Psi$ with Killing constant $\lambda$, and let $M\subset Z$ be an oriented hypersurface of signature $(p,q)$; denote by $\nu$ the unit normal compatible with the orientations and by $A$ the Weingarten operator. The spinor bundle $\Sigma M$ can be identified with the restriction to $M$ of $\Sigma Z$ if $n$ is even, and of $\Sigma_+Z$ if $n$ is odd; in both cases, Clifford multiplication on $M$ is related to Clifford multiplication on $Z$ by
\begin{equation}
\label{eqn:inducedCliffordMult}
X\cdot u=\nu\cdot_Z X\cdot_Z u.
\end{equation}
Under this identification, $M$ inherits a real harmful structure $(\phi,\psi,A,\lambda)$ (see \cite{conti_segnan2024}). If $n$ is even, $\psi$ is the restriction of $\Psi$, and $\phi=i^{\frac{2-3q-p}2}\omega\cdot\psi$ is the restriction of $\nu\cdot_Z\Psi$. If $n$ is odd, write $\Psi=\Psi^++\Psi^-$ for the decomposition of $\Psi$ into chiral components; then $\psi$ and $\phi$ are the restrictions of $\Psi^+$ and $\nu\cdot_Z\Psi^-$ respectively, so that $\Psi$ coincides with $\psi-\nu\cdot_Z\phi$ along $M$.

\begin{remark}
A standard computation from \eqref{eqn:systemSpace} shows that the scalar curvature $s$ is determined by
\begin{equation*}\label{eqn:scalarcurvatureSpace}
    s\psi=(4n(n-1)\lambda^2-\Tr(A^2)+(\Tr A)^2)\psi-2(d\Tr A+\delta^gA)\cdot\psi
\end{equation*}
(see \cite{conti_segnan2024}). In particular, this implies that the constant  $\lambda$ is either real or imaginary.
\end{remark}

\begin{remark}
If we allow exactly one of $\phi$ or $\psi$ to be identically zero in \eqref{eqn:systemSpace}, the other is necessarily a generalized Killing spinor. Hence, with reference to \eqref{eqn:systemSpace}, harmful structures and generalized Killing spinors appear in complementary situations.
\end{remark}

\begin{remark}
\label{remark:Anotunique}
In a harmful structure, the constant $\lambda$ is uniquely determined by $\phi$, $\psi$ and $A$. However, the tensor $A$ is not uniquely determined by $\phi$ and $\psi$. Indeed, given a harmful structure, suppose that there is a distribution $\mathcal D$ such that any $X$ contained in $\mathcal D$ satisfies
\begin{equation}
\label{eqn:annihilatesphipsi}
X\cdot\phi=0=X\cdot\psi.
\end{equation}
Then $\mathcal D$ is totally isotropic, as~\eqref{eqn:annihilatesphipsi} implies $0=X\cdot X\cdot\phi=-g(X,X)\phi$. It follows that we can locally choose a frame $e_1,\dotsc, e_n$ such that $e_1,\dotsc, e_k$ generate $\mathcal D$ and the metric takes the form
\[e^1\odot e^{k+1}\dots + e^k\odot e^{2k}+\epsilon_{2k+1}e^{2k+1}\otimes e^{2k+1}+\dots + \epsilon_ne^n\otimes e^n.\]
Consider a tensor of the form $H=h_1e^{k+1}\otimes e_1+\dotsc +h_{k} e^{2k}\otimes e_{k}$, which is symmetric relative to the above metric. Since $H$ takes values in $\mathcal D$, \eqref{eqn:annihilatesphipsi} implies that replacing $A$ with $A+H$ does not affect the first two equations in~\eqref{eqn:systemSpace}; moreover, $\Tr H=0$. Therefore, $(\phi,\psi,A+H,\lambda)$ is again a harmful structure, provided that the functions $h_1,\dotsc, h_k$ are chosen so that $\delta H=0$.
\end{remark}

In the case where $\phi,\psi$ are linearly independent, there is no other ambiguity than the one described in Remark~\ref{remark:Anotunique}, as seen in the following:
\begin{proposition}
\label{prop:Auniqueupto}
If $(\phi,\psi,A,\lambda )$ is a harmful structure and $\phi,\psi$ are linearly independent, then any harmful structure of the form $(\phi,\psi,\hat A,\hat \lambda)$ satisfies:
\begin{enumerate}
 \item $\lambda=\hat\lambda$;
 \item  $\Tr A=\Tr \hat A$
 \item $A-\hat A$ is a symmetric tensor with image in the distribution annihilating both $\phi$ and $\psi$;
 \item $(A-\hat A)^2=0$.
\end{enumerate}
\end{proposition}
\begin{proof}
Given an orthonormal frame $e_1,\dotsc, e_n$, so that the metric takes the form \eqref{eqn:generaldiagonalmetric}, we compute
\[\sum_j \epsilon_j e_j\cdot\nabla_{e_j}\psi = \frac12\epsilon_je_j\cdot A(e_j)\cdot\psi + \lambda \epsilon_j e_j\cdot e_j\cdot\phi = -\frac12\Tr(A)\psi -n \lambda\phi\]
and similarly
\[\sum_j \epsilon_j e_j\cdot\nabla_{e_j}\phi = \frac12 \Tr(A)\phi -n \lambda\psi.\]
Since $\phi$ and $\psi$ are linearly independent and the left-hand sides do not depend on $A$ and $\lambda$, we see that $\Tr A$ and $\lambda$ are uniquely determined.

As $\lambda$ is uniquely determined, it follows that the pair $(A(X)\cdot\psi,A(X)\cdot\phi)$ is uniquely determined for all $X$, and the third item follows.

The last item uses the fact observed in Remark~\ref{remark:Anotunique} that the distribution annihilating $\phi$ and $\psi$ is isotropic. A self-adjoint matrix $B$ whose image is isotropic necessarily satisfies $B^2=0$ because $\ker B=(\im B)^\perp\supset\im B$; notice that this implies $\Tr B=0$, consistently with the second item.
\end{proof}

\begin{remark}
In definite signature, a symmetric nilpotent matrix is zero, hence $A$ is uniquely determined by $\phi$ and $\psi$.
\end{remark}

In \cite{conti_segnan2024}, the following was proved:
\begin{theorem}[{\cite[Theorem 5.4]{conti_segnan2024}}]
\label{thm:harmfulExtension}
Assume $(M,g)$ is a real analytic pseudo-Riemannian spin manifold of signature $(p,q)$ with a real analytic real harmful structure $(\phi,\psi,A,\lambda)$. Then $(M,g)$  embeds isometrically in a pseudo-Riemannian Einstein spin manifold $(Z,h)$ with signature $(p+1,q)$ in such a way that $A$ is the Weingarten operator and either $\psi$  or $\psi-\nu\cdot_Z\phi$ (according to whether the dimension of $M$ is even or odd) extends to a Killing spinor $\Psi$ on $Z$ with Killing constant $\lambda$.
\end{theorem}
In \cite{conti_segnan2024} we also considered \textit{imaginary harmful structures} $(\phi,\psi,A,\lambda)$, where the pair of spinors satisfies
\[
\begin{cases}
	\nabla_X\psi=-\frac i2A(X)\cdot\psi-i\lambda X\cdot\phi\\
	\nabla_X\phi=i\lambda X\cdot\psi+\frac i2 A(X)\cdot\phi\\
	d\Tr A+\delta A=0
 \end{cases}
\]
In this case there is a similar embedding result as Theorem~\ref{thm:harmfulExtension}, where the metric on $Z$ has signature $(p,q+1)$. However, in this paper we will only employ real harmful structures, motivated by the following:
\begin{proposition}\label{prop:RealImmaginaryCorrespondence}
Let $(M,g)$ be a pseudo-Riemannian spin manifold of signature $(p,q)$. The following are equivalent:
\begin{itemize}
    \item $(\phi,\psi,A,\lambda)$ is a real harmful structure on $(M,g)$;
    \item $(\hat\phi,\hat\psi,\hat A,\hat\lambda)$ is an imaginary harmful structure on $(M,-g)$, where \[\hat\phi=(-1)^qi\phi\qquad\hat{\psi}=\psi\qquad\hat A=(-1)^qA\qquad\hat{\lambda}=-i\lambda.\]
\end{itemize}
\end{proposition}
\begin{proof}
We identify the spinor bundles of $(M,g)$ and $(M,-g)$ as in Remark~\ref{remark:changesign}.

If we set $(\hat\phi,\hat\psi)=((-1)^qi\phi,\psi)$, $\hat A=(-1)^qA$, $\hat \lambda = -i\lambda$, we compute
\[
\begin{split}
	\hat\nabla_X\hat\psi &= \nabla_X\psi =\frac12 A(X)\cdot\psi +\lambda X\cdot\phi=-\frac i2  (-1)^{q}A(X)\hat\cdot\psi-(-1)^{q}i\lambda X\hat\cdot\phi\\
	&=-\frac i2 \hat A(X)\hat\cdot\hat\psi -i\hat\lambda X\hat\cdot\hat\phi
\end{split}
\]
and
\[
\begin{split}
	\hat\nabla_X\hat \phi &= (-1)^qi\nabla_X\phi =(-1)^qi\lambda X\cdot\psi-\frac12 A(X)\cdot\hat\phi =\lambda X\hat\cdot\hat\psi+\frac i2 \hat A(X)\hat\cdot\hat\phi\\
	&=i\hat\lambda X\hat\cdot\hat\psi+\frac i2 \hat A(X)\hat\cdot\hat\phi.
\end{split}\]
Since by linearity $d\Tr A +\delta A=0$ implies that $d\Tr\hat A+\delta \hat A=0$, we obtain an imaginary  harmful structure.

By reversing the preceding construction, one can obtain a real harmful structure from an imaginary one.
\end{proof}

A special class of harmful structures is the one where $A$ is a multiple of the identity, which arise from pairs of Killing spinors with opposite Killing constants, as shown in the following:
\begin{proposition}\label{prop:AmultipleIdentity}
Let $(\phi,\psi,A,\lambda)$ be a real harmful structure on $(M,g)$ and suppose that $A=a\id$. Then $a$ is locally constant; if it is constant and $w$ satisfies $w^2=a^2/4+\lambda^2$, then
\[\eta=\lambda\psi + (w-\frac a2)\phi, \quad \xi=-\lambda\psi + (w+\frac a2)\phi\]
are Killing spinors with Killing constants respectively $w$ and $-w$; if $M$ has even dimension, $\omega\cdot\eta$ and $\xi$ are linearly dependent.

Conversely, let $\eta,\xi$ be Killing spinors on $(M,g)$ with Killing constants respectively $w$ and $-w$ and let $a\in\R$; if $M$ has even dimension, suppose in addition that $\xi=h\omega\cdot\eta$ for some constant $h\in\C^*$ such that
\[ h^2(a-2w)=(-1)^{q+n(n+1)/2}(a+2w).\]
Then
$(M,g)$ has a real harmful structure given by
\[
 \psi = \frac a2(\eta+\xi)+w(\eta-\xi), \quad \phi=\lambda(\eta+\xi), \quad A= a\id,\]
where $\lambda$ is one of the two solutions of $\lambda^2=w^2-a^2/4$; in even dimensions, $\lambda$ must be chosen so that $\phi=i^{\frac{2-3q-p}2}\omega\cdot\psi$.
\end{proposition}
\begin{proof}
Assume that $(\phi,\psi,A,\lambda)$ is a real harmful structure such that $A=a\id$. Then $d\Tr A + \delta A=0$ implies that $a$ is locally constant.

Suppose that $a$ is constant. We will obtain two Killing spinors by decoupling the harmful system via a diagonalization process. Indeed, we can write the harmful equations in the form
\begin{equation}
\label{eqn:harmfulmatrix}
\begin{pmatrix}\nabla_X\psi\\ \nabla_X\phi \end{pmatrix}=K_{a,\lambda}\begin{pmatrix}X	\cdot\psi\\ X\cdot \phi \end{pmatrix},
\quad K_{a,\lambda}=\begin{pmatrix}a/2 & \lambda \\ \lambda & - a/2\end{pmatrix}.
\end{equation}
For $\lambda w\neq0$, the matrix $K_{a,\lambda}$ is diagonalizable;  eliminating denominators, we obtain the identity
\[BK_{a,\lambda}=\begin{pmatrix}w &0 \\0 & -w\end{pmatrix}B,\quad B=\begin{pmatrix}\lambda & w-a/2\\- \lambda & w+a/2\end{pmatrix},\]
which holds generally. Define spinors $\eta,\xi$ so that
\[\begin{pmatrix} \eta \\ \xi \end{pmatrix}=B \begin{pmatrix} \psi \\ \phi \end{pmatrix}=\begin{pmatrix} \lambda \psi +(w-\frac a2)\phi \\
-\lambda \psi +(w+\frac a2)\phi \end{pmatrix}.
\]
Then\[
\begin{split}
\begin{pmatrix} \nabla_X\eta \\ \nabla_X\xi \end{pmatrix} &=  B\begin{pmatrix} \nabla_X\psi \\ \nabla_X\phi \end{pmatrix}
= BK_{a,\lambda}\begin{pmatrix}X\cdot \psi \\ X\cdot\phi \end{pmatrix}
=\begin{pmatrix}w &0 \\0 & -w \end{pmatrix} B\begin{pmatrix}X\cdot \psi \\ X\cdot\phi \end{pmatrix}\\
&=\begin{pmatrix}w &0 \\0 & -w \end{pmatrix} \begin{pmatrix}X\cdot \eta \\ X\cdot\xi \end{pmatrix}
=\begin{pmatrix}w X\cdot\eta \\ -w X\cdot\xi \end{pmatrix}.
\end{split}\]
In even dimensions, the spinors $\phi,\psi$ must satisfy the compatibility relation $\phi=i^{\frac{2-3q-p}2}\omega\cdot\psi$, so one ends up with
\[\eta = \lambda\psi + (w-\frac a2)i^{\frac{2-3q-p}2}\omega\cdot\psi, \quad \xi= -\lambda\psi + (w+\frac a2)i^{\frac{2-3q-p}2}\omega\cdot\psi,\]
implying that $\xi$ is linearly dependent with
\[\omega\cdot\eta= \lambda\omega\cdot\psi + (w-\frac a2)i^{\frac{2-3q-p}2}(-1)^{q+n(n+1)/2}\psi,\]
as
\[\det\begin{pmatrix}
-\lambda &   (w+\frac a2)i^{\frac{2-3q-p}2}\\
(w-\frac a2)i^{\frac{2-3q-p}2}(-1)^{q+n(n+1)/2} & \lambda
      \end{pmatrix}=0.\]

Conversely, suppose that $\eta$ and $\xi$ are Killing spinors with Killing constants respectively $w$ and $-w$. For any $a,\lambda$ such that $w^2=a^2/4+\lambda^2$, consider the  identity
\[K_{a,\lambda}B^{adj}=B^{adj}\begin{pmatrix}w &0 \\0 & -w\end{pmatrix},\quad B^{adj}=\begin{pmatrix} w+a/2 & -w+a/2\\ \lambda & \lambda\end{pmatrix},\]
and define
\[\begin{pmatrix} \psi \\ \phi \end{pmatrix}=B^{adj} \begin{pmatrix} \eta \\ \xi \end{pmatrix}
=\begin{pmatrix}\frac a2(\eta+\xi)+w(\eta-\xi)\\ \lambda(\eta+\xi)\end{pmatrix}
.\]
Then
\[\begin{split}
	\begin{pmatrix}\nabla_X\psi \\ \nabla_X\phi \end{pmatrix} &=B^{adj}\begin{pmatrix}\nabla_X\eta \\ \nabla_X\xi\end{pmatrix}= B^{adj}\begin{pmatrix}w X\cdot \eta \\ -w X\cdot \xi\end{pmatrix}
	=K_{a,\lambda}B^{adj}\begin{pmatrix}X\cdot \eta \\ X\cdot \xi\end{pmatrix}\\
	&=K_{a,\lambda}\begin{pmatrix}X\cdot \psi \\ X\cdot \phi\end{pmatrix},
\end{split}\]
and \eqref{eqn:harmfulmatrix} is satisfied.

In even dimensions, we obtain a harmful structure provided $\phi$ equals
\begin{multline*}
i^{\frac{2-3q-p}2}\omega\cdot ( \frac a2(\eta+\xi)+w(\eta-\xi))
=i^{\frac{2-3q-p}2}(\frac a2+w)\omega\cdot\eta +i^{\frac{2-3q-p}2}(\frac a2-w)h\omega^2\cdot\eta\\
=i^{\frac{2-3q-p}2}(\frac a2+w)\omega\cdot\eta +i^{\frac{2-3q-p}2}(\frac a2-w)h(-1)^{q+n(n+1)/2}\eta,
\end{multline*}
i.e.
\begin{equation}
 \label{eqn:lambdamakesharmful}
 \lambda =i^{\frac{2-3q-p}2}(\frac a2+w)h^{-1}=i^{\frac{2-3q-p}2+2q+n(n+1)}(\frac a2-w)h,
\end{equation}
which gives
\[(\frac a2-w)h^2=(\frac a2+w)(-1)^{q+n(n+1)/2}.\]
This is equivalent to the condition appearing in the statement, assuming which one can define $\lambda$ by~\eqref{eqn:lambdamakesharmful} and obtain $\lambda^2=w^2-\frac {a^2}4$.
\end{proof}

We now describe some general properties of harmful structures, for the two cases of even and odd dimension.
\begin{proposition}\label{prop:harmfulpropertiesMeven}
Let $(M,g)$ be a pseudo-Riemannian manifold of even dimension and signature $(p,q)$. Then the following holds:
\begin{itemize}
	\item if $\psi$ is a generalized Killing spinor on $(M,g)$ such that $\nabla_X\psi=\frac12A(X)\cdot\psi$, $d\Tr A+\delta A=0$ and $\phi=i^{\frac{2-3q-p}{2}}\omega\cdot\psi$, then $(\phi,\psi,A,0)$ is a real harmful structure;
	\item if $(\phi,\psi,A,\lambda)$ is a real harmful structure on $(M,g)$ then $\phi$ is chiral if and only if $\psi$ is chiral; in this case they are both parallel and annihilated by the image of $\frac12A\pm \lambda i^{1+n^2/2}\id$, where the sign depends on the chirality;
	\item if $(\phi,\psi,A,\lambda)$ is a real harmful structure on $(M,g)$, then $(-\psi,\phi,-A,-\lambda)$  is a real harmful structure.
\end{itemize}
\end{proposition}
\begin{proof}
The first item follows from a direct computation, using that the volume element $\omega$ is parallel and  $v\cdot\omega=-\omega\cdot\ v$ for any vector $v$.

To obtain the second item, recall that $\phi$ is a complex multiple of $\omega\cdot\psi$ and both covariant derivative and multiplication by the volume element preserve chirality, whilst multiplication by a vector reverses it. Now, if for instance $\psi$ has positive chirality, then so does $\phi=i^{\frac{2-3q-p}2}\omega\cdot\psi$. Then
\[\nabla_X\psi = \frac12A(X)\cdot\psi+\lambda X\cdot\phi\]
must have both positive and negative chirality, so it vanishes. The vanishing of the left-hand side implies that $\psi$ is parallel, and the vanishing of the right-hand side that
\[0=(\frac12A(X)\pm\lambda X i^{\frac{2-3q-p}2+q+n(n+1)/2})\cdot\psi = (\frac12A(X)\pm\lambda i^{1+n^2/2} X)\cdot\psi.\]

Finally, from $\omega^2\cdot\psi = (-1)^{q+n(n+1)/2}\psi=-(-1)^{\frac{2-3q-p}2}\psi$ we get the last item (see \eqref{eqn:evenvolumeacts}).
\end{proof}

\begin{remark}
\label{remark:invertedtimeingeneralizedcylinder}
We can illustrate the geometric meaning of the third item in Proposition~\ref{prop:harmfulpropertiesMeven} as follows. Let $Z$ have a Killing spinor with Killing number $w$, and let $M$ be a hypersurface with spacelike unit normal $\nu$. Inversion in each fibre of the normal bundle yields a diffeomorphism $r$ from a tubular neighbourhood of $M$ to itself; if the metric on $Z$ takes the form of a generalized cylinder $g_t+dt^2$, then $r^*(g_t+dt^2)=g_{-t}+dt^2$. The orientation-reversing diffeomorphism $r$ induces a Killing spinor $r_*\Psi$ with Killing number $-w$, which induces a harmful structure via \eqref{eqn:inducedCliffordMult}. To determine it explictly, we identify the spinor bundles as in Remark~\ref{remark:reverseorientation} and write
\[r_*(X)\widehat{\cdot} r_*\Psi = r_*(X\cdot\Psi).\]
Notice that $dr_{x,0}$ is a reflection, which can be written as conjugation $\Ad(\nu)$ by the unit normal $\nu$ inside the Clifford algebra, so for $n$ even at $t=0$ we have $r_*\Psi=\nu\cdot\Psi=\phi$.  Accordingly, the induced spinor $\phi$ is replaced by $i^{\frac{2-3q-p}2}\omega\cdot\phi=i^{2-3q-p}\omega\cdot\omega\cdot\psi=-\psi$ (see the proof of Proposition~\ref{prop:harmfulpropertiesMeven}).
 Since $r$ reverses the unit normal, the Weingarten operator changes its sign. Thus,  the harmful structure induced by $r_*\Psi$ is $(-\psi,\phi,-A,-w)$.
\end{remark}

The three items of Proposition~\ref{prop:harmfulpropertiesMeven} highlight three differences between even and odd dimensions. First, in odd dimensions a generalized Killing spinor $\psi$ does not induce a harmful structure, because $\psi$ does not induce a canonical nowhere vanishing spinor $\phi$ such that $\nabla_X\phi=-\frac12A(X)\cdot\phi$. Secondly, there is no notion of chirality. Finally, the symmetry of the third item can be viewed as a composition of two distinct symmetries, whose geometric meaning is shown by the following:
\begin{proposition}\label{prop:harmfulpropertiesModd}
Let $(\phi,\psi, A, \lambda)$ be a real harmful structure on an odd-dimensional pseudo-Riemannian manifold $(M,g)$, embedding isometrically as a hypersurface in $(Z,h)$ with induced Killing spinor $\Psi$. Then the following holds:
\begin{enumerate}
	\item  $(-\phi,\psi,A,-\lambda)$ is a real harmful structure, and the corresponding Killing spinor on $(Z,h)$ is $\Psi^+-\Psi^-$;
	\item  $(\psi,\phi,-A,\lambda)$ is a real harmful structure, which coincides with the real harmful structure on $(M,g)$ induced by reversing the orientation of $Z$.
\end{enumerate}
\end{proposition}
\begin{proof}
Denote by $\nu$ the unit normal compatible with the orientations, and let $\Psi=\Psi^++\Psi^-$ be the decomposition according to chirality. Since Clifford multiplication by a tangent vector interchanges $\Sigma_+Z$ and $\Sigma_-Z$, the Killing spinor equation $\nabla^Z_X\Psi=\lambda X\cdot_Z\Psi$ splits into
\[\nabla^Z_X\Psi^+=\lambda X\cdot_Z\Psi^-,\qquad \nabla^Z_X\Psi^-=\lambda X\cdot_Z\Psi^+,\]
whence $\nabla^Z_X(\Psi^+-\Psi^-)=-\lambda X\cdot_Z(\Psi^+-\Psi^-)$, i.e. $\Psi^+-\Psi^-$ is a Killing spinor with Killing number $-\lambda$. By the description accompanying~\eqref{eqn:inducedCliffordMult}, the spinors $\psi,\phi$ of the harmful structure induced by $\Psi$ are the restrictions of $\Psi^+$ and $\nu\cdot_Z\Psi^-$; replacing $\Psi$ with $\Psi^+-\Psi^-$ leaves the first unchanged and changes the sign of the second, whilst the Weingarten operator is not affected. Therefore, the induced real harmful structure is $(-\phi,\psi,A,-\lambda)$.

The second part follows from the fact that changing orientation of $Z$ (whilst keeping the orientation of $M$ fixed) amounts to changing the sign of $\nu$,  hence of $A$, and interchanging $\Sigma_+Z$ and $\Sigma_-Z$.
\end{proof}

We conclude the section by discussing the case where the spinors $\phi,\psi$ are linearly dependent. This is only of interest in odd dimensions, because in even dimension the definition of harmful structure requires $\phi$ to be a multiple of $\omega\cdot\psi$, so the spinors $\phi,\psi$ can only be linearly dependent if they are chiral, forcing them to be parallel (see Proposition~\ref{prop:harmfulpropertiesMeven}).

\begin{proposition}\label{prop:PhiPsiLinDip}
Let $(M,g)$ be an odd-dimensional manifold; let $(\phi,\psi,A,\lambda)$ be a real harmful structure with $A$ diagonalizable over $\C$ at each point and $\phi,\psi$ linearly dependent, also over $\C$. Then $A=a\id$, $\psi$ is a Killing spinor with Killing constant $w$, $\phi=k\psi$,
\begin{equation}
 \label{eqn:klambdamu}
a=\lambda\frac{1-k^2}k, \quad w=\lambda\frac{k^2+1}{2k}, \text{ and } \begin{cases}k\in\R\bigcup i\R\setminus\{0\} &\lambda\in\R\\ k\in S^1& \lambda\in i\R \end{cases}.
 \end{equation}
Conversely, if $\psi$ is a Killing spinor with Killing constant $w$ and $a,k,\lambda$ satisfy \eqref{eqn:klambdamu}, there is a real harmful structure $(\phi,\psi,A,\lambda)$ with $A=a\id$ and $\phi=k\psi$.
\end{proposition}
\begin{proof}
Recall that $\phi$ and $\psi$ are nowhere zero by definition. Suppose $\phi=k\psi$ with $k\in\C^*$; then  $k\nabla_X\psi-\nabla_X \phi$ is zero. If $(\phi,\psi,A,\lambda)$ is a real harmful structure with $A$ diagonalizable at each point,
the harmful  equations \eqref{eqn:systemSpace} give
\[
\begin{split}0&=k\bigl(\frac12 A(X)\cdot\psi + \lambda k X\cdot\psi\bigr)-\lambda X\cdot \psi +\frac12A(X)\cdot k\psi\\
	&=(kA+\lambda (k^2-1)\id)(X)\cdot\psi.
\end{split}
\]
Thus, $A$ differs from a multiple of the identity by a matrix taking values in the distribution that annihilates $\psi$. Arguing as in Proposition~\ref{prop:Auniqueupto}, we see that this difference is nilpotent, and so it must vanish because $A$ is diagonalizable over $\C$; in sum,
\[A=\frac{\lambda(1-k^2)}k\id.\]
This implies that $1/k-k$ is either real or imaginary, i.e.
\[\R\ni (1/k-k)^2=\frac1{k^2}+k^2-2;\]
writing $k^2=\rho e^{i\theta}$ we see that either $\sin\theta=0$ or $\rho=1$. Thus, $k$ is real, imaginary, or $\abs{k}=1$.

Moreover, the harmful equations give
\[\nabla_X\psi =\lambda\left(\frac{1-k^2}{2k}+k\right)X\cdot\psi
=\lambda\frac{k^2+1}{2k}X\cdot\psi.\]

Vice versa, if $\psi$ is a Killing spinor with Killing constant $w$, $a,k,\lambda$ satisfy \eqref{eqn:klambdamu} and $\phi=k\psi$, then
\[
\begin{split}
	\nabla_{X}\psi&=\lambda\frac{k^2+1}{2k}X\cdot\psi  = \lambda\frac{1-k^2}{2k}X\cdot\psi + k\lambda X\cdot\psi=\frac12A(X)\cdot\psi+\lambda X\cdot\phi,\\
	\nabla_{X}\phi&=\lambda\frac{k^2+1}{2}X\cdot\psi  = \lambda\frac{k^2-1}{2}X\cdot\psi + \lambda X\cdot\psi = -\frac12 A(X)\cdot \phi +\lambda X\cdot\psi,
\end{split}
\]
and $d\Tr A+\delta A$ vanishes trivially. Notice that $a$ is real by \eqref{eqn:klambdamu}.
\end{proof}

\begin{remark}
In the situation of Proposition~\ref{prop:PhiPsiLinDip}, one of the Killing spinors $\eta,\xi$ of Proposition~\ref{prop:AmultipleIdentity} is identically zero. Indeed, the constant $w$ of Proposition~\ref{prop:AmultipleIdentity} satisfies
\[w^2=\lambda^2\left(1+\frac{(1-k^2)^2}{4k^2}\right)=\frac{\lambda^2(1+k^2)^2}{4k^2}\]
i.e. $w=\pm \frac{\lambda(1+k^2)}{2k}$. Then
\begin{gather*}
	\eta = (\lambda + (w-\frac a2)k)\psi = \left(\lambda \pm \frac{\lambda(1+k^2)}{2} - \frac{\lambda(1-k^2)}2\right)\psi,\\
	\xi = (-\lambda + (w+\frac a2)k)\psi = \left(-\lambda \pm \frac{\lambda(1+k^2)}{2} + \frac{\lambda(1-k^2)}2\right)\psi.
\end{gather*}
so one of them is zero, depending on the choice of the sign for $w$.
\end{remark}

\begin{remark}
In Proposition~\ref{prop:PhiPsiLinDip}, the diagonalizability of $A$ implies that $A$ is a multiple of the identity; this is only true under the assumption that $\phi,\psi$ are linearly dependent, and does not hold generally, as made clear from the examples in the following sections.
\end{remark}

\section{Harmful structures on metric Lie algebras}\label{section:harmfulOnMetricLie}
In this section we study left-invariant harmful structures on a Lie group with Lie algebra $\g$ endowed with a metric $g$; since we only consider left-invariant structures, from now a spinor will simply mean an element of the spinor representation. In preparation for the classification results of the following  sections, we introduce the operators
\begin{equation}
\label{eqn:Li_as_operator}
\psi\mapsto -\epsilon_j e_j\cdot \nabla_{e_j}\phi
\end{equation} associated to an orthonormal basis $\{e_j\}$, whose sum gives (up to a sign) the well-known Dirac operator of spin geometry; we characterize these operators in the presence of a harmful structure, and show that they determine the structure constants. These operators give an effective way of rewriting the harmful equations when the basis elements $\{e_j\}$ are eigenvectors of $A$. For this reason, we will make the technical assumption  that $A$ is diagonalizable over $\R$; for brevity, this condition will be expressed in the sequel by writing that $A$ is diagonalizable, without specifying the field. The condition is not restrictive in definite signature, as $A$ is symmetric.

First, we tackle the third equation of  \eqref{eqn:systemSpace}, which in the invariant setting simply reads $\delta A=0$.
\begin{lemma}\label{lemma:diagonlAdelta}
Let $g$ be a metric on $\g$; let $A$ be a symmetric diagonalizable endomorphism. Then $\delta A=0$ if and only if
\[\Tr(A\circ \ad v)-\Tr\ad Av=0, \quad v\in\g.\]
\end{lemma}
\begin{proof}
Since $A$ is diagonalizable, there is an orthonormal basis of eigenvectors, so we can write
\[A=\sum a_je^j\otimes e_j, \quad g=\epsilon_1e^1\otimes e^1+\dotsc + \epsilon_n e^n\otimes e^n.\]
Thus, since $\nabla_{e_j}e_j=-(\ad e_j)^*(e_j)$,
\[
\begin{split}
\delta A =& \sum_{j} \epsilon_j(\nabla_{e_j}A)(e_j)=\sum_{j}\epsilon_j \big(\nabla_{e_j}(A(e_j))-A(\nabla_{e_j}e_j)\big)\\
=&\sum_j\epsilon_j \big(a_j\nabla_{e_j}e_j-A(\nabla_{e_j}e_j)\big)=\sum_{j}\epsilon_j \bigl(A((\ad e_j)^*(e_j)) -(\ad e_j)^*(A(e_j))\bigr)\\
=&\sum_{j}\epsilon_j [A,(\ad e_j)^*](e_j)=-\sum_{j}\epsilon_j [A,\ad e_j]^*(e_j)=-\sum_j e^j([A,\ad e_j]).
\end{split}
\]
Consequently,
\[
\begin{split}
0=&v\mapsto \sum_j e^j\bigl(A([e_j,v])-[e_j,Av]\bigr)= \sum_j e^j\bigl(-A\ad(v)(e_j)+[Av,e_j]\bigr)\\
=&\Tr(A\circ \ad v)-\Tr\ad Av.\qedhere
\end{split}
\]
\end{proof}

\begin{remark}
 If $\g$ is unimodular, $\delta A=0$ is equivalent to $A$ being orthogonal to $\ad\g\subset\g^*\otimes\g$ relative to the scalar product $\langle f,g\rangle_{\Tr}=\Tr(fg)$. In particular, if $\g$ is nilpotent and $A$ preserves each term of the lower central series, then $\delta A=0$.
\end{remark}

\begin{remark}
For $\su(2)$, the condition of Lemma~\ref{lemma:diagonlAdelta} implies that $A$ is symmetric both relative to $g$ and the ad-invariant metric $\langle,\rangle_{\Tr}$ on $\su(2)$. Let $e_1,e_2,e_3$ be a basis of $\su(2)$ satisfying $[e_1,e_2]=e_3$ and its cyclic permutations, and let $f\colon\su(2)\to\su(2)$ be a diagonalizable endomorphism with eigenvectors $e_1,e_2,e_3$, each of multiplicity one. If $g=\langle f\cdot,\cdot\rangle_{\Tr}$, then $A$ is symmetric relative to $g$ if and only if it commutes with $f$. Therefore, it preserves the three eigenvectors, i.e.
\[A=a_1e^1\otimes e_1+a_2e^2\otimes e_2+a_3e^3\otimes e_3.\]
We will see in Section~\ref{sec:ClassificationDim3} that $A$ is a multiple of the identity for every harmful structure on $\su(2)$.
\end{remark}

Fix a Lie algebra $\g$ with a basis $\{e_j\}$  and a metric $g$ of the form~\eqref{eqn:generaldiagonalmetric}; let $\{c_{jkh}\}$ be the structure constants, satisfying
\[[e_j,e_k]=\sum_h c_{jkh}e_h.\]
To describe the operators~\eqref{eqn:Li_as_operator}, observe that the Koszul formula gives
\begin{equation}
 \label{eqn:omegasharp}
\theta^\sharp= \frac12\sum_{j,k,h} \epsilon_k(c_{jkh}-\epsilon_k\epsilon_hc_{jhk}-\epsilon_j\epsilon_hc_{khj})e^j\otimes e_k\otimes e_h.
\end{equation}
Then~\eqref{eqn:ConnectionOnSpinor} shows that for fixed $j$ the operator~\eqref{eqn:Li_as_operator} is given by Clifford multiplication by the element
\begin{equation}
\label{eqn:L_i_in_terms_of_cijk}
L_j=-\frac18 \sum_{k,h}\bigl(\epsilon_j\epsilon_kc_{jkh}-\epsilon_h\epsilon_jc_{jhk}-\epsilon_k\epsilon_hc_{khj}\bigr) e_j\cdot e_k\cdot e_h.
\end{equation}
We can now characterize harmful structures in terms of the operators $L_j$. We will first consider the case $\lambda\neq0$.
\begin{lemma}
\label{lemma:characterizediagonalharmful}
On a Lie algebra $\g$, consider a metric of signature $(p,q)$ of the form \eqref{eqn:generaldiagonalmetric}, let $A=\sum a_je^j\otimes e_j$ and $\lambda\neq0$, and suppose that $\Tr (A\circ \ad v)=\Tr (\ad Av)$ for all $v\in\g$. If the dimension $p+q$ is odd,  $(\phi,\psi,A,\lambda)$ is a real  harmful structure if and only if the following hold:
\begin{gather}
\label{eqn:LiminusLj}
(L_j-L_k)\psi=\frac12(a_j-a_k)\psi,\\
\label{eqn:Lisquared}
L_j^2\psi = \bigl(\frac14a_j^2+\lambda^2\bigr)\psi,\\
\label{eqn:lambdaphi}
\lambda\phi =  L_j\psi -\frac12a_j\psi \text{  for some (hence all) }j.
\end{gather}
If the dimension $n=p+q$ is even, $(\phi,\psi,A,\lambda)$  is a real  harmful structure if and only if \eqref{eqn:LiminusLj}, \eqref{eqn:lambdaphi} hold and
\[\phi=i^{\frac{2+n^2}2}(\psi_+-\psi_-).\]
\end{lemma}
\begin{proof}
We first prove that \eqref{eqn:LiminusLj}, \eqref{eqn:Lisquared}, \eqref{eqn:lambdaphi} are necessary. Setting $X=e_j$ in the first equation of \eqref{eqn:systemSpace}  and taking the Clifford multiplication  by $-\epsilon_j e_j$ gives
\[-\epsilon_j e_j\nabla_{e_j}\psi = -\frac12\epsilon_j e_j\cdot A(e_j)\cdot \psi - \lambda \epsilon_j e_j\cdot e_j\cdot \phi\]
so
\[L_j\cdot\psi = \frac12 a_j \psi + \lambda \phi.\]
This implies that $\lambda\phi=L_j\cdot\psi - \frac12 a_j \psi$ for all $i$. Additionally, the second equation of~\eqref{eqn:systemSpace} multiplied by $-\epsilon_j  e_j$ gives
\[-\epsilon_j e_j\cdot \nabla_{e_j} \phi = -\lambda \epsilon_j e_j\cdot e_j\cdot\psi+\frac12  \epsilon_j e_j\cdot a_je_j\cdot\phi\]
i.e.
\begin{equation}
\label{eqn:Liphi}
(L_j+ \frac12a_j)\cdot \phi=\lambda\psi,
\end{equation}
implying \eqref{eqn:Lisquared}.
Conversely, suppose the dimension $p+q$ is odd and \eqref{eqn:LiminusLj}, \eqref{eqn:Lisquared}, \eqref{eqn:lambdaphi} hold. Since we assume $\lambda\neq0$,  one can multiply by $e_j$ in the above relations and recover~\eqref{eqn:systemSpace} for $X=e_j$.

If $p+q$ is even, we use the fact that
\[L_j\cdot \omega \cdot\psi = -\epsilon_j e_j\cdot\nabla_{e_j}(\omega\cdot\psi)= -\epsilon_j e_j\cdot\omega\cdot \nabla_{e_j}\psi = \epsilon_j \omega\cdot e_j\nabla_{e_j}\psi = -\omega\cdot L_j\cdot\psi.\]
Thus if \eqref{eqn:LiminusLj}, \eqref{eqn:lambdaphi} and $\phi=i^{\frac{2-3q-p}2}\omega\cdot\psi$ hold,
\[
\begin{split}
L_j\cdot\phi &=- i^{\frac{2-3q-p}2}\omega\cdot L_j\cdot \psi=  -i^{\frac{2-3q-p}2}\omega\cdot (\lambda\phi + \frac12a_j\psi)
=-i^{\frac{2-3q-p}2}\omega\cdot \lambda\phi - \frac12a_j\phi\\
&=-i^{2-3q-p}\lambda \omega\cdot\omega\cdot\psi- \frac12a_j\phi
\end{split}
\]
which implies~\eqref{eqn:Liphi} (this argument also holds if $\lambda=0$, but see Lemma~\ref{lemma:characterizediagonalgeneralizedKillingspinors} below).

The condition $d\Tr A+\delta A=0$ is satisfied by Lemma~\ref{lemma:diagonlAdelta}.
\end{proof}
In the case $\lambda=0$, the harmful system reduces to a pair of generalized Killing spinor equations. Adapting the proof of Lemma~\ref{lemma:characterizediagonalharmful}, we can use our language to characterize generalized Killing spinors.
\begin{lemma}
\label{lemma:characterizediagonalgeneralizedKillingspinors}
On a Lie algebra $\g$, consider a metric of signature $(p,q)$ of the form \eqref{eqn:generaldiagonalmetric}, and let $A=\sum a_je^j\otimes e_j$.
A spinor $\psi$ satisfies the generalized Killing equation
\[\nabla_X\psi=\frac12A(X)\cdot\psi\]
if and only
\[L_j\psi =\frac12a_j\psi, \quad j=1,\dotsc, n.\]
\end{lemma}
\begin{remark}\label{rmk:GeneralizedKillingandRicciflat}
The role of the assumption $\Tr (A\circ \ad v)=\Tr (\ad Av)$ (equivalently, $d\Tr A+\delta A=0$) is to ensure the existence of an Einstein extension. The assumption is part of the definition of harmful structure, but not of the definition of generalized Killing spinor; for this reason, it appears in Lemma~\ref{lemma:characterizediagonalharmful}, but not in  Lemma~\ref{lemma:characterizediagonalgeneralizedKillingspinors}. Examples of generalized Killing spinors for which a Ricci-flat extension does not exist can be obtained by considering hypersuraces in a manifold with a parallel spinor which is not Ricci-flat; see for instance the Lorentzian symmetric spaces studied in \cite{Baum:TwistorSpinorsSymmetric}.
\end{remark}

\begin{lemma}
\label{lemma:unimod}
Let $\g$ be a unimodular Lie algebra of dimension $n$ with a metric and a real harmful structure $(\phi,\psi,A,\lambda)$ with $A$ diagonalizable.  Then
\begin{gather*}
L_j=M+\mu_j+\xi_j,
\end{gather*}
where:
\begin{enumerate}
\item $\mu_j$ has degree one, and $M,\xi_j$ have degree three;
\item $g(\mu_j,e_j)=0$;
\item $\sum \mu_j=0=\sum a_j\mu_j=\sum \xi_j$;
\item $(\mu_j+\xi_j)\psi=\frac12(a_j-\frac{1}{n}\Tr A)\psi$.
\end{enumerate}
\end{lemma}
\begin{proof}
Since $A$ is symmetric, it has an orthonormal basis $e_1,\dotsc, e_n$, $n=p+q$ of eigenvectors. Let $M$ be the average of the $L_j$, i.e. $M=\frac1n(\sum_j L_j)$, and write
\[L_j=M+\mu_j+\xi_j,\]
where the $\mu_j, \xi_j$ have degrees respectively one and three. Recalling that
\[(e_j\hook de^j)(e_k)=de^j(e_j,e_k)=-c_{jkj},\]
i.e. $e_j\hook de^j=-c_{jkj}e^k$, we see that
\[\begin{split}
L_j&=-\frac18 \sum_{k\neq j,h\neq j}\bigl(\epsilon_j\epsilon_kc_{jkh}+
\epsilon_h\epsilon_jc_{hjk}-\epsilon_k\epsilon_hc_{khj}\bigr) e_j\cdot e_k\cdot e_h\\
&\quad-\frac14 \sum_{k=j,h\neq j}\bigl(\epsilon_j\epsilon_kc_{jkh}+
\epsilon_h\epsilon_jc_{hjk}-\epsilon_k\epsilon_hc_{khj}\bigr) e_j\cdot e_k\cdot e_h\\
&=-\frac18 \sum_{k\neq j,h\neq j}\bigl(\epsilon_j\epsilon_kc_{jkh}+
\epsilon_h\epsilon_jc_{hjk}-\epsilon_k\epsilon_hc_{khj}\bigr)e_j\cdot e_k\cdot e_h\\
&\quad+\frac12 \sum_{h\neq j}\bigl(\epsilon_j\epsilon_hc_{jhj}\bigr) e_j\cdot e_j\cdot e_h.
\end{split}\]
Hence
\begin{equation}
 \label{eqn:Lifromcijk}
L_j=-\frac1{4}\sum_{\stackrel{k<h}{k,h\neq j}} \bigl(\epsilon_j\epsilon_kc_{jkh}-\epsilon_h\epsilon_jc_{jhk}-\epsilon_k\epsilon_hc_{khj}\bigr) e_j\wedge e_k\wedge e_h+\frac1{2}(e_j\hook de^j)^\sharp.
\end{equation}
Taking the average over $j$, one sees that $M$ has pure degree three, as $\sum e_j\hook de^j$ is zero. By Lemma~\ref{lemma:diagonlAdelta}, since we assume $\g$ unimodular, $\delta A=0$ amounts to
\[0=\Tr (A\ad e_k)=\sum a_jc_{kjj} = \sum a_j de^J(e_j,e_k) =\left(\sum a_j e_j\hook de^j\right)(e_k),\]
hence
$\delta A=0$ if and only if $\sum a_je_J\hook de^j=0$, i.e.
\[\sum a_j\mu_j=0.\]
By construction, we also have
\[\sum \mu_j=0=\sum \xi_j.\]
Since $\mu_j$ is the degree one component of $L_j$, it is orthogonal to $e_j$, as one can see from \eqref{eqn:Lifromcijk}.

By Lemma~\ref{lemma:characterizediagonalharmful} (or Lemma~\ref{lemma:characterizediagonalgeneralizedKillingspinors} for $\lambda=0$), we have
\begin{equation}
 \label{eqn:lambdaphiequalsunimod}
 \lambda\phi = L_j\psi - \frac12a_j\psi = (M + \mu_j+\xi_j)\cdot \psi - \frac12a_j\psi.
\end{equation}
Averaging~\eqref{eqn:lambdaphiequalsunimod} over the index $i$, one obtains
\begin{equation}
 \label{eqn:lambdaphiMpsi}
 \lambda\phi = M\psi -\frac1{2n}(\Tr A)\psi.
\end{equation}
Finally, subtracting \eqref{eqn:lambdaphiMpsi} from~\eqref{eqn:lambdaphiequalsunimod} gives
\[ 0=(\mu_j+\xi_j)\psi-\frac12(a_j-\frac1n\Tr A)\psi.\qedhere\]
\end{proof}

We can now see that in many cases the only harmful structures with $A$ diagonalizable on a unimodular Lie algebra are those determined by a Killing spinor as in Proposition~\ref{prop:AmultipleIdentity}.
\begin{lemma}
\label{lemma:gamma1nonzero}
Let $\g$ be a unimodular Lie algebra of dimension $n$ with a real harmful structure $(\phi,\psi,A,\lambda)$ with $A$ diagonalizable. Assume furthermore that $A$ is not a multiple of the identity.
Then there exists a real three-vector $\xi\in\Lambda^3\g$ such that $\xi\cdot\psi=\psi$ and $\xi\cdot\phi=-\phi$; if $\lambda\neq0$, for any choice of $\xi$, $M$ and $\xi$ are linearly independent.
\end{lemma}
\begin{proof}
Arguing as in Lemma~\ref{lemma:unimod}, we can get a relation between $\phi$ and $M,\mu_j$ and $\xi_j$, namely:
\begin{gather*}
M\cdot\phi = \lambda\psi -\frac{1}{2n}(\Tr A)\phi,\\
(\mu_j+\xi_j)\cdot\phi=-\frac12(a_j-\frac{1}{n}\Tr A)\phi.
\end{gather*}
Multiplying by $a_j$ the above equation and the analogous one for $\psi$ (see Lemma~\ref{lemma:unimod}) and summing over $j$ we obtain
\begin{gather*}
\sum a_j\xi_j\cdot\psi=\sum a_j(\mu_j+\xi_j)\cdot\psi=\frac12(\Tr A^2-\frac{1}{n}(\Tr A)^2)\psi,\\
\sum a_j\xi_j\cdot\phi=\sum a_j(\mu_j+\xi_j)\cdot\phi=-\frac12(\Tr A^2-\frac{1}{n}(\Tr A)^2)\phi.
\end{gather*}
Suppose first that  $n\Tr A^2-(\Tr A)^2$ is zero. Denoting by $[1]$ the vector in $\R^n$ with every coordinate equal to $1$ and setting $y=(a_1,\dotsc, a_n)$, the above computation gives \[\langle [1],[1]\rangle \langle y,y\rangle = \langle y,[1]\rangle^2;\]
by the Cauchy-Schwartz inequality, $y$ is a multiple of $[1]$, i.e. $A$ is a multiple of the identity, which is  absurd.

Setting
\[\xi=\frac{2n }{n\Tr A^2-(\Tr A)^2}\sum a_j\xi_j\]
proves the statement.

Now suppose for a contradiction that $M$ and $\xi$ are linearly dependent; then so are $M\cdot\psi$ and $\psi$. By \eqref{eqn:lambdaphiMpsi}, this implies that $\phi,\psi$ are linearly dependent. If $n$ is even, this implies that they are chiral, hence $A$ is a multiple of the identity by Proposition~\ref{prop:harmfulpropertiesMeven}. If $n$ is odd, $A$ is a multiple of the identity by Proposition~\ref{prop:PhiPsiLinDip}. Either way, we obtain a contradiction.
\end{proof}

It turns out that the operators $L_j$ determine the structure constants. It will be convenient to denote wedge products $e^j\wedge e^k$ with the abbreviated notation $e^{jk}$, and similarly in higher degree. For forms and multivectors, we also introduce the notation
\[\gamma=\prescript{}{1}\gamma+\dots + \prescript{}{p}\gamma,\]
where $\prescript{}{k}\gamma$ denotes the component of degree $k$. We find
\begin{proposition}
\label{prop:constantsfromLi}
Given $c=\{c_{jkh}e^{ij}\otimes e_k\}$, the elements
\begin{equation}\label{eqn:Lifromcijk2}
L_j=-\frac1{4}\sum_{\stackrel{k<h}{k,h\neq j}} \bigl(\epsilon_j\epsilon_kc_{jkh}-\epsilon_h\epsilon_jc_{jhk}-\epsilon_k\epsilon_hc_{khj}\bigr) e_j\wedge e_k\wedge e_h-\frac1{2}\sum_k \epsilon_kc_{jkj}e_k
\end{equation}
satisfy
\begin{equation}\label{eqn:L_i_satisfy}
e_j\hook \prescript{}{1}L_j=0, \quad e^j\wedge \prescript{}{3}L_J=0.
\end{equation}
Conversely, given $L_1,\dotsc, L_n\in\R^n\oplus\Lambda^3\R^n$ satisfying \eqref{eqn:L_i_satisfy} there is a unique
$c=\{c_{jkh}e^{jk}\otimes e_h\}$ inducing the $L_j$ via \eqref{eqn:Lifromcijk2}, given by
\[\begin{split}
c_{jkh}&=-2\epsilon_j\epsilon_ke^{jkh}\hook \prescript{}{3}(L_j+L_k), \quad j,k,h \text{ distinct;}\\
c_{jkj}&=-2\epsilon_k e^k\hook \prescript{}{1}L_j,\quad  j,k \text{ distinct}\\
\end{split}
\]
If the Jacobi identity is satisfied, the Lie algebra can be written in the form
\[
de^k=2\sum_{j} (e^k\hook  \prescript{}{3}L_j)^\flat-2(e^k\hook  \prescript{}{3}L_k)^\flat+2e^k\wedge \prescript{}{1}L_k^\flat\\
\]
and $L_j\cdot\psi=-\epsilon_j e_j\cdot\nabla_{e_j}\cdot\psi$ for all spinors on $\g$.
\end{proposition}
\begin{proof}
The operators $L_i$ satisfying \eqref{eqn:L_i_satisfy} and the scalars $c_{jkh}$ depend on the same number of parameters, as
\[n\left(\binom{n-1}{2}+n-1\right)=n\binom{n}2.\]
We need to give an inverse to the linear map $\{c_{jkh}\}\to \{L_j\}$.

Fixing distinct indices $j,k,h$, we have
\[\prescript{}{3}L_j(e^j,e^k,e^h)=-\frac14 \bigl(\epsilon_j\epsilon_kc_{jkh}+\epsilon_h\epsilon_jc_{hjk}-\epsilon_k\epsilon_hc_{khj}\bigr),
\]
so
\[
\begin{split}
\epsilon_j\epsilon_kc_{jkh}&=\frac12(\epsilon_j\epsilon_kc_{jkh}+\epsilon_h\epsilon_jc_{hjk}-\epsilon_k\epsilon_hc_{khj})+\frac12(\epsilon_k\epsilon_hc_{khj}+\epsilon_j\epsilon_kc_{jkh}-\epsilon_h\epsilon_jc_{hjk})\\
&= -2e^{jkh}\hook \prescript{}{3}(L_j+ L_k).
\end{split}
\]
The fact that the structure constants $c_{jkj}$ can be recovered from the scalars $e^k\hook \prescript{}{1}L_j$ is straightforward from \eqref{eqn:Lifromcijk2}.

To prove the second part, write
\[
\begin{split}
de^k&=-\sum_{\substack{h<j\\ h,j\neq k}} c_{hjk}e^{hj} -\sum_j c_{kjk}e^{kj}\\&=2\sum_{\substack{h<j\\ h,j\neq k}} \epsilon_h\epsilon_j\bigl(e^{hjk}\hook\prescript{}{3}( L_h+L_j)\bigr) e^{hj}+2\sum_{j}\epsilon_j\bigl(e^j\hook \prescript{}{1} L_k\bigr) e^{kj}\\
&=2\sum_{j\neq k} \epsilon_j\bigl(e^{jk}\hook \prescript{}{3}L_j)\bigr)^\flat \wedge e^j+2e^k\wedge \prescript{}{1}L_k^\flat\\
&=2\sum_{j\neq k} \bigl(e^k\hook ((e^{j}\hook \prescript{}{3}L_j)\wedge e_j)\bigr)^\flat+2e^k\wedge \prescript{}{1}L_k^\flat\\
&=2\sum_{j} (e^k\hook  \prescript{}{3}L_j)^\flat-2(e^k\hook  \prescript{}{3}L_k)^\flat+2e^k\wedge \prescript{}{1}L_k^\flat\\
\end{split}
\]
The fact that the operators $L_j$ acts as expected on spinors follows by comparing \eqref{eqn:Lifromcijk} and \eqref{eqn:Lifromcijk2}.
\end{proof}

In our setting, the average of the $L_j$ has pure degree three (see Lemma~\ref{lemma:unimod}), and we obtain the following:
\begin{corollary}
\label{cor:cijkfromLiwhenDiracthreeform}
Let $\g$ be a unimodular Lie algebra of dimension $n$ with a metric of the form \eqref{eqn:generaldiagonalmetric} such that $M=\frac1n\sum L_j$ has pure degree $3$; write
\[L_j=M+\xi_j +\mu_j,\]
where $\xi_j$ have degree $3$ and $\mu_j$ have degree $1$. Then
\[de^k=2\bigl((n-1)e^k\hook M-e^k\hook\xi_k\bigr)^\flat+2e^k\wedge\mu_k^\flat.\]
\end{corollary}
\begin{example}
\label{ex:consistencycheck}
Consider  the case where the metric is positive definite and
\[L_1=0, \quad L_2=xe_{234}, \quad L_3=ye_{234},\quad L_4=ze_{234}.\]
Then Proposition~\ref{prop:constantsfromLi} gives
\begin{equation}
 \label{eqn:consistencycheck}
 de^1=0, \quad de^2= 2(y+z) e^{34}, \quad de^3=2(z+x)e^{42}, \quad de^4=2(x+y)e^{23}.
\end{equation}
The same conclusion can be obtained applying Corollary~\ref{cor:cijkfromLiwhenDiracthreeform} with
 $M=\frac14(x+y+z)e_{234}$.

As a matter of notation, a Lie algebra with a fixed basis $e_1,\dotsc, e_4$ such that the dual basis $e^1,\dotsc, e^4$ satisfies~\eqref{eqn:consistencycheck} will be denoted by
\[\bigl(0,2(y+z) e^{34},2(z+x)e^{42}, 2(x+y)e^{23}\bigr).\]
\end{example}

\section{Harmful structures in dimension $3$}\label{sec:ClassificationDim3}
In this section, we show that real harmful structures  $(\phi,\psi, A,\lambda)$  with $A$ diagonalizable on unimodular $3$-dimensional Lie algebras are determined by a Killing spinor, and derive a classification. Notice that requiring the symmetric tensors $A$ to be diagonalizable is not restrictive in the definite setting.

In the following, $\g$ is a unimodular Lie algebra with a metric of the form \eqref{eqn:standard3metric}, so that  Clifford multiplication is given by \eqref{eqn:cliff3dim}. We have the following:
\begin{lemma}\label{lemma:unim3HarmAreKilling}
Let $\g$ be a $3$-dimensional unimodular Lie algebra and $(\phi,\psi,A,\lambda)$ a harmful structure with $A$ diagonalizable. Then either $\g$ has a parallel spinor, or $\psi$ is a  Killing spinor with Killing number $w\neq0$, $\phi=k\psi$, $A=a\id$, and
\begin{itemize}
    \item if $w\in\R^*$ then $k\in\R\cup i\R\setminus\{0,\pm i\}$ and
    \begin{equation}
     \label{eqn:alambdaforunim3harmful}
a=2w\frac{1-k^2}{1+k^2}\quad \text{and}\quad\lambda=2w\frac{k}{1+k^2}.    \end{equation}
    \item if $w\in i\R^*$ then $k\in S^1\setminus\{\pm i\}$, and $a,\lambda$ are given by \eqref{eqn:alambdaforunim3harmful}.
\end{itemize}
\end{lemma}
\begin{proof}
As a first step, we show that any  two (nonzero) Killing spinors on a fixed unimodular Lie algebra $\g$ of dimension $3$ have the same Killing number $w$. Indeed, by \eqref{eqn:ConnectionOnSpinor}, we have $\nabla_X\psi=\frac14\theta^\sharp(X)\cdot\psi$; since $\theta^\sharp(X)$ is a two-form, by \eqref{eqn:starproduct} we have
\[\nabla_X\psi = B(X)\cdot\psi, \quad B(X)=\frac14i^{p+1}{*}\theta^\sharp(X).\]
We claim that $B$ is symmetric, i.e. $g(B(X),Y)=g(X,B(Y))$. With no loss of generality, assume $X=e_1$, $Y=e_2$; then
\[\theta^\sharp(e_1)\wedge e_2=\theta_{13}(e_1)\epsilon_3 e_3\wedge e_1\wedge e_2 = c_{131}e_1\wedge e_2\wedge e_3,\]
where we have used \eqref{eqn:omegasharp}.
Similarly,
\[\theta^\sharp(e_2)\wedge e_1= c_{322}e_2\wedge e_3\wedge e_1;\]
since $\g$ is unimodular,
\[0=\Tr \ad e_3 = c_{311}+c_{322}.\]
If now $\psi$ is Killing with Killing number $w$, we have
\[w X\cdot\psi = B(X)\cdot\psi,\]
hence $B=w I+C$, where $C$ is a symmetric matrix taking values in
\[V_\psi=\{X\in\g^\C\mid X\cdot\psi=0\}.\]
We claim that $C$ is trace-free. Since $V_\psi$ has dimension one, either $C$ is zero or $\im C=V_\psi$; in the latter case, $\ker C=(\im C)^\perp=(V_\psi)^\perp$ contains $V_\psi$. This forces $C$ to be nilpotent.

Thus, $C$ is trace-free; this shows that $w=\frac13\Tr B$, and therefore every Killing spinor has Killing constant $w$.

Now let $(\phi,\psi,A,\lambda)$ be a harmful structure. Under the appropriate identification, the only $3$-vectors on $\g$ are multiples of the volume element $\omega$ of the Clifford algebra, which acts on spinors as multiplication by $i^{q+2}$. By Lemma~\ref{lemma:gamma1nonzero} it follows that $A$ is a multiple of the identity, say $A=a\id$.

If $\lambda=0$, then $\phi,\psi$ are Killing spinors with Killing numbers $\pm a$, so by the first part of the proof they are parallel.

Assume now that $\lambda$ is nonzero. By Proposition~\ref{prop:AmultipleIdentity}, the spinors
\[\eta=\lambda\psi+(w-\frac a2)\phi,\quad \xi=-\lambda\psi+ (w+\frac a2)\phi\]
are Killing with Killing numbers respectively $w$ and $-w$, where $w^2=a^2/4+\lambda^2$. By the first part of the proof, either $w=0$ and $\eta=-\xi$ is parallel, or one of $\eta,\xi$ must vanish. Thus, either there is a nonzero parallel spinor, or $\phi,\psi$ must be linearly dependent.

If $w$ is nonzero and $\phi,\psi$ are linearly dependent, the harmful structure is as determined in  Proposition~\ref{prop:PhiPsiLinDip}; $k$ is necessarily nonzero, and by the assumption on $w$, $k\neq\pm i$. By \eqref{eqn:klambdamu},  $w/\lambda$ is real, and  $a/\lambda$ is either real or imaginary depending on whether $k\in\R\cup i\R$ or $k\in S^1$; the statement follows.
\end{proof}
\begin{remark}
The argument at the beginning of the proof shows that for every metric of signature $(3,0)$ or $(1,2)$ on a unimodular Lie algebra of dimension $3$ all invariant spinors are generalized Killing. A case-by-case proof of this fact for positive definite signature appears in \cite{Artacho_2025}.
\end{remark}
Hence, by Lemma~\ref{lemma:unim3HarmAreKilling}, the classification of harmful structures is reduced to the classification of Killing spinors, which arguing as in Lemma~\ref{lemma:characterizediagonalgeneralizedKillingspinors} reduces to
\begin{equation}\label{eqn:generalizedKillingfromLi}
L_j\xi=w\xi\quad j=1,2,3\qquad\xi\in\Sigma\g,\,w\in\R\cup i\R.
\end{equation}
We obtain the following:
\begin{theorem}\label{thm:killing3dim}
The unimodular Lie algebras of dimension $3$ admitting a Killing spinor with Killing constant $w\in\R\cup i\R$ are:
\begin{itemize}
 \item $\R^3$ with any metric, and $w=0$;
 \item $\su(2)$ or  $\Sl(2,\R)$, with metric $-\frac12B$ and $w=\pm\frac14$ or metric $\frac12B$ and $w=\pm \frac{i}4$, with $B$ denoting the Killing form.
\end{itemize}
In each case, every spinor is Killing.
\end{theorem}
\begin{proof}
Let $\g$ be a unimodular Lie algebra of dimension $3$; fix a basis relative to which the metric takes the form
 \[
 g=\epsilon (e^1\otimes e^1+e^2\otimes e^2)+\epsilon_3e^3\otimes e^3,\quad\epsilon,\epsilon_3=\pm1.\]
Since \eqref{eqn:generalizedKillingfromLi} is invariant under $\Spin(p,3-p)$,
Killing spinors on $\g$ correspond to quadruplets $(0,\eta,\frac12w\id,0)$, where $\eta$ is given by Lemma~\ref{lemma:orbitseta3dim}, namely $\eta=u_0+zu_1$ with $z\in\{0,1\}$, and $w\in\R\cup i\R$, to which, though not formally a harmful structure, we can apply Lemma~\ref{lemma:unimod}. More precisely, for each instance of $\eta$, we want to obtain $L_j=M+\mu_j+\xi_j$, where $\mu_j$, $\xi_j$ and $M$ are as in Lemma~\ref{lemma:unimod}, and apply Corollary~\ref{cor:cijkfromLiwhenDiracthreeform}.

We claim that $L_j=m\omega$ for each instance of $\eta$. First notice that, since $\g$ is $3$-dimensional, both $M$ and $\xi_j$ are a multiple of the volume element, thus we may write $M=m\omega$, $\xi_j=h_j\omega$ and $\mu_j=k_j^1e_1+k_j^2e_2+k_j^3e_3$ for $m,h_j,k_j^p\in\R$. Furthermore, condition (4) of Lemma~\ref{lemma:unimod} in this case reads $(\mu_j+\xi_j)\cdot\eta=0$, which gives
\begin{equation}\label{eqn:homogeq3dimKilling}
	\tau(ik_j^1-k_j^2)z+\tau_3(ik_j^3-\epsilon h_j)=0=\tau(ik_j^1+k_j^2)-\tau_3(ik_j^3+\epsilon h_j)z\quad j=1,2,3,
\end{equation}
where $\tau,\tau_3$ are $1$ or $i$ and satisfy $\tau^2=\epsilon$, $\tau_3^2=\epsilon_3$. We now show that $\mu_j+\xi_j=0$ is the only possibility.

If $z=0$, i.e. $\eta=u_0$, by separating real and imaginary part, \eqref{eqn:homogeq3dimKilling} clearly implies $k_j^p=h_j=0$.

If $z=1$, i.e. $\eta=u_0+u_1$, \eqref{eqn:homogeq3dimKilling} depends on the signature of the metric. If $g$ is definite, separating real and imaginary part gives $k_j^1\pm k_j^3=0=k_j^2\pm h_j$, thus $\mu_j+\xi_j=0$. On the other hand, if $g$ is indefinite, $\tau\tau_3=i$ and, again by separating real and imaginary part, we obtain $k_j^1\pm h_j=0$ and $k_j^2=\epsilon_3k_j^3$.
Furthermore, conditions (2) and (3) of Lemma~\ref{lemma:unimod} imply that
\[
k_3^1=-k_2^1,\quad k_3^2=-k_1^2,\quad k_2^3=-k_1^3\quad h_3=-h_1-h_2\quad k_j^j=0,
\]
hence $\mu_j+\xi_j=0$.

Thus, we proved that $L_j=m\omega$ and by Proposition~\ref{prop:constantsfromLi} we obtain
\[
de^1=4\epsilon\epsilon_3me^{23}\quad
de^2=-4\epsilon\epsilon_3me^{13}\quad
de^3=4me^{12}.
\]
Furthermore, by \eqref{eqn:generalizedKillingfromLi} we have $w=-\epsilon\tau_3m$. Hence, either $\g$ is abelian and $\eta$ is parallel, or up to scaling of the metric, we may assume  $g=\pm\frac12B$ and $m=\frac14$. Killing spinors with Killing constants $-w$ are obtained by reversing the orientation.
\end{proof}
As a consequence, by applying Lemma~\ref{lemma:unim3HarmAreKilling}, we obtain real harmful structures on $3$-dimensional unimodular Lie algebras.
\begin{corollary}
\label{cor:harmful3dim}
If $\g$ is a $3$-dimensional unimodular Lie algebra with metric $g$ admitting a real harmful structure $(\phi,\psi,A,\lambda)$ with $A$ diagonalizable, then $\psi$ is a Killing spinor, $A$ is a multiple of the identity  and $\g$, $g$,  $A$, $\lambda$ are as in Table~\ref{table:harm3dim}; moreover,  $\phi=k\psi$ with $k\in K$ as in the table.
\end{corollary}

\begin{table}[th]
	\centering
	\caption{$3$-dimensional unimodular Lie algebras $\g$ with a metric $g$ admitting real harmful structures $(k\psi,\psi,A,\lambda)$ with $k\in K$.}\label{table:harm3dim}
	\begin{tabular}{C C C C C C}
		\toprule
		\g & g & K & A & \lambda\\
		\midrule
		\R^3 &\text{any}&\{\pm i\}&0&0\\[5pt]
		\su(2)&-\frac12B&\R\cup i\R\setminus\{0,\pm i\}&\frac{k^2-1}{2(k^2+1)}\id&-\frac k{2(1+k^2)}\\[5pt]
		\su(2)&\frac12B&S^1\setminus\{\pm i\}&i\frac{1-k^2}{2(1+k^2)}\id&i\frac k{2(1+k^2)}\\[5pt]
		\Sl(2,\R)&\frac12B&S^1\setminus\{\pm i\}&i\frac{k^2-1}{2(k^2+1)}\id&-i\frac k{2(1+k^2)}\\[5pt]
		\Sl(2,\R)&-\frac12B&\R\cup i\R\setminus\{0,\pm i\}&\frac{1-k^2}{2(1+k^2)}\id&\frac k{2(1+k^2)}\\[5pt]
		\bottomrule
	\end{tabular}
\end{table}

\section{Dimension $4$: general results}\label{sec:general4dim}
The classification of  harmful structures with $A$ diagonalizable on unimodular Lie algebras of dimension $4$ will occupy three sections. Here, we begin by specializing the results of Section~\ref{section:harmfulOnMetricLie} to four dimensions, where the multivectors of degree three $M$ and $\xi$ determine elements of $\g$ by the Hodge star operator. We then show that in definite or Lorentzian signature $*\xi$ is uniquely determined (indeed, in the Lorentzian case $*\xi$ is a multiple of the Dirac current, as will be shown in Appendix~\ref{appendix:diraccurrent}), and the operators $L_j$ are completely determined by $M$, $\xi$ and $A$. The last ingredient is a relation between $A$, $M$ and $\xi$ that paves the way for the classifications in the coming sections. Only few of the results of this section apply to the case in which $A$ is a multiple of the identity, which will be studied separately.

The final classification that we will obtain can be summarized in the following.
\begin{table}[th]
\centering
\caption{$4$-dimensional unimodular Lie algebras, according to the classification of~\cite{Ova06}, and existence of a harmful structure of Riemannian (R) or Lorentzian (L) signature with $A$ diagonalizable.
\label{table:4dimHarm}}
\begin{tabular}{L L c c c}
\toprule
\text{Name} & \g & $A=a\id$ & \multicolumn{2}{C}{A\neq a\id}\\
&& & $\lambda=0$ & $\lambda\neq0$\\
\midrule
\R^4 & (0,0,0,0) & RL &  \\
\lie{r}\lie{h}_3 & (0,0,e^{12},0)&L & RL\\
\lie{r}\lie{r}_{3,-1} & (0,-e^{12},e^{13},0)&L & RL \\
\lie{r}\lie{r}'_{3,0} & (0,-e^{13},e^{12},0) &RL&RL \\
\su(2)\times\R& (-e^{23},-e^{31},-e^{12},0)&&RL &RL \\
\Sl(2,\R)\times\R&(-e^{23},e^{31},e^{12},0)&&RL &RL \\
\lie{n}_4& (0,e^{14},e^{24},0)  &L&& \\
\lie{d}_4& (e^{14},-e^{24},-e^{12},0) & L&\\
\lie{d}'_{4,0}& (e^{24},-e^{14},-e^{12},0) &L& &RL \\
\lie{r}_{4,-1/2} & (e^{14},-\tfrac12 e^{24}+e^{34},-\tfrac12e^{34},0)&& \\
\lie{r}_{4,\mu,-1-\mu}& (e^{14},\mu e^{24},-(\mu+1) e^{34},0)&&\\
\lie{r}'_{4,-1/2,\delta} & (e^{14},-\tfrac12 e^{24}+\delta e^{34}, -\delta e^{24}-\tfrac12\mu e^{34})&&\\
\bottomrule
\end{tabular}
\end{table}
\begin{theorem}
\label{main:generaltheorem4}
Table~\ref{table:4dimHarm} contains a classification of four-dimensional unimodular Lie algebras admitting a harmful structure of Riemannian or Lorentzian signature of the following types:
\begin{itemize}
\item with $A$ a multiple of the identity, in which case the spinors $\phi,\psi$ are parallel;
\item with $A$ diagonalizable but not a multiple of the identity, and $\lambda=0$;
\item with $A$ diagonalizable but not a multiple of the identity, and $\lambda\neq0$.
\end{itemize}
\end{theorem}
The list of unimodular Lie algebras of dimension $4$ in Table~\ref{table:4dimHarm} has been extracted from the classification of~\cite{Ova06}; the structure constants are given with the notation introduced in Example~\ref{ex:consistencycheck}.

Our first step toward the classification is to replace Lemma~\ref{lemma:unimod} with the following:
\begin{lemma}
\label{lemma:unimod4}
Let $\g$ be a unimodular Lie algebra of dimension $4$ with a metric of signature $(p,q)$ of the form~\eqref{eqn:generaldiagonalmetric}, and a real harmful structure $(\phi,\psi,A,\lambda)$ with $A=\sum a_je^j\otimes e_j$. Then
\begin{gather*}
L_j=M+\mu_j+\xi_j,
\end{gather*}
where:
\begin{enumerate}
\item $\mu_j$ has degree one, and $M,\xi_j$ have degree three;
\item $g(\mu_j,e_j)=0$;
\item $\sum \mu_j=0=\sum a_j\mu_j=\sum \xi_j$;
\item $e_j\wedge (M+\xi_j)=0$;
\item $(\mu_j+\xi_j)\psi=\frac12(a_j-\frac14\Tr A)\psi$ for all $j$;
\item $g(*M,{*}M)=(-1)^{q}(\lambda^2+\frac1{64}(\Tr A)^2)$;
\item the harmful structure  $(\phi,\psi,A,\lambda)$ satisfies
\begin{equation}\label{eqn:*Mcdotpsi+}
\begin{gathered}
	*M\cdot\psi_+=i^{q}(i\lambda-\frac18\Tr A)\psi_-,\\
	*M\cdot\psi_-= i^{q}(i\lambda+\frac18\Tr A)\psi_+.
	\end{gathered}\end{equation}
\end{enumerate}
\end{lemma}
\begin{proof}
The first three conditions are proved in Lemma~\ref{lemma:unimod}, and the fourth follows from~\eqref{eqn:L_i_satisfy}. By Lemma~\ref{lemma:characterizediagonalharmful}, we have
\begin{equation}
 \label{eqn:lambdaphiequalsdim4}
 \lambda\phi = L_j\psi - \frac12a_j\psi = (M + \mu_j+\xi_j)\cdot \psi - \frac12a_j\psi,
\end{equation}
where \[\phi = i (\psi_+-\psi_-).\]
Averaging~\eqref{eqn:lambdaphiequalsdim4} over the index $i$, one obtains
\begin{equation}
\label{eqn:lambdaphiMcpsitrApsi}
\lambda\phi = M\cdot\psi -\frac18(\Tr A)\psi,
\end{equation}
Additionally,
\begin{equation}
\label{eqn:MdotpsiintermsofstarM}
M\cdot\psi = i^{-q+2}(*M)\cdot(\psi_+-\psi_-),
\end{equation}
so we get
\begin{gather}
\label{eqn:Mpsiplus}
*M\cdot\psi_+=i^{q-2}(-i\lambda+\frac18\Tr A)\psi_-\\
\label{eqn:Mpsiminus}
*M\cdot\psi_-= i^{q-2}(-i\lambda-\frac18\Tr A)\psi_+
\end{gather}
Taking the Clifford multiplication by $*M$ of each of~\eqref{eqn:Mpsiplus}, \eqref{eqn:Mpsiminus} and substituting in the other, we get
\[(*M)\cdot(*M)\cdot \psi_\pm = (-1)^{q-2}(-\lambda^2-\frac1{64}(\Tr A)^2)\psi_\pm,\]
and since $\psi$ is non-zero this determines $g(*M,{*}M)$.
\end{proof}

\begin{remark}
Using the equivalent of \eqref{eqn:MdotpsiintermsofstarM}, the condition $\xi\cdot\psi=\psi$ of Lemma~\ref{lemma:gamma1nonzero} can be rewritten as
\begin{equation}
 \label{eqn:nonlethal4}
i^{2-q}{*}\xi\cdot\psi_+=\psi_-, \qquad
-i^{2-q}{*}\xi\cdot\psi_-=\psi_+,\qquad
g({*}\xi,{*}\xi)=(-1)^q,
\end{equation}
where any two imply the third.
\end{remark}

\begin{remark}
\label{remark:ScalarProductXiM}
Using Lemma~\ref{lemma:unimod4} and \eqref{eqn:nonlethal4}, we compute
\begin{multline*}
-2g(*\xi,{*}M)\psi_+ = *\xi\cdot {*}M\cdot \psi_+ + {*}M\cdot {*}\xi\cdot\psi_+
 =*\xi\cdot i^q(i\lambda - \frac18\Tr A)\psi_- + {*}M\cdot i^{q-2}\psi_-\\
 =(-1)^q(i\lambda - \frac18\Tr A)\psi_+ - (-1)^q(i\lambda+\frac18\Tr A)\psi_+=-(-1)^q\frac14\Tr A\psi_+,
\end{multline*}
i.e.
\begin{equation}\label{eq:ScalarProductXiM}
g({*}\xi,{*}M)= (-1)^q\frac18\Tr{A}.
\end{equation}
Still keeping into account Lemma~\ref{lemma:unimod4} and \eqref{eqn:nonlethal4}, we see that the restriction of the metric to $\Span{*M,{*}\xi}$ is determined by
\[\begin{pmatrix} g(*M,{*}M) & g({*}M,{*}\xi)\\g({*}M,{*}\xi)&g({*}\xi,{*}\xi)\end{pmatrix}= (-1)^q\begin{pmatrix}\lambda^2+\frac1{64}(\Tr A)^2 & \frac18\Tr A\\\frac18\Tr A & 1\end{pmatrix},\]
which is nonsingular when $\lambda$ is non-zero.
\end{remark}

Some of the results of this section are best expressed in terms of the traceless part of $A$, namely
\[A_0(X)=A(X)-\frac14(\Tr A)(X).\]
\begin{lemma}
\label{lemma:alternative}
Let $\g$ be a four-dimensional unimodular Lie algebra with  a real harmful structure $(\phi,\psi,A,\lambda)$, with $A$ diagonalizable. If $M=\frac14(L_1+\dotsc + L_4)$, then
\[(\nabla_X {*}M)\cdot\psi_\pm=\lambda i^{q+1}A_0(X)\cdot\psi_\pm +g(A_0(X),{*}M)\psi_\mp+(\frac14\Tr A\mp 2i\lambda)g(X,{*}M)\psi_\mp.\]
\end{lemma}
\begin{proof}
Write \eqref{eqn:*Mcdotpsi+} as
\[
*M\cdot\psi_\pm
=\mp i^{q}(\frac18\Tr A \mp i\lambda)\psi_\mp.\]
Then
\[(\nabla_X {*}M)\cdot\psi_\pm  + {*}M\cdot\nabla_X\psi_\pm \pm i^{q}(\frac18\Tr A \mp i\lambda)\nabla_X\psi_\mp=0\]
Separating according to chirality, we can write  \eqref{eqn:systemSpace} as
\[\nabla_X \psi_\pm = (\frac12 A(X)\mp i\lambda X)\cdot \psi_\mp.\]
We get
\begin{multline*}
0=(\nabla_X {*}M)\cdot\psi_\pm  + {*}M\cdot(\frac12 A(X) \mp i\lambda X)\cdot \psi_\mp \pm i^{q}(\frac18\Tr A \mp i\lambda)(\frac12 A(X) \pm i\lambda X)\cdot \psi_\pm\\
=(\nabla_X {*}M)\cdot\psi_\pm  -(\frac12 A(X) \mp i\lambda X)\cdot  {*}M\cdot\psi_\mp -g(A(X)\mp 2i\lambda X,{*}M)\psi_\mp \\
\pm
i^{q}(\frac18\Tr A \mp i\lambda)(\frac12 A(X) \pm i\lambda X)\cdot \psi_\pm\\
=(\nabla_X {*}M)\cdot\psi_\pm  \mp(\frac12 A(X) \mp i\lambda X)\cdot
i^{q}(\frac18\Tr A \pm i\lambda)\psi_\pm -g(A(X)\mp 2i\lambda X,{*}M)\psi_\mp \\
\pm
i^{q}(\frac18\Tr A \mp i\lambda)(\frac12 A(X) \pm i\lambda X)\cdot \psi_\pm\\
=(\nabla_X {*}M)\cdot\psi_\pm
+\lambda i^{q+1}(- A(X)
+\frac14 (\Tr A)X)\cdot\psi_\pm -g(A(X)\mp 2i\lambda X,{*}M)\psi_\mp;
\end{multline*}
plugging in the definition of $A_0$, we obtain the statement.
\end{proof}

It will be convenient in subsequent computations to denote by $V_{\psi}$ the subspace of $\g^\C$ that annihilates a spinor $\psi$, i.e.
\[V_{\psi}=\{v\in\g^\C\mid v\cdot\psi=0\}.\]
We observe that $V_{\psi}$ is totally isotropic, i.e. for $u,w\in V_{\psi}$ we have
\begin{equation*}\label{eqn:VpsiIsotropic}
-2g(u,w)\psi=(u\cdot w + w\cdot u)\psi=0,
\end{equation*}
thus $g(u,w)=0$.

\begin{lemma}\label{lemma:4OnlyTrivialSolution}
Let $\g$ be a Lie algebra with a metric of signature $(4,0)$, $(3,1)$ or $(0,4)$, let $\psi$ be a non-zero spinor, and assume that there exists  $\xi\in\Lambda^3\g $ such that $\xi\cdot\psi=\psi$.  Then the equation
\[(\mu + \eta)\cdot\psi=0,\quad \mu\in\g, \eta\in\Lambda^3\g\]
has unique solution $\mu=0$, $\eta=0$.
\end{lemma}
\begin{proof}
Since $\xi\cdot\psi_+=\psi_-$ and $\xi\cdot\psi_-=\psi_+$, both $\psi_+$ and $\psi_-$ are non-zero.

In the definite case,  observe that,  by ~\eqref{eqn:starproductOnChiral}, for $\eta\in\Lambda^3\g$ we have that ${*}\eta$ is a real vector and $\eta\cdot\psi_\pm=\mp{*}\eta\cdot\psi_\pm$ on chiral spinors. Hence the condition $(\mu + \eta)\cdot\psi=0$  implies
\[(\mu - {*}\eta)\cdot\psi_+=0,\qquad (\mu + {*}\eta)\cdot\psi_-=0,\]
where $\mu, {*}\eta\in\g$. This implies that $\mu-{*}\eta$, $\mu+{*}\eta$ are isotropic, hence zero, i.e. $\mu=0$, $\eta=0$.

If the signature is $(3,1)$, by~\eqref{eqn:starproductOnChiral} the condition $(\mu + \eta)\cdot\psi=0$  implies
\[(\mu +i{*}\eta)\cdot\psi_+=0,\qquad (\mu - i{*}\eta)\cdot\psi_-=0,\]
where $\mu, {*}\eta\in\g$, i.e.
\[(\mu +i{*}\eta)\in V_{\psi_+} \cap\overline{V_{\psi_-}}.\]
Up to the action of $\Spin(3,1)$ and multiplying $\psi$ by a non-zero complex scalar, we can assume $\psi_+=u_0$ and $*\xi=e_4$, which implies
\[\psi_-=ie_4\cdot u_0=u_2.\]
Therefore,
\[V_{\psi_+}=\Span{e_1-ie_2, e_3-e_4}_{\C},\quad V_{\psi_-}=\Span{e_1-ie_2,e_3+e_4}_{\C}
,\]
and  $V_{\psi_+}\cap \overline{V_{\psi_-}}$ is trivial.
\end{proof}

\begin{remark}\label{rmk:Lambda0Dipendenti}
Given a harmful structure $(\phi,\psi,A,\lambda)$ with $\lambda\neq0$, the computation of Remark~\ref{remark:ScalarProductXiM} shows that ${*}M$ and ${*}\xi$ are linearly independent and span a nondegenerate subspace.

On the other hand, given a harmful structure  $(\phi,\psi,A,\lambda)$ with $\lambda=0$ and $\xi\cdot\psi=\psi$, by~\eqref{eqn:lambdaphiMcpsitrApsi} we have
\[M\cdot\psi =(\frac18\Tr A)\psi =(\frac18\Tr A)\xi\cdot\psi.\]
If we additionally assume that the  signature equals $(4,0)$ or $(3,1)$, it follows by Lemma~\ref{lemma:4OnlyTrivialSolution} that ${*}M$ and ${*}\xi$ are linearly dependent, and more explicitly
\[{*}M=(\frac18\Tr A){*}\xi.\]
\end{remark}

\begin{proposition}
\label{prop:nablastarMskew}
Let $\g$ be a  unimodular Lie algebra of dimension $4$ with a harmful structure of signature $(4,0)$ or $(3,1)$ with $A$ diagonalizable and not a multiple of the identity, and let $\xi$ be the unique element of $\Lambda^3\g$ such that $\xi\cdot\psi=\psi$. Assume that $\psi$ is not parallel. Then the following hold:
\begin{enumerate}
\item $\nabla {*}\xi$ is skew-symmetric;
 \item $\nabla {*}M$ is skew-symmetric;
 \item whenever $\g_\mu$ and $\g_\nu$ are distict eigenspaces of $A$,
 \[\g_\mu\wedge \g_\nu\wedge {*M}\wedge {*}\xi=0;\]
 \item if $\lambda\neq0$, then $\g=W\oplus W^\perp$, where $W=\Span{*\xi,{*}M}$, and $W^\perp$ is contained in an eigenspace of $A_0$.
 \end{enumerate}
\end{proposition}
\begin{proof}
The proof is a linear algebra computation exploiting Corollary~\ref{cor:cijkfromLiwhenDiracthreeform} and the Jacobi identity.

For the first item, since $e_j\cdot L_j$ is a bivector, for any $v\in\g$, $\psi\in\Sigma$ we have
\[(\nabla_{e_j} v)\cdot \psi = \nabla_{e_j} (v\cdot\psi)-v\cdot\nabla_{e_j}\psi=(e_j\cdot L_j\cdot v - v\cdot e_j\cdot L_j)\cdot\psi = (2v\hook (e_j\cdot L_j))\cdot\psi.
\]
Since this holds for all $\psi$, it follows that
\[\nabla_{e_j} v = 2v\hook (e_j\cdot L_j).\]
In particular,
\[g(\nabla_{e_j} v,e_k)=2(e_j\cdot L_j)( v^\flat,e_k^\flat).\]
Using that each $L_i$ has degree three, we can write
\[e_j\cdot L_j=-e_j\hook L_j\]
and
\[g(\nabla_{e_j} v,e_k)=-2L_j(e_j^\flat, v^\flat, e_k^\flat).\]
Thus $\nabla v$ is skew-symmetric if
\[L_j(e_j^\flat, v^\flat, e_k^\flat)=-L_k(e_k^\flat, v^\flat, e_j^\flat)=L_k(e_j^\flat, v^\flat, e_k^\flat).\]
Writing now $L_j=M+\frac12x_j \xi$, $\nabla v$ is skew-symmetric if
\begin{equation}
 \label{eqn:nablavskewsymmetric}
 (x_j-x_k)\xi(e_j^\flat,v^\flat,e_k^\flat)=0.
\end{equation}
Setting $v={*}\xi$ and using \eqref{eqn:hookandwedgeadual} we obtain the first item.

We will prove the second and third items simultaneously. For $\lambda=0$, both follow from \eqref{eqn:lambdaphiMpsi} and Lemma~\ref{lemma:4OnlyTrivialSolution}, implying that $M$ and $\xi$ are linearly dependent. Assume now that $\lambda\neq0$; this is where the Jacobi identity enters.  Choose an orthonormal basis of eigenvectors of $A$, and write
\[L_j=M+\frac12x_j\xi, \quad \sum x_j=0.\]
Then, by Corollary~\ref{cor:cijkfromLiwhenDiracthreeform},
\[de^k=(e^k\hook (6M-x_k\xi))^\flat.\]
Thus\[
\begin{split}
0&=(d^2e^k)^\sharp = \sum_j (de^j \wedge (e_j\hook de^k))^\sharp\\
&=\sum_j (de^j)^\sharp \wedge (e_j^\flat\hook (de^k)^\sharp)
=\sum_j (de^j)^\sharp\wedge ((e^k\wedge e_j^\flat)\hook (6M-x_k\xi))\\
&=\sum_j(e^j\hook (6M-x_j\xi))\wedge ((e^k\wedge e_j^\flat)\hook (6M-x_k\xi)).
\end{split}\]
Observe that $(e^j\hook (6M-x_j\xi))$ and $((e^k\wedge e_j^\flat)\hook (6M-x_k\xi))$ lie in the exterior algebra over $e_j^\perp\subset\g$; since we have fixed the dimension to four, each term in the sum vanishes and we obtain
\[ 0=(e^j\hook (6M-x_j\xi))\wedge ((e^k\wedge e_j^\flat)\hook (6M-x_k\xi)).\]
Moreover,
\begin{equation*}
 \label{eqn:becausethreedim}
 (e^j\hook M)\wedge(e^k\wedge e^j)\hook M= -\frac12e^k\hook (e^j\hook M)^2=0,
\end{equation*}
and similarly for $\xi$.
Let $M^{(j)}=e^j\hook M$, $\xi^{(j)}=e^j\hook\xi$. Then $M^{(j)}\wedge \xi^{(j)}=0$, and
\[
\begin{split}
 0&=(e^j\hook (6M-x_j\xi))\wedge ((e^k\wedge e_j^\sharp)\hook (6M-x_k\xi))\\
 &=-6x_k(e^j\hook M)\wedge ((e^k\wedge e_j^\sharp)\hook \xi)
 -6 x_j(e^j\hook \xi)\wedge ((e^k\wedge e_j^\sharp)\hook M)\\
 & =6\epsilon_j \bigl(x_kM^{(j)}\wedge (e^k \hook\xi^{(j)})
 + x_j\xi^{(j)}\wedge (e^k\hook M^{(j)})\bigr)\\
   &=6\epsilon_j \bigl(x_kM^{(j)}\wedge (e^k \hook\xi^{(j)})
 - x_j(e^k\hook \xi^{(j)})\wedge M^{(j)})\bigr)\\
 &=6\epsilon_j(x_k-x_j) M^{(j)}\wedge (e^k \hook\xi^{(j)}).
\end{split}\]
Equivalently, for $x_k\neq x_j$ we have
\[0=e_j\wedge M^{(j)}\wedge (e_k \hook\xi^{(j)}),\]
which using~\eqref{eqn:hookandwedgeadual} implies
\[\begin{split}
0&=g(e_j\wedge M^{(j)},*(e_k \hook\xi^{(j)}))= -g(e_j\wedge M^{(j)},e_k\wedge *(\xi^{(j)}))\\
&= -g(e_j\wedge M^{(j)},e_k\wedge e_j\wedge {*}\xi)= -g(M,e_k\wedge e_j\wedge {*}\xi),
\end{split}\]
and therefore
\[*M\wedge e_k\wedge e_j\wedge {*}\xi=0,\]
proving the third item. This implies that \eqref{eqn:nablavskewsymmetric} holds, and so $\nabla {*}M$ is skew-symmetric.

For the last item, observe  first that $W=\Span{*M,{*}\xi}$ is a nondegenerate subspace of dimension $2$ by Remark~\ref{rmk:Lambda0Dipendenti} and the assumption that $\lambda$ is non-zero. Since $A_0$ is diagonalizable, we can complete $\{*M,{*}\xi\}$ to a basis of $\g$ by adding two eigenvectors of $A_0$, call them $v_1,v_2$.  By the third item, $v_1$ and $v_2$ lie in the same eigenspace $\g_\alpha$. For the same reason, any other eigenspace $g_\beta$ is contained in $W$, as
\[\g_\beta\wedge v_1\wedge {*}M\wedge {*}\xi=0=\g_\beta\wedge v_2\wedge {*}M\wedge{*}\xi.\]
Since $\g_\alpha$ is the orthogonal complement of the sum of the other eigenspaces, which are contained in $W$, it follows that $W^\perp\subset \g_\alpha$.
\end{proof}

We can now obtain a characterization that applies to the case where $A$ is not a multiple of the identity.
\begin{proposition}\label{prop:LieAlgebraFromMXiLi}
Fix a metric on $\R^4$ of the form~\eqref{eqn:generaldiagonalmetric} and signature $(4,0)$ or $(3,1)$. Let $M,\xi$ be elements of $\Lambda^3\R^4$, and let $x_1,x_2,x_3,x_4$ in $\R$ be such that
\begin{gather*}
x_1+x_2+x_3+x_4=0;\\
(M+\frac12x_j\xi)\wedge e_j=0, \quad j=1,2,3,4;\\
(x_j-x_k)e_j\wedge e_k\wedge {*}M\wedge {*}\xi=0, \quad 1\leq j,k\leq 4.
\end{gather*}
Let $\psi$ be a spinor with $\xi\cdot\psi=\psi$, $y\in\R$, $\lambda\in\R\cup i\R$ such that
\[M\cdot\psi_+= (\frac12y-\lambda i)\psi_-, \quad M\cdot\psi_-= (\frac12y+i\lambda)\psi_+.\]
Then setting
\[de^k=(e^k\hook (6M-x_k\xi))^\flat, \quad \phi=i(\psi_+-\psi_-), \quad A=\sum (x_j+y)e^j\otimes e_j\]
defines a Lie algebra structure on $\R^4$ with a real harmful structure $(\phi,\psi,A,\lambda)$. Every harmful structure on a unimodular four-dimensional Lie algebra with $A$ diagonalizable and not a multiple of the identity,  and signature $(4,0)$, $(3,1)$ is obtained in this way.
\end{proposition}
\begin{proof}
Set $L_j=M+\frac12x_j\xi$. By Proposition~\ref{prop:constantsfromLi} this defines a Lie bracket operation on $\R^4$, where, by Corollary~\ref{cor:cijkfromLiwhenDiracthreeform}, the $de^k$ are given as in the statement. The computation of Proposition~\ref{prop:nablastarMskew} shows that the Jacobi identity holds, so we obtain a Lie algebra structure on $\R^4$, and $L_j\cdot\psi=-\epsilon_j e_j\nabla_{e_j}\psi$ for all spinors $\psi$. The proof of Lemma~\ref{lemma:unimod} implies that the hypotheses of Lemma~\ref{lemma:characterizediagonalharmful} are satisfied, and we obtain a harmful structure.

For second part of the statement, apply Lemma~\ref{lemma:unimod4} and observe that by Lemma~\ref{lemma:4OnlyTrivialSolution} $(\mu_j+\xi_j)\cdot\psi=\frac12x_j\psi$ implies $\mu_j=0$, $\xi_j=\frac12 x_j\xi$. Proposition~\ref{prop:nablastarMskew} implies that $(x_j-x_k)e_j\wedge e_k\wedge {*}M\wedge {*}\xi$ must vanish, and the conditions on $M\cdot\psi_\pm$ correspond to \eqref{eqn:*Mcdotpsi+}, with $\Tr A=4y$.
\end{proof}
Proposition~\ref{prop:LieAlgebraFromMXiLi} by itself is not sufficient to obtain a classification, because the orthonormal basis diagonalizing $A$ is unrelated to $M$ and $\xi$. We conclude this section by showing that a relation can be established under the assumption that $\lambda i^q$ is real. We will see in the coming sections that this assumption is automatically satisfied in signatures $(4,0)$ and $(3,1)$.

\begin{proposition}
\label{prop:A0*xigeneral}
Let $\g$ be a four-dimensional unimodular Lie algebra with  a real harmful structure $(\phi,\psi,A,\lambda)$ of signature $(p,q)$. Suppose that $\lambda i^q$ is real and non-zero. If $\xi\in\Lambda^3\g$ is such that $\xi\cdot\psi=\psi$, then
\[A_0(*\xi)=-2{*}M, \quad g(\nabla_X {*}M,{*}\xi) =0.\]
\end{proposition}
\begin{proof}
We will construct two distinct vectors in the isotropic space annihilating one of $\psi_+$ or $\psi_-$; the vanishing of their scalar product determines a complex equation, whose real and imaginary part give the statement.

It is clear that one of $i\lambda\mp\frac18\Tr A$ is non zero. Assume that $i\lambda-\frac18\Tr A\ne0$. We can rewrite~\eqref{eqn:*Mcdotpsi+} and~\eqref{eqn:nonlethal4}  as
\begin{equation}\label{eqn:*xi*Mcdotpsipm}
(i\lambda-\frac18\Tr A)\psi_- =i^{-q}{*M}\cdot\psi_+, \quad \psi_-=-i^{-q}{*}\xi\cdot\psi_+.
 \end{equation}
We consider $(\nabla_X{*}M)\cdot\psi_+$ in Lemma~\ref{lemma:alternative} and compute
 \[\begin{split}
    0&=((\nabla_X {*}M)+ \lambda i^{q-1}A_0(X))\cdot\psi_+-(g(A_0(X)-2( i\lambda -\frac18\Tr A)X,{*}M))\psi_-\\
&=\big((\nabla_X {*}M)+ \lambda i^{q-1}A_0(X)+ i^{-q}g(A_0(X),{*}M){*}\xi+ 2i^{-q}g(X,{*}M){*}M\big)\cdot\psi_+
\end{split}\]
i.e.
\[Y=\nabla_X {*}M + \lambda i^{q-1}A_0(X)+i^{-q}g(A_0(X),{*}M){*}\xi+2i^{-q}g(X,{*}M){{*}M}
\in V_{\psi_+}.\]
By~\eqref{eqn:*xi*Mcdotpsipm}, we have $(i\lambda-\frac18\Tr A){*}\xi+{*}M\in V_{\psi_+}$ and since the scalar product of two vectors in $V_{\psi_\pm}$ is zero,
\begin{equation}
 \label{eqn:A0*xigeneral:orthogonalvectors}
(i\lambda-\frac18\Tr A)g(Y,{*}\xi)+g(Y,{*}M)=0.
 \end{equation}
Recalling \eqref{eq:ScalarProductXiM}, we obtain
\[
\begin{split}
    g(Y,{*}M)&=-i^q(i\lambda-\frac18\Tr A)g(A_0(X),{*}M)+ 2i^qg(X,{*}M)(\lambda^2+\frac1{64}(\Tr A)^2),\\
    g(Y,{*}\xi)&=g(\nabla_X {*}M,{*}\xi) + \lambda i^{q-1}g(A_0(X),{*}\xi)+i^{q}g(A_0(X),{*}M)\\
    &\quad+i^{q}g(X,{*}M)\frac14\Tr A.
\end{split}
\]
Since we assume $i\lambda-\frac18\Tr A\ne0$, \eqref{eqn:A0*xigeneral:orthogonalvectors} is equivalent to
\[
\begin{split}
    0&=g(Y,{*}\xi)+(i\lambda-\frac18\Tr A)^{-1}g(Y,{*}M)\\
&=g(\nabla_X {*}M,{*}\xi) + \lambda i^{q-1}g(A_0(X),{*}\xi)+i^{q}g(A_0(X),{*}M)+2i^{q}g(X,{*}M)\frac18\Tr A\\
&\quad- i^qg(A_0(X),{*}M)-2i^qg(X,{*}M)(i\lambda+\frac18\Tr A)\\
&=g(\nabla_X {*}M,{*}\xi) + \lambda i^{q-1}g(A_0(X),{*}\xi)+ 2i^{q-1}\lambda g(X,{*}M).
\end{split}
\]
Using that $A_0$ is symmetric it follows that
\[0= \lambda i^{q-1}g(X,2{*}M+A_0(*\xi))+g(\nabla_X {*}M,{*}\xi).\]
The same equation is obtained through an analogous computation if we assume $i\lambda+\frac18\Tr A\ne0$ and consider $(\nabla_X{*}M)\cdot\psi_-$ in Lemma~\ref{lemma:alternative}. In both cases, separating the real and imaginary part we obtain the statement.
\end{proof}

\section{Dimension $4$: Killing spinors are parallel}
\label{sec:Amultipleofidentity}
In this section we consider harmful structures $(\phi,\psi,A,\lambda)$ on $4$-dimensional unimodular Lie algebras with $A$ a multiple of the identity and signature different from $(2,2)$, and show that they are determined by parallel spinors, i.e. $A=0=\lambda$.  We write down all $4$-dimensional unimodular Lie algebras and Lorentzian metrics admitting a parallel spinor. Comparing with the classification of~\cite{Ova06}, we obtain a classification of unimodular Lie algebras admitting a Lorentzian metrics with a parallel spinor.

By Proposition~\ref{prop:AmultipleIdentity}, harmful structures with $A$ a multiple of the identity are obtained as linear combinations of Killing spinors. Thus, we only need to prove that Killing spinors on a $4$-dimensional unimodular Lie algebras with a metric of signature different from $(2,2)$ are parallel. When the metric is definite, this can be seen as a consequence of the characterization of complete manifolds with a Killing spinor in~\cite{Bar:RealKillingSpinors} and~\cite{Baum:CompleteRiemannian}, though it can also be deduced directly as follows:
\begin{proposition}
\label{prop:definiteAmultipleodentity}
Let $\g$ be a four-dimensional unimodular Lie algebra with  a real harmful structure $(\phi,\psi,A,\lambda)$ of definite signature. If $A$ is a multiple of the identity, then $\psi$ is parallel and $\g$ is flat; in particular, $\g$ is either  abelian or $\lie r \lie r'_{3,0}$.
\end{proposition}
\begin{proof}
Assuming $A_0=0$, Lemma~\ref{lemma:alternative} and~\eqref{eqn:*Mcdotpsi+} give
\[(\nabla_X {*}M)\cdot \psi_\pm =2(\frac18\Tr A\mp i\lambda )g(X,{*}M)\psi_\mp=\mp 2g(X,{*}M){*} M\cdot\psi_\pm;\]
since the metric is definite, this implies
\[\nabla_X {*}M=\mp 2g(X,{*}M){* M}, \quad X\in\g.\]
Since $\nabla_X {*}M$ is orthogonal to $*M$, it follows that $*M$ is zero.

If $\psi$ is chiral, then it is parallel by Proposition~\ref{prop:harmfulpropertiesMeven}. Otherwise, \eqref{eqn:*Mcdotpsi+} implies that $\lambda$ and $\Tr A$ are zero, and again $\psi$ is parallel.

This implies that $\g$ is Ricci-flat, hence flat by~\cite{AlekseevskyKimelfeld}. By~\cite{Milnor_1976}, $\g=\lie{h}\rtimes\lie{a}$, where the sum is orthogonal, $\lie h$ and $\lie a$ are abelian, and $\lie a$ acts on $\lie h$ by skew-symmetric derivations. Since $\R^2$ only has one skew-symmetric derivation, we can assume that $\lie h$ has dimension $3$, and $\g$ is determined up to isometric isomorphism by a skew-symmetric endomorphism $D$  of $\R^3$. Up to an orthonormal change of basis, we can assume $D=\lambda(e^2\otimes e_3-e^3\otimes e_2)$; since rescalings are isometric isomorphism, we can further assume $\lambda=0,1$, obtaining either $\R^4$ or $\lie r \lie r'_{3,0}$.
\end{proof}

Now assume that the metric has signature $(p,q)=(3,1)$ or $(1,3)$. In the following, we prove that Killing spinors are parallel also in this case and derive a classification. We first find all possible solutions in signature $(3,1)$ to the equations~\eqref{eqn:*Mcdotpsi+} for $\lambda=0$ and $\Tr A=4w$; since we are interested in Killing spinors, we will allow  $w\in\R\cup i\R$.
\begin{lemma}\label{lemma:31}
Let $\g$ be a $4$-dimensional unimodular Lie algebra with a metric of signature $(3,1)$. Let $M\in\Lambda^3\g$, and suppose that $\eta$ is a non-zero spinor satisfying
\[
*M\cdot\eta_+=-\frac12 i w\eta_-,  \quad *M\cdot\eta_-=\frac12 iw\eta_+, \quad w\in\R\cup i\R.
\]
Up to rescaling $M$ and $\eta$, there is an orthonormal basis such that the metric takes the form~\eqref{eqn:standard31metric}
and one of the following holds:
\begin{compactenum}
	\item  $M=0$, $w=0$, and $\eta\in\{u_0+ru_1,u_0+ru_2,u_1\mid r\geq0\}$;
	\item $*M=e_3-e_4$, $w=0$, and $\eta\in\{u_0+\alpha u_1,u_1\mid \alpha\in\C\}$;
	\item  $*M=e_1$, $w=\pm 2i$ and $\eta= u_0\pm i u_1$;
	\item  $*M=e_3$, $w=\pm 2i$ and $\eta=u_0\mp i u_2$; or
	\item  $*M=e_4$, $w=\pm2$ and $\eta=u_0\pm u_2$.
\end{compactenum}
\end{lemma}
\begin{proof}
Assume first that the spinor $\eta$ is nonchiral. We may apply Lemma~\ref{lemma:orbitsPsiM} with $v={*}M$, and assume $\eta_+=u_0$.  Notice also that $*M\cdot{*}M\cdot\eta_\pm=\frac14w^2\eta_\pm$, hence
\begin{equation}
 \label{eqn:gstarmstarm}
 g(*M,{*}M)=-\frac14w^2.
\end{equation}
According to Lemma~\ref{lemma:orbitsPsiM}, we must consider the following possibilities.
	\begin{itemize}
		\item If $M=0$, then $w\eta_-$ and $w\eta_+$ must vanish. Since $\eta$ is non-zero, it follows that $w=0$. In this particular case, we can use the stabilizer of $[u_0]$ to act on $\eta_-$: e.g., the element $\exp (xe_1\cdot (e_3-e_4))$ acts on $\Span{u_1,u_2}$ as
		\begin{equation*}
			\exp x\begin{pmatrix}
				0 & 2  \\
				0 & 0 \\
			\end{pmatrix}=
			\begin{pmatrix}
				1 & 2 x \\
				0 & 1 \\
			\end{pmatrix}.
		\end{equation*}

		It follows that $\eta_-$  may be assumed to be a multiple of $u_1$ or $u_2$. We can then act by the abelian Lie group with Lie algebra $\Span{e_1\cdot e_2,e_3\cdot e_4}$, which maps
		\[u_0\mapsto iu_0, u_1\mapsto -iu_1, u_2\mapsto iu_2, \quad u_0\mapsto -u_0, u_1\mapsto -u_1, u_2\mapsto u_2\]
so we can assume the spinor is $u_0 + ru_1$, $u_0+ru_2$ with $r\geq0$.
		\item If $*M=e_3-e_4$ then~\eqref{eqn:gstarmstarm} implies $w=0$, thus  $(e_3-e_4)\cdot\eta_-=0$, which implies $\eta_-=\alpha u_1$ for some $\alpha\in\C$.
		\item If $*M=e_3+e_4$ again~\eqref{eqn:gstarmstarm} implies $w=0$, so $0=(e_3+e_4)\cdot u_0=-2iu_2$, which is absurd.
		\item If $*M=me_1$,  by  rescaling we may assume $m=1$. It follows that
		\[i u_1=e_1\cdot u_0=-\frac i2w\eta_-,\]
		hence $\eta_-=-\frac{2}{w}u_1$. Moreover~\eqref{eqn:gstarmstarm} implies $w=\pm 2i$.
		\item If $*M=me_3$, by rescaling  we may assume $m=1$. It follows that
		\[-iu_2=e_3\cdot u_0=-\frac i2w\eta_-,\]
		hence $\eta_-=\frac2wu_2$ and~\eqref{eqn:gstarmstarm} gives $w=\pm2i$.
		\item If $*M=me_4$, by rescaling  we may assume $m=1$. It follows that
		\[-i u_2=e_3\cdot u_0=-\frac12w\eta_-,\]
		hence $\eta_-=\frac2w u_2$ and~\eqref{eqn:gstarmstarm} gives $w=\pm2$.
	\end{itemize}
Now assume $\eta\in\Sigma^+_{3,1}$. By Lemma~\ref{lemma:orbitsPsiM} we may assume $\eta=u_0$, and the fact that $M\cdot\eta=0$ implies $g(*M,{*}M)=0$, giving the following possibilities for $*M$:
\[0,\ e_3-e_4,\ e_3+e_4.\]
However, $0=(e_3+e_4)\cdot u_0=-2iu_2$ which is absurd, and the only possiblities are $*M=0$ or $*M=e_3-e_4$.

If $\eta\in\Sigma^-_{3,1}$ we may assume $\eta=u_1$ since $\Spin(3,1)$ acts transitively also on $\Sigma^-_{3,1}$, and the same reasoning as above holds.
\end{proof}

We can now show that on a $4$-dimensional unimodular Lorentzian (or anti-Lorentzian) Lie algebra any Killing spinor is parallel. We exclude the trivial case where the Lie algebra is abelian and all spinors are parallel.
\begin{theorem}\label{thm:killing31}
Let $\g$ be a $4$-dimensional, nonabelian, unimodular Lie algebra with a metric of the form
\[g=\pm (e^1\otimes e^1+e^2\otimes e^2+e^3\otimes e^3-e^4\otimes e^4)\]
and a Killing spinor $\eta$. Then $\eta$ is parallel and up to $\Spin(p,q)$-symmetry and rescaling of $\eta$ and $g$, one of the following holds:
\begin{enumerate}
 \item[I.]
  $\eta\in\{u_0+ru_1,u_1\mid r\geq0\}$ and
\begin{gather*}
\g=(y(e^{23}+e^{24})+x(e^{13}+e^{14}),y(e^{13}+e^{14})-x(e^{23}+e^{24}),\\-z(e^{14}+e^{13})-s(e^{24}+e^{23}),z(e^{13}+e^{14})+s(e^{23}+e^{24})),
\end{gather*}
for $x,y,z,s$ in $\R$;
\item[II.] $\eta\in\{u_0+\alpha u_1,u_1\mid\alpha\in\C\}$ and
\begin{gather*}
	\g=(x(e^{13}+e^{14})-2e^{24}-2e^{23},-x(e^{24}+e^{23})+2e^{13}+2e^{14},\\-s(e^{24}+e^{23})-z(e^{13}+e^{14})-4e^{12},s(e^{24}+e^{23})+z(e^{13}+e^{14})+4e^{12})
\end{gather*}
for $x,z,s$ in $\R$.
\end{enumerate}
\end{theorem}
\begin{proof}
Notice first that it suffices to prove the statement for the signature $(3,1)$. Indeed, if $\psi$ is a Killing spinor in signature $(1,3)$ with Killing number $\lambda$, then $\psi$ is a Killing spinor in signature $(3,1)$ with Killing number $i\lambda$.

Let $\nabla_X\eta=wX\cdot\eta$ for some $w\in\C$. Write $L_j=M+\mu_j+\xi_j$ as in  Lemma~\ref{lemma:unimod4}; since the proof of Lemma~\ref{lemma:unimod4}  does  not depend on $A$ being real, we see that $M$ and $\eta$ satisfy the hypothesis of  Lemma~\ref{lemma:31} and $(\mu_j+\xi_j)\cdot\eta=0$.

Write
\[\mu_j=\sum_t k_j^te_t, \quad \xi_j=h_j^1e_{234}+h_j^2e_{134}+h_j^3e_{124}+h_j^4e_{123}.\]
We apply Corollary~\ref{cor:cijkfromLiwhenDiracthreeform} and in particular
\begin{equation}\label{eqn:lemma413dim4Aid}
	de^j=2\bigl(3e^j\hook M-e^j\hook\xi_j\bigr)^\flat+2e^j\wedge\mu_j^\flat
\end{equation}
to each case appearing in Lemma~\ref{lemma:31}; notice that rescaling the metric has the effect of rescaling $M$.

Assume first that $\eta\in\Span{u_0,u_1}$. If $\eta=u_0+\alpha u_1$, a direct computation gives
\[
\begin{split}
	(\mu_j+\xi_j)\cdot\eta&=i(k^1_j-h^2_j)(u_1+\alpha u_0)+(k^2_j-h^1_j)(u_1-\alpha u_0)\\
	&\quad+(h^3_j+h^4_j)(u_2+\alpha u_3)-i(k^3_j+k^4_j)(u_2-\alpha u_3).
\end{split}
\]
Imposing $(\mu_j+\xi_j)\cdot\eta=0$, we have two cases. If $\alpha\ne0$ then $u_1+\alpha u_0$, $u_1-\alpha u_0$, $u_2-\alpha u_3$, $u_2+\alpha u_3$ are linearly independent on $\C$, it follows that
\begin{align*}
	k^1_j&=h^2_j&k^2_j&=h^1_j&k^3_j&=-k^4_j&h^3_j&=-h^4_j.
\end{align*}
If $\eta$ is chiral, carrying out the same computations and separating real and imaginary part for each equation we obtain the same result. Now recall that $\sum\mu_j=0=\sum\xi_j$ and \eqref{eqn:L_i_satisfy} implies $e_j\wedge M=-e_j\wedge\xi_j$, $e_j\hook\mu_j=0$. Thus, for $x,y,z,s\in\R$, up to $\Spin(3,1)$-action, we have the following cases.
\begin{compactenum}
	\item If $M=0$, then $w=0$ and $\eta=u_0+ru_1$, $x\geq0$ or $\eta=u_1$; furthermore,
	\[
	k_1^1=k_1^2=0\qquad k_2^2=k_2^1=0\qquad  k_3^3=h_3^3=0\qquad k_4^3=h_4^3=0.
	\]
	Setting \[
	z=k^1_3=-k^1_4\qquad s=k^2_3=-k^2_4\qquad x=k^3_1=-k^3_2\qquad y=h^3_1=-h^3_2
	\]
	and applying \eqref{eqn:lemma413dim4Aid} we get
	\begin{align*}
		de^1&=2(ye^{23}+ye^{24}+xe^{13}+xe^{14})&
		de^2&=2(ye^{13}+ye^{14}-xe^{23}-xe^{24})\\
		de^3&=-2(ze^{14}+se^{24}+ze^{13}+se^{23})&
		de^4&=2(ze^{13}+se^{23}+ze^{14}+se^{24}).
	\end{align*}
	The Jacobi identity is always satisfied, and the fact that $\eta$ is parallel follows from $w=0$.
	\item If $*M=e_3-e_4$, then $\eta=u_0+\alpha u_1$ for $\alpha\in\C$ or $\eta=u_1$, $w=0$ and
	\[
	k_1^1=k_1^2=0\qquad k_2^2=k_2^1=0\qquad k_3^3=k_4^3=0\qquad -h_3^3=h_4^3=1.
	\]
	In this case, setting
	\[
	z=k^1_3=-k^1_4\qquad s=k^2_3=-k^2_4\qquad x=k^3_1=-k^3_2,
	\]
	we obtain
	\begin{align*}
		de^1 &= 2  x {(e^{13}+e^{14})}-4  e^{24}-4  e^{23}\\
		de^2 &= -2  x {(e^{24}+e^{23})}+4  e^{13}+4  e^{14}\\
		de^3 &= -2  {(e^{24}+e^{23})} s-2  {(e^{13}+e^{14})} z-8  e^{12}\\
		de^4 &= 2  {(e^{24}+e^{23})} s+2  {(e^{13}+e^{14})} z+8  e^{12}
	\end{align*}
	The Jacobi identity is always satisfied, and again $\eta$ is parallel as $w=0$.

	\item If $*M=e_1$, then  $\eta=u_0\pm iu_1$, $w=\pm2i$ and
	\[
	k_1^1=0,\,k_1^2=-1\qquad k_2^2=k_2^1=0\qquad k_3^3=h_3^3=0\qquad k_4^3=h_4^3=0
	\]
	In this case, setting
	\[
	z=k^1_3=-k^1_4\qquad s=k^2_3=1-k^2_4\qquad x=k^3_1=-k^3_2\qquad y=h^3_1=-h^3_2,
	\]
	we obtain
	\[
	\begin{split}
		de^1 &= 2  y {(e^{24}+e^{23})}-2  e^{12}+2x  {(e^{13}+e^{14})} \\
		de^2 &= 2  y {(e^{13}+e^{14})}-6  e^{34}-2 x {(e^{24}+e^{23})} \\
		de^3 &= -2  s e^{23}-2  {(-3+s)} e^{24}-2  z {(e^{13}+e^{14})}\\
		de^4 &= 2  {(-1+s)} e^{24}+2  z {(e^{13}+e^{14})}+2  {(2+s)} e^{23}
	\end{split}
	\]
	However, $d^2=0$ is never satisfied for any choice of $x,y,z,s\in\R$.
\end{compactenum}

Now assume $\eta=u_0+\alpha u_2$. A similar computation gives $\mu_j=0=\xi_j$. Again, by~\eqref{eqn:L_i_satisfy} we have $e_j\wedge M=-e_j\wedge\xi_j$ for each $j=1,\dots,4$, thus $M=0$. It follows that $\g=\R^4$ and all spinors are parallel.
\end{proof}

Comparing Theorem~\ref{thm:killing31} with the classification of unimodular $4$-dimensional Lie algebras, we can determine which among the latter admit a Lorentzian metric with a parallel spinor; this will require a few algebraic lemmas which are collected in Appendix~\ref{appendix:liealgebras}.
\begin{table}[th]
\centering
\caption{\label{table:4dimKill} $4$-dimensional unimodular Lie algebras admitting a Lorentzian metric with a parallel spinor.}
\begin{tabular}{L L C L}
\toprule
\text{Lie algebra} &  \text{Parameters in Theorem~\ref{thm:killing31}}\\
\midrule
\R^4 &  I : x=y=s=z=0\\
\lie{r}\lie{h}_3   & I: x=y=0, z^2+s^2\neq0\\
\lie{r}\lie{r}_{3,-1}  & I: x^2+y^2\neq0\\
\lie{n}_4&    II: x^2-4=0\\
\lie{d}_4&   II: x^2-4>0\\
\lie{d}'_{4,0}  & II: x^2-4<0\\
\bottomrule
\end{tabular}
\end{table}

\begin{corollary}
\label{cor:4dimParallel}
The $4$-dimensional unimodular Lie algebras admitting a Lorentzian metric with a parallel spinor are:
\begin{itemize}
 \item the nilpotent Lie algebras $\R^4,\lie r\lie h_3, \lie n_4$ endowed with a flat metric;
 \item the non-nilpotent Lie algebras $\lie{r}\lie{r}_{3,-1}$, $\lie{d}_4$, $\lie{d}'_{4,0}$, with a non-Ricci-flat metric of
  nilpotent holonomy $\Hom(\R^2,\R)$.
\end{itemize}
The correspondence with the families of Theorem~\ref{thm:killing31} is given in Table~\ref{table:4dimKill}.
\end{corollary}
\begin{proof}
With reference to Theorem~\ref{thm:killing31}, the spinors $u_0$ and $u_1$ are parallel, hence the holonomy is contained in the nilpotent group with Lie algebra
\[\Span{e_1\cdot (e_3-e_4),e_2\cdot (e_3-e_4)}\]
(see the proof of Lemma~\ref{lemma:orbitsPsiM}). The exact holonomy can be determined using the Ambrose-Singer theorem. Indeed, a straightforward computation shows that the metrics in family I are flat when $x^2+y^2=0$, and those in family II when $x^2-4=0$, whilst in the other cases
\[\Span {R(e_j,e_k)}=\Span{e_1\cdot (e_3-e_4),e_2\cdot (e_3-e_4)}\]
and the Ricci tensor is not zero. The statement is then obtained by applying respectively Lemma~\ref{lemma:parallelFamilywithAbelian} and Lemma~\ref{lemma:parallelFamilyNoAbelian} to the first and second families of Lie algebras of Theorem~\ref{thm:killing31}.
\end{proof}

\begin{remark}
The non-flat metrics appearing in Corollary~\ref{cor:4dimParallel} belong to the class of indecomposable, reducible Lorentzian manifolds with a parallel spinor with light-like Dirac current (see \cite[Corollary~3.2]{leistner2001}); for the explicit computation of the Dirac current, we refer to Appendix~\ref{appendix:diraccurrent}. In general, the Dirac current of a parallel spinor is a parallel vector field; Lorentzian metrics admitting a parallel light-like vector field are known as \emph{Brinkmann waves}.
\end{remark}

\begin{corollary}
\label{cor:4dimHarmfulAId}
The $4$-dimensional unimodular Lie algebras admitting a Lorentzian metric with a  harmful structure $(\phi,\psi,A,\lambda)$ with $A$ a multiple of the identity are $\R^4,\lie r\lie h_3, \lie{r}\lie{r}_{3,-1}, \lie n_4, \lie{d}_4, \lie{d}'_{4,0}$; the list is summarized in the third column of Table~\ref{table:4dimHarm}. Each such harmful structure satisfies $A=a\id$, $\lambda= \pm\frac{i}2a$, and $\nabla\psi=0$; for each Lie algebra in the list, every value of $a\in\R$ is attained.
\end{corollary}
\begin{proof}
Let $\g$ be a $4$-dimensional unimodular Lie algebras admitting a Lorentzian metric with a  harmful structure $(\phi,\psi,A,\lambda)$ with $A=a\id$; since Killing spinors are parallel, Proposition~\ref{prop:AmultipleIdentity} implies that $a^2/4+\lambda^2=0$ and
\[\eta = \lambda\psi - \frac a2\phi, \quad \xi=-\lambda\psi + \frac a2\phi\]
are both parallel. This implies that $\phi,\psi$ are also parallel.

Conversely, if $\g$ is a Lorentzian Lie algebra with a parallel spinor $\eta$, set $\xi= h\omega\cdot\eta$, $h=\pm i$ in such a way that $\eta+\xi$ is non-zero. For any $a\in\R$, Proposition~\ref{prop:AmultipleIdentity} shows that
\[\psi = \frac a 2(\eta+\xi), \quad \phi = \lambda (\eta+\xi),\quad A=a\id, \quad \lambda=\frac{h}2a\]
determines a harmful structure.
\end{proof}

\section{Dimension $4$: Riemannian signature}
\label{sec:40}
In this section we complete the classification of definite harmful structures on unimodular Lie algebras $\g$ of dimension $4$. We concentrate on positive definite signature, for the following reason. The case where $A$ is a multiple of the identity is covered by Proposition~\ref{prop:definiteAmultipleodentity}. When $A$  is not a multiple of the identity, Lemma~\ref{lemma:gamma1nonzero} implies that there exists $\xi\in\Lambda^3\g$ such that $\xi\cdot\psi=\psi$; $\xi$ is unique by Lemma~\ref{lemma:4OnlyTrivialSolution}, and \eqref{eqn:nonlethal4} gives
\begin{equation}
 \label{eqn:gxixi1}
g(\xi,\xi)=1,
\end{equation}
which rules out the case of negative definite signature.

We will distinguish two cases, according to whether $\lambda$ is zero or not.
\begin{proposition}
\label{prop:40lambdazero}
Let $\g$ be a four-dimensional unimodular Lie algebra with  a real harmful structure $(\phi,\psi,A,\lambda)$ of signature $(4,0)$ with $\lambda=0$ and $A$ not a multiple of the identity. Then there is a basis such that the metric takes the form~\eqref{eqn:standard31metric},
\[
de^1=(3y -x_1)e^{23},\quad de^2=-(3y -x_2)e^{13},\quad
de^3=(2y+x_1+x_2)e^{12},\quad  de^4=0\]
and $A$, $\psi$ satisfy
\[A=\diag(x_1,x_2,y-x_1-x_2,-y)+y\id, \quad \psi=u_0+u_2+z(u_1+u_3),\]
where $x_1,x_2,y$ are real parameters, not all of them zero, and $z\in\C$.
\end{proposition}
\begin{proof}
By Remark~\ref{rmk:Lambda0Dipendenti}, we have
\[{*}M=(\frac18\Tr A){*}\xi.\]
If $v$ is a unit eigenvector of $A_0$ with eigenvalue $\alpha$, we can complete $v$ to an orthonormal basis of eigenvectors and apply Lemma~\ref{lemma:unimod4} to obtain
\[0=v\wedge (M+\frac12\alpha\xi)=v\wedge (\frac18\Tr A+\frac12\alpha)\xi.\]
It follows that unless $\alpha=-\frac14\Tr A$, $v$ is orthogonal to ${*}\xi$, i.e.
\[{*}\xi\in \left(\bigoplus_{\alpha\neq -\frac14\Tr A}\g_\alpha\right)^\perp;\]
therefore, $*\xi$ is an eigenvector of $A_0$ with eigenvalue $-y$, where $y=\frac14\Tr A$.

By~\eqref{eqn:gxixi1}, we can choose an orthonormal basis of eigenvectors of $A$ such that $*\xi=e_4$ and
\[A_0=x_1e^1\otimes e_1+x_2e^2\otimes e_2+(y-x_1-x_2)e^3\otimes e_3-ye^4\otimes e_4.\]
Then  $M=\frac12ye_{123}$, $\xi=e_{123}$, and Proposition~\ref{prop:LieAlgebraFromMXiLi} determines the structure constants.

The element $e_1\cdot e_3$ preserves $A_0$ and acts on spinors by
\[u_0\mapsto u_3, \quad u_3\mapsto -u_0.\]
Thus we can assume up to $\Spin(4,0)$ that $\psi_+=u_0+zu_3$, with $z$ a complex parameter; using~\eqref{eqn:nonlethal4}, we obtain $\psi_-=-{*}\xi\cdot\psi_+=u_2+zu_1$.

Notice that when each of $x_1,x_2,y$ is zero $A_0=0$, against the hypothesis.
\end{proof}

We now concentrate on the case where $\lambda$ is non-zero. By Lemma~\ref{lemma:gamma1nonzero}, $M$ is nonzero. Regardless of the signature, it is clear from~\eqref{eqn:L_i_in_terms_of_cijk} that rescaling the metric we can assume $M$ to satisfy
\begin{equation}
\label{eqn:normalizeM}
g(*M,{*}M)=\epsilon=\pm1;
\end{equation}
with this normalization,   Lemma~\ref{lemma:unimod4} gives
\begin{equation}
\label{eqn:lambdasquaredplustraceAsquared}
\epsilon = (-1)^q(\lambda^2+\frac1{64}(\Tr A)^2), \qquad \psi_-= -\epsilon i^{q} {*}M\cdot (i\lambda+\frac18\Tr A)\psi_+.
\end{equation}
In the present situation, $\epsilon=1$ and $q=0$.

In the next lemma, we show that $\lambda$ is necessarily real.
\begin{lemma}
 \label{lemma:harmfuldefinitelambdareal}
Let $\g$ be a four-dimensional unimodular Lie algebra with a positive definite real harmful structure $(\phi,\psi,A,\lambda)$ satisfying $g(*M,{*}M)=1$. Then $\lambda$ is real.
\end{lemma}
\begin{proof}
Suppose that $\lambda = it$ with $t\in\R$. We must show that $t$ is zero. Since the signature is definite, $A$ is diagonalizable. Lemma~\ref{lemma:alternative} and~\eqref{eqn:lambdasquaredplustraceAsquared} give
\begin{multline*}
0=(\nabla_X {*}M+tA_0(X))\cdot\psi_+ -g\left(A_0(X)+\frac14(\Tr A)X+2t X,{*}M\right) {*}M\cdot (t-\frac18\Tr A)\psi_+\\
=(\nabla_X {*}M+tA_0(X))\cdot\psi_+ -(t-\frac18\Tr A) g\left(A_0(X),{*}M\right)  {*}M\cdot \psi_+
+2g(X,{*}M){*}M\cdot\psi_+.
\end{multline*}
Because the metric is definite,  $V_{\psi_+}$ contains no real vectors, so
\begin{equation}
\label{eqn:nablaXstarMdefiniteimaginary}
\nabla_X {*}M+t A_0(X) + (-t+\frac18\Tr A) g\left(A_0(X),{*}M\right)  {*}M+2g(X,{*}M){*}M=0.
\end{equation}
Taking the inner product with $*M$ gives
\[\begin{split}
0&=tg(A_0(X),{*}M)+(-t+\frac18\Tr A) g\left(A_0(X),{*}M\right)  +2 g(X,*M)\\
&=g(X,\frac18(\Tr A) A_0(*M)+2 {*}M);
\end{split}\]
since $X$ is arbitrary, we obtain
\begin{equation}
 \label{eqn:A0starMdefiniteimaginary}
\frac18(\Tr A) A_0(*M)+2 {*}M=0.
\end{equation}
Substituting in \eqref{eqn:nablaXstarMdefiniteimaginary}, we get
\begin{multline*}0=\nabla_X {*}M+t A_0(X) - t g\left(A_0(X),{*}M\right)  {*}M
+\frac18 (\Tr A)g(X,A_0(*M)){*}M+2g(X,{*}M){*}M\\
=\nabla_X {*}M+t (A_0(X) -  g\left(A_0(X),{*}M\right)  {*}M)-2g(X,{*}M) {*}M+2g(X,{*}M){*}M\\
=\nabla_X {*}M+t (A_0(X) -  g\left(A_0(X),{*}M\right)  {*}M).
\end{multline*}

It follows from \eqref{eqn:A0starMdefiniteimaginary} that $A_0\neq0$, i.e. $A$ is not a multiple of the identity; let $\xi\in\Lambda^3\g$ such that $\xi\cdot\psi=\psi$. Lemma~\ref{lemma:unimod4} and \eqref{eqn:lambdasquaredplustraceAsquared}  imply
\[*\xi\cdot\psi_+ =-\psi_-= {*M}(-t+\frac18\Tr A)\psi_+,\]
hence
\[*\xi= *M(-t+\frac18\Tr A).\]
Comparing the norms of $*\xi$ and $*M$ we see that
\[\frac18\Tr A=t\pm 1;\]
using~\eqref{eqn:lambdasquaredplustraceAsquared} we conclude that $1=\pm 2t +1$, i.e. $t=0$.
\end{proof}

\begin{lemma}
\label{lemma:fourdimdefinite}
Let $\g$ be a four-dimensional unimodular Lie algebra with  a real harmful structure $(\phi,\psi,A,\lambda)$ of positive definite signature, with $\lambda\neq0$ and $g(*M,{*}M)=1$. Then, up to $\Spin(4,0)$ action,
\begin{equation}\label{eqn:A0psiTrAlambda}\begin{gathered}
	A=2\diag\left(-\frac{\cos\theta}{\sin(\alpha+\theta)},\frac{\sin\theta}{\cos(\alpha+\theta)},\frac{\cos(\alpha+2\theta)}{\sin 2(\alpha+\theta)},\frac{\cos(\alpha+2\theta)}{\sin 2(\alpha+\theta)}\right)+2\sin\alpha\id,\\
	\psi=u_0+e^{-i(\alpha+\theta)}u_1, \quad \lambda=\cos\alpha, \quad *M=\cos\theta e_1+\sin\theta e_2,
\end{gathered}\end{equation}
where $\theta\in(-\pi/2,\pi/2)$ and $\alpha+\theta$ is not an integer multiple of $\pi/2$.
\end{lemma}
\begin{proof}
By Lemma~\ref{lemma:harmfuldefinitelambdareal}, $\lambda$ is a real number and $A_0$ is not zero. Let $\xi\in\Lambda^3\g$ satisfy $\xi\cdot\psi=\psi$. By~\eqref{eqn:lambdasquaredplustraceAsquared},
\[*\xi -(i\lambda+\frac18\Tr A){*}M\in V_{\psi_+}.\]
Since the metric is definite,
$V_{\psi_+}$ is the space of vectors of type $(0,1)$ for some almost complex $J$; for instance, if $\psi_+=u_0$,  $V_{\psi_+}=\Span{e_1-ie_2,e_3-ie_4}$ and $Je_1=e_2, Je_3=e_4$.

A vector $v+iw$ with $v,w$ real is in $V_{\psi_+}$ if $v+Jw=0$. We obtain
\[*\xi -\frac18(\Tr A){*}M - \lambda J{*}M=0.\]

Since  $\lambda\neq0$,  Proposition~\ref{prop:nablastarMskew} implies that
\[A_0=A_W+ A_\perp, \quad A_W\colon W\to W, \quad A_\perp\colon W^\perp\to W^\perp,\]
where $A_\perp$ is a multiple of the identity.

Moreover,  Proposition~\ref{prop:A0*xigeneral} gives
 \begin{equation}
  \label{eqn:fourdimdefinite:A0starm}
  -2{*}M=A_0(*\xi)=\frac18(\Tr A)A_0(*M)+\lambda A_0(J(*M)).
 \end{equation}
Up to the action of $\Spin(4)$, we can assume that $\psi_+=u_0$, $*M\in\Span{e_1,e_2}$. We can further act by the $U(1)$ that fixes $u_0$ and $\Span{e_1,e_2}$ so that $A_0$ takes the diagonal form
\begin{equation*}
 \label{eqn:fourdimdefinite:A0}
A_0=a e^1\otimes e^1+be^2\otimes e^2 -\frac12(a+b)(e^3\otimes e^3+e^4\otimes e^4)
\end{equation*}
and $g(e_1,{*}M)>0$. Choose angles $\theta,\alpha$ such that
\[*M=\cos\theta e_1+\sin\theta e_2, \quad J{*}M=-\sin\theta e_1+\cos\theta e_2, \quad \lambda=\cos\alpha, \quad \frac18\Tr A=\sin\alpha;\]
by construction, $\cos\theta>0$.

Then~\eqref{eqn:fourdimdefinite:A0starm} gives
\begin{multline}
\label{eqn:fourdimdefinite:sincos}
0=2\cos\theta e_1+2\sin\theta e_2+\frac18(\Tr A)(a\cos\theta e_1+b\sin\theta e_2)+\lambda (-a\sin\theta e_1+b\cos\theta e_2)\\
=\bigl(2\cos\theta +a\sin\alpha\cos\theta -a\cos\alpha\sin\theta\bigr)e_1+
\bigl(2\sin\theta + b\sin\alpha\sin\theta+b\cos\alpha\cos\theta\bigr)e_2.
\end{multline}
Since $\cos\theta$ is non-zero, \eqref{eqn:fourdimdefinite:sincos} implies that $\sin(\alpha-\theta)$ is also non-zero.

We claim that $\cos(\alpha-\theta)$ is non-zero. Indeed, supposing otherwise, \eqref{eqn:fourdimdefinite:sincos} gives $\sin\theta=0$, so $\theta=0$, which gives $\cos\alpha=0$, against the hypothesis. Thus, \eqref{eqn:fourdimdefinite:sincos} gives
\[a=-\frac{2\cos\theta}{\sin(\alpha-\theta)},\quad b=-\frac{2\sin\theta}{\cos(\alpha-\theta)}\]
so that
\[a+b= -\frac{4\cos(\alpha-2\theta)}{\sin(2(\alpha-\theta)}.\]
Now observe that
\[\psi_-=-(i\cos\alpha + \sin\alpha){*M}\cdot u_0 = -e^{-i\alpha}i (i\cos\theta + \sin\theta)u_1 = e^{-i(\alpha+\theta)}u_1.\qedhere\]
\end{proof}

Recall that we can change the sign of $A$ and $\lambda$ by the symmetry of Proposition~\ref{prop:harmfulpropertiesMeven}. We have
\begin{proposition}
\label{prop:fourdimdefinitecijk}
Up to isometric isomorphism, rescaling of the metric, rescaling of the spinors, and the symmetry of Proposition~\ref{prop:harmfulpropertiesMeven}, the four-dimensional, unimodular Lie algebras with a real harmful structure $(\phi,\psi,A,\lambda)$ with positive definite metric and $\lambda\neq0$ are given by
\begin{gather*}
 de^1=(6\sin\theta -2\cos\theta\cot(\alpha+\theta))e^{34} \\
de^2=(-6\cos\theta + 2\sin\theta\tan(\alpha+\theta))e^{34}\\
de^3=(5\cos\theta +\sin\theta\tan(\alpha+\theta))e^{24}-(5\sin\theta+\cos\theta\cot(\alpha+\theta))  e^{14}\\
de^4=-(5\cos\theta+\sin\theta\tan(\alpha+\theta)) e^{23}+(5\sin\theta+\cos\theta\cot(\alpha+\theta)) e^{13}
\end{gather*}
where $\alpha,\theta\in(-\pi/2,\pi/2)$, $\alpha+\theta\notin\{ 0, \pm \pi/2\}$, the metric takes the form~\eqref{eqn:generaldiagonalmetric}, and $A$, $\psi$ and $\lambda$ are given by~\eqref{eqn:A0psiTrAlambda}.
\end{proposition}
\begin{proof}
Up to the symmetry of Proposition~\ref{prop:harmfulpropertiesMeven}, we can assume that $\lambda>0$.

We have
\[e_{134}\cdot u_0=-u_1, \quad e_{134}\cdot u_1=-u_0, \quad e_{234}\cdot u_0=iu_1, \quad e_{234}\cdot u_1=-iu_0.\]
By Lemma~\ref{lemma:fourdimdefinite},  for $\psi=u_0+e^{-i(\alpha+\theta)}u_1$, we have that $\xi\cdot \psi=\psi$ is satisfied by
\[\xi=-\cos(\alpha+\theta)e_{134}-\sin(\alpha+\theta)e_{234}.\]
We also have $M=-\cos\theta e_{234}+\sin\theta e_{134}$. The statement is obtained applying  Proposition~\ref{prop:LieAlgebraFromMXiLi}  with
\[x_1=-\frac{2\cos\theta}{\sin(\alpha+\theta)}, \quad x_2=\frac{2\sin\theta}{\cos(\alpha+\theta)}, \quad x_3=x_4=\frac{2\cos(\alpha+2\theta)}{\sin 2(\alpha+\theta)}.\qedhere\]
\end{proof}
As in the case of parallel spinors, we can use the known classification of $4$-dimensional Lie algebras and Appendix~\ref{appendix:liealgebras} to determine which isomorphism classes of Lie algebras appear in the previous proposition.
\begin{corollary}
\label{cor:4dimHarmful}
The $4$-dimensional unimodular Lie algebras admitting a positive definite metric with a real harmful structure $(\phi,\psi,A,\lambda)$, $A$ not a multiple of the identity are:
\begin{itemize}
 \item  $\lie r\lie h_3$, $\lie{r}\lie{r}_{3,-1},\lie{r}\lie{r}'_{3,0}$, $\su(2)\times\R$ and $\Sl(2,\R)\times\R$ if $\lambda=0$;
 \item  $\su(2)\times\R$, $\Sl(2,\R)\times\R$ and $\lie{d}'_{4,0}$ if $\lambda\ne0$.
\end{itemize}
The two lists are summarized in Table~\ref{table:4dimHarm}.
\end{corollary}
\begin{proof}
If $\lambda=0$ we observe that the Lie algebras appearing in Proposition~\ref{prop:40lambdazero} are reducible; noticing that the only constraint on the parameters $x_1,x_2,y$ is that they do not vanish simultaneously, we obtain the first part of the statement by setting
\begin{gather*}
	-a=3y-x_1,\qquad -b=3y-x_2, \qquad -c=2y+x_1+x_2
\end{gather*}
and applying Lemma~\ref{lem:ClassificazioneBaseDim3}.

If $\lambda\ne0$, we want to apply Lemma~\ref{lemma:algebreHarmful40lamdanonzero} to the Lie algebras appearing in Proposition~\ref{prop:fourdimdefinitecijk}. In order to do so we set
\begin{gather*}
a=6\sin\theta -2\cos\theta\cot(\alpha+\theta),\quad b=-6\cos\theta + 2\sin\theta\tan(\alpha+\theta),\\
c=5\cos\theta +\sin\theta\tan(\alpha+\theta),\quad d=-(5\sin\theta+\cos\theta\cot(\alpha+\theta)).
\end{gather*}
We have $c^2+d^2\neq0$;  indeed, if $c$ and $d$ vanish,
\[\cos\theta =-\frac15\sin\theta\tan(\alpha+\theta), \quad -5\sin\theta +\frac15\sin\theta=0,\]
hence $\sin\theta=0=\cos\theta$, which is absurd. A straightforward computation shows that
\[ad+bc=4 (4 \cos (4 (\alpha +\theta ))+\cos (2 \alpha )-3) \csc ^2(2 (\alpha +\theta ))\]
can be positive, negative or zero. However, $ad+bc=0$ and $ac-bd$ cannot vanish simultaneously, for otherwise
\[
\begin{cases}
	16 \cos (\alpha ) \cos (\alpha +2 \theta ) \csc (\alpha +\theta ) \sec (\alpha +\theta )=0\\
	4 \csc ^2(2 (\alpha +\theta ))\bigl(4 \cos (4 (\alpha +\theta ))+\cos (2 \alpha )-3\bigr) =0,
\end{cases}\]
which gives
\[
\begin{cases}
	\cos (\alpha ) \cos (\alpha +2 \theta )=0\\
	\cos (4 (\alpha +\theta ))+\cos (2 \alpha ) -3=0.
\end{cases}
\]
Since  $\cos(\alpha)\neq0$, the first equation has solution $\alpha=-2\theta+\frac{\pi }{2}+k\pi$ for $k\in\Z$. Then the second equation gives $\sin{2\theta}=0$, i.e. $\theta=0$ and $\alpha=\frac\pi2+k\pi$, which is absurd.

Thus, the first three cases of Lemma~\ref{lemma:algebreHarmful40lamdanonzero} occur, and we obtain the second part of the statement.
\end{proof}

\section{Dimension $4$: Lorentzian signature}\label{sec:31}
In this section we study harmful structures of signature $(3,1)$ with $A$ diagonalizable; we will assume that $A$ is not a multiple of the identity, since this particular case has been handled in Section~\ref{sec:Amultipleofidentity}. Compared to the Riemannian case, there is an extra situation to consider, corresponding to $M$ being isotropic but non-zero.

First, we study the case where $\lambda=0$. In analogy with Proposition~\ref{prop:40lambdazero}, and with the same proof, we obtain:
\begin{proposition}
\label{prop:31lambdazero}
Let $\g$ be a four-dimensional unimodular Lie algebra with  a real harmful structure $(\phi,\psi,A,\lambda)$ of signature $(3,1)$ with $\lambda=0$, $A$ diagonalizable but not a multiple of the identity. Then there is a basis such that the metric takes the form~\eqref{eqn:standard31metric},
\[
de^1=(3y -x_1)e^{23},\quad de^2=-(3y -x_2)e^{13},\quad
de^3=(2y+x_1+x_2)e^{12},\quad  de^4=0\]
and $A$, $\psi$ satisfy
\[A=\diag(x_1,x_2,y-x_1-x_2,-y)+y\id, \quad \psi=u_0+u_2+z(u_1+u_3),\]
where $x_1,x_2,y$ are real parameters, not all of them zero, and $z\in\C$.
\end{proposition}

In the rest of this section, we will assume $\lambda\neq0$. By Remark~\ref{remark:ScalarProductXiM}, this implies that $W=\Span{*\xi,{*}M}$ defines a nondegenerate two-plane; in particular, $M$ is non-zero. We will consider two subcases, according to whether $M$ is isotropic or not.
\begin{proposition}
\label{prop:31Mlightlike}
Let $\g$ be a four-dimensional unimodular Lie algebra with  a real harmful structure $(\phi,\psi,A,\lambda)$ of signature $(3,1)$ with $M$ light-like, $\lambda\neq0$, $A$ diagonalizable but not a multiple of the identity. Up to rescaling the metric and the symmetry of Proposition~\ref{prop:harmfulpropertiesMeven},
there exist a real $x>0$, $x\neq1$ and an orthonormal basis of $\g$ such that the metric takes the form~\eqref{eqn:standard31metric}, the Lie algebra is given by
\begin{gather*}
	de^1=-\frac{6x+4x^3}{1+x^2}e^{24}+\frac{6x-4x^3}{1-x^2}e^{23},\qquad de^3= \frac{8x^3-4x}{x^2-1}e^{12},\\
	de^2 =\frac{6x+4x^3}{1+x^2}e^{14}-\frac{6x-4x^3}{1-x^2}e^{13},\qquad de^4= \frac{8x^3+4x}{x^2+1}e^{12},
\end{gather*}
we have
\[
A=4x^2\diag\left(\frac{x^2}{x^{4}-1},\frac{x^2}{x^{4}-1},-\frac1{x^2-1},-\frac1{x^2+1}\right)+2\id,
\]
and either $\lambda=i$, $\psi=u_0 +xu_2$, or $\lambda=-i$, $\psi=u_1+xu_3$.
\end{proposition}
\begin{proof}
The fact that $*M$ is isotropic implies that  $\lambda$ is imaginary by Lemma~\ref{lemma:unimod4}.

By Proposition~\ref{prop:nablastarMskew}, we have an $A_0$-invariant orthogonal decomposition $\g=W\oplus W^\perp$, with $W=\Span{{*}M,{*}\xi}$. Up to rescaling the metric, we will assume that
\begin{equation}
 \label{eqn:31Mlightlike:norms}
 \Tr A=8, \quad g(*M,{*}\xi)=-1=g(*\xi,{*}\xi).
\end{equation}

Choose an orthonormal basis of eigenvectors $e_1,\dotsc,e_4$ with  $\Span{e_3,e_4}=W$, so that the metric takes the form
\[e^1\otimes e^1+e^2\otimes e^2+e^3\otimes e^3-e^4\otimes e^4.\]
Up to a reflection around the origin in the planes $\Span{e_1,e_3}$, $\Span{e_2,e_4}$ we can assume that $*M$ is in $\Span{e_3+e_4}$ and $g(*\xi,e_4)<0$. Using~\eqref{eqn:31Mlightlike:norms}, we can write
\[*M = (\exp \theta)(e_3+e_4), \quad *\xi=\sinh\theta e_3 + \cosh\theta e_4, \quad\theta\in\R,\]
or equivalently
\[*M = x(e_3+e_4), \quad *\xi=\frac12(x-x^{-1}) e_3 + \frac12(x+x^{-1}) e_4,\quad x>0.\]

By Proposition~\ref{prop:A0*xigeneral},  $A_0(*\xi)=-2{*}M$, so $x\neq1$ and
 \[A_0=\frac{4x^{4}}{x^{4}-1}(e^1\otimes e_1+e^2\otimes e_2)-\frac{4x^2}{x^2-1}e^3\otimes e_3-\frac{4x^2}{x^2+1}e^4\otimes e_4.\]
Notice that
 \[e_3\wedge (M-\frac12 \frac{4x^2}{x^2-1}\xi) = 0 = e_4\wedge (M-\frac12 \frac{4x^2}{x^2+1}\xi).\]
In order to solve~\eqref{eqn:Mpsiplus}, observe that we have
\[*M\cdot u_0 = -2ixu_2, \quad *M\cdot u_3=0,\]
\[*\xi\cdot u_0=-ixu_2,\quad *\xi\cdot u_3=-ix^{-1}u_1\]
Observe that  either $\lambda=i$ or $\lambda=-i$. If $\lambda=i$,
\eqref{eqn:Mpsiplus} gives
\[*M\cdot \psi_+ =-2i\psi_-, \quad *\xi\cdot\psi_+=-i\psi_-,\]
so $\psi_+$ has no component along $u_3$, and up to rescaling the spinor we have
\[\psi=u_0 +xu_2.\]
Similarly, if $\lambda=-i$,
\[*M\cdot \psi_+ =0, \quad *M\cdot\psi_-=2i\psi_+, \quad *\xi\cdot\psi_+=-i\psi_-,\quad *\xi\cdot\psi_-=i\psi_+,\]
so $\psi_+$ is a multiple of $u_3$, and we can assume
\[\psi=u_3 +x^{-1}u_1.\]
Applying Proposition~\ref{prop:LieAlgebraFromMXiLi} we obtain the statement.
\end{proof}

\begin{proposition}
\label{prop:31FourDimcijknew}
Let $\g$ be a four-dimensional unimodular Lie algebra with  a real  harmful structure $(\phi,\psi,A,\lambda)$, with $A$ diagonalizable,  signature $(3,1)$, $*M$ non-isotropic and $\lambda\neq0$. Then up rescaling of the metric, rescaling of the spinors, and the symmetry that reverses the sign of $A$, there is an orthonormal basis of $\g$ such that the metric takes the form~\eqref{eqn:standard31metric}, the structure constants are given by
\begin{gather*}
	de^1=-\frac1x\left(3(x^2+\epsilon)-\frac{x^2y^2+\epsilon}{y^2+1}\right)e^{24}+\frac1x\left(3(x^2-\epsilon)-\frac{x^2y^2+\epsilon}{y^2-1}\right)e^{23}\\
	de^2=\frac1x\left(3(x^2+\epsilon)-\frac{x^2y^2+\epsilon}{y^2+1}\right)e^{14}-\frac1x\left(3(x^2-\epsilon)-\frac{x^2y^2+\epsilon}{y^2-1}\right)e^{13}\\
	de^3=\frac1x\left(3(x^2-\epsilon)+(x^2+\epsilon)\frac{y^2+1}{y^2-1}\right)e^{12}\\
	de^4=\frac1x\left(3(x^2+\epsilon)+(x^2-\epsilon)\frac{y^2-1}{y^2+1}\right)e^{12},
\end{gather*}
where $x,y>0$, $y\neq1$, $\epsilon=\pm1$, with $y\neq x$ when $\epsilon=-1$, and the harmful structure satisfies
\begin{gather*}
A=\frac{2\epsilon y}x\diag\left(-\frac{x^2y^2+\epsilon}{y^4-1},-\frac{x^2y^2+\epsilon}{y^4-1},\frac{x^2+\epsilon}{y^2-1},\frac{x^2-\epsilon}{y^2+1}\right)+\frac{y^2-\epsilon x^2}{xy}\id,\\
\psi=u_0 -\epsilon y u_2, \quad \lambda = -\frac i{2xy}(y^2+\epsilon x^2).
\end{gather*}
\end{proposition}
\begin{proof}
By Lemma~\ref{lemma:orbitsPsiM}, we may assume that up to $\Spin(3,1)$ action and rescaling the metric,  $\psi_+=u_0$ and $*M$ is one of $e_1, e_3,e_4$. Comparing~\eqref{eqn:nonlethal4} with~\eqref{eqn:lambdasquaredplustraceAsquared}, we see that
\begin{equation}
 \label{eqn:starxistarMpsi+}
 i{*}\xi\cdot u_0=\psi_- = - i\epsilon (i\lambda+\frac{1}{8}\Tr A) {*}M\cdot u_0,
\end{equation}
with $\epsilon=\pm1$ as in~\eqref{eqn:normalizeM}. This shows that the case $*M=e_1$ does not occur, as otherwise $\psi_-\in\Span{u_1}$, and $*\xi\in\Span{e_1,e_2,e_3-e_4}$, which contradicts~\eqref{eqn:nonlethal4}. Moreover, if $*M$ equals either $e_3$ or $e_4$, we see that $\psi_-=
\epsilon (i\lambda+\frac{1}{8}\Tr A) u_2$, implying that $*\xi\in\Span{e_3,e_4}$ and $\lambda$ is imaginary.

By Proposition~\ref{prop:nablastarMskew}, $A_0$ split as a sum of two symmetric diagonalizable operators $A_\perp\colon W^{\perp}\to W^{\perp}$ and $A_W\colon W\to W$, where $A_\perp$ is a multiple of the identity on $W^{\perp}$. Relative to the basis introduced above, we have $W=\Span{e_3,e_4}$ in the basis chosen above, implying that $e_1$ and $e_2$ are eigenvectors. We can modify the basis acting by hyperbolic rotations in the plane $W$ to obtain that $e_3$ and $e_4$ are also eigenvectors,
and $A_0$ takes the diagonal form
\begin{equation} \label{eqn:31:equationA0eigenvectors}
A_0=-\frac12(a+b)(e^1\otimes e^1+e^2\otimes e^2) + ae^3\otimes e^3+be^4\otimes e^4,
\end{equation}
This change of basis preserves the condition $\psi_+=u_0$, though it destroys the condition that $*M$ equals $e_3$ or $e_4$; up to reflection around the origin in $W$, we can assume that $*M$ has the form $\cosh \theta e_3 +\sinh\theta e_4$ or $\sinh\theta e_3+\cosh\theta e_4$ according to the sign of $\epsilon$, or equivalently
\[*M=\frac12(x+\epsilon x^{-1})e_3+\frac12(x-\epsilon x^{-1})e_4, \quad x>0.\]
Similarly, the real numbers $i\lambda$ and $\frac18\Tr A$,
are related by~\eqref{eqn:lambdasquaredplustraceAsquared} and can be simultaneously changed of sign by the symmetry of Proposition~\ref{prop:harmfulpropertiesMeven};  using again hyperbolic functions and changing the variable, we obtain
\begin{equation}
 \label{eqn:ilambdaandtrAfromy}
 i\lambda = \frac12(t+\epsilon t^{-1}), \quad \frac18\Tr A =\frac12(t-\epsilon t^{-1}), \quad t>0.
\end{equation}
Equation~\eqref{eqn:starxistarMpsi+} gives
\[*\xi\cdot u_0=- \epsilon (i\lambda+\frac{1}{8}\Tr A) \bigl(\frac12(x+\epsilon x^{-1})e_3+\frac12(x-\epsilon x^{-1})e_4\bigr)\cdot u_0=
 \epsilon ixtu_2;\]
setting $y=xt$, we obtain
 \[*\xi=pe_3 + qe_4, \quad p+q = -\epsilon y,\]
where $p^2-q^2=-1$ by~\eqref{eqn:nonlethal4}; thus,
\[*\xi = \frac 12\epsilon(-y+y^{-1}) e_3 - \frac 12\epsilon(y+y^{-1})e_4,\]
so that $x$ and $y$ play symmetric roles as the hyperbolic parameters of ${*}M$ and ${*}\xi$ in $W$.
Due to Proposition~\ref{prop:A0*xigeneral}, $A_0(*\xi)=-2{*M}$, which using~\eqref{eqn:31:equationA0eigenvectors} gives
\[ \frac 12a\epsilon(-y+y^{-1}) e_3 - \frac 12b\epsilon(y+y^{-1})e_4=
-(x+\epsilon x^{-1})e_3-(x-\epsilon x^{-1})e_4.\]
This equation implies in particular that $y\neq1$: otherwise, $x+\epsilon x^{-1}=0$, giving $\epsilon=-1$, $x=1$, $t=1$, which  gives $\lambda=0$ by~\eqref{eqn:ilambdaandtrAfromy}, against the assumptions. Therefore, we obtain
\[a=\frac{2\epsilon(x+\epsilon x^{-1})}{y-y^{-1}}, \quad b=\frac{2\epsilon(x-\epsilon x^{-1})}{y+y^{-1}},
\quad -\frac12(a+b)=-\frac{2\epsilon xy + 2 (xy)^{-1}}{y^2-y^{-2}}
.\]
The statement now follows from Proposition~\ref{prop:LieAlgebraFromMXiLi}.
\end{proof}

\begin{corollary}
\label{cor:31dimHarmful}
The $4$-dimensional unimodular Lie algebras admitting a Lorentzian metric of signature $(3,1)$ with a harmful structure, $A$ diagonalizable but not a multiple of the identity  are:
\begin{itemize}
 \item  $\lie r\lie h_3$, $\lie{r}\lie{r}_{3,-1},\lie{r}\lie{r}'_{3,0}$, $\su(2)\times\R$ or $\Sl(2,\R)\times\R$, for $\lambda=0$;
 \item $\Sl(2,\R)\times\R$, for $\lambda\neq0$ and ${*}M$  light-like;
 \item  $\su(2)\times\R$, $\Sl(2,\R)\times\R$ or $\lie{d}'_{4,0}$, for $\lambda\neq0$ and ${*}M$ not isotropic.
\end{itemize}
The three lists are summarized in Table~\ref{table:4dimHarm}.
\end{corollary}
\begin{proof}
The proof is entirely similar to Corollary~\ref{cor:4dimHarmful}.

If $\lambda=0$ we proceed in the exact same way to obtain the first part of the statement.

If $\lambda\ne0$, then both the families of Lie algebras appearing in Proposition~\ref{prop:31Mlightlike} and Proposition~\ref{prop:31FourDimcijknew} are of the form appearing in Lemma~\ref{lemma:algebreHarmful40lamdanonzero} by interchanging $e_1$ with $e_3$ and $e_2$ with $e_4$.

In the case of the family of Proposition~\ref{prop:31Mlightlike},  the discriminating quantities  of Lemma~\ref{lemma:algebreHarmful40lamdanonzero} are
\[
c^2+d^2=\frac{8 x^2 \left(4 x^8-11 x^4+9\right)}{\left(x^4-1\right)^2}>0,\quad ad+bc=\frac{32 x^4}{\left(x^4-1\right)^2}>0,\]
thus we only obtain $\Sl(2,\R)\times\R$.

For the family of Proposition~\ref{prop:31FourDimcijknew}, observe that $c$ and $d$ cannot vanish simultaneously, since
\[c  ( y^2+1)  + d(y^2-1)=\frac{6\epsilon}x (y^2+\epsilon x^2)\neq0,\]
where $y^2+\epsilon x^2\neq0$ because $\lambda\neq0$. If $\epsilon=1$, we have
\[ad+bc=8\,\frac{4x^2y^8+y^6-6x^2y^4+x^4y^2+4x^2}{ x^2\left(y^4-1\right)^2}>0,\]
since $4x^2y^8+4x^2\geq 8x^2y^4$ and $2x^2y^4+x^4y^2+y^6=y^2(x^2+y^2)^2$; thus again we get $\Sl(2,\R)\times\R$. For $\epsilon=-1$,
\[ad+bc=8\,\frac{(x^2y-2xy^4+2x-y^3)(x^2y+2xy^4-2x-y^3)}{x^2\left(y^4-1\right)^2},\]
which can be positive, negative or zero: on the curve $xy=1$ it equals $-24$, whereas for fixed $x\neq1$ it tends to $+\infty$ as $y\to1$ (only the point $y=1$ itself is excluded), so by continuity it also vanishes somewhere. At any zero we have $a^2+b^2\neq0$: indeed
\[a(y^2-1)+b(y^2+1)=\frac 8x(x^2y^2-1),\]
so $a=b=0$ would force $x^2y^2=1$, where $a=2x+2x^{-1}\neq0$.
Thus, this family can be $\Sl(2,\R)\times\R$, $\su(2)\times\R$ or $\lie{d}'_{4,0}$.
\end{proof}

\appendix
\section{Clifford multiplication}
\label{appendix:A}
This appendix contains  explicit formulae for Clifford multiplication in a basis in arbitrary dimension, from which the special cases given in Section~\ref{sec:Clifford} can be deduced. We  follow~\cite{Baum_Kath_1999}, with minimal variations in order to adhere with~\eqref{eqn:evenvolumeacts} and~\eqref{eqn:oddvolumeacts}. Let $\{u(1),u(-1)\}$ be a basis of $\C^2$ and consider the linear endomorphisms of $\C^2$ that in the basis $\{u(1),u(-1)\}$ are represented by
\[E=\begin{pmatrix} 1 & 0 \\ 0 & 1 \end{pmatrix}, \quad
T=\begin{pmatrix} -1& 0 \\ 0 & 1 \end{pmatrix}, \quad
U=\begin{pmatrix} 0 & i \\ i & 0 \end{pmatrix}, \quad
V=\begin{pmatrix} 0 & -1 \\ 1 & 0 \end{pmatrix}.\]
Set $m=[n/2]$, and write
\[u(\sigma_m,\dotsc, \sigma_1)=u(\sigma_m)\otimes \dots \otimes u(\sigma_1)\in (\C^2)^{\otimes^m}.\]
We define a bilinear map
\[\R^{p,q}\otimes(\C^2)^{\otimes^m}\to (\C^2)^{\otimes^m}, \quad v\otimes u\mapsto \Phi(v)u\] that satisfies the Clifford identity $\Phi(v)\Phi(w)+\Phi(w)\Phi(v)=-2g(v,w)$, thus realizing $\Sigma$ as $(\C^2)^{\otimes^m}$ and Clifford multiplication as $\Phi$.

For $k\leq m$, we set
\begin{gather*}
    \Phi(e_{2k-1})=\tau_{2k-1}E\otimes \dots \otimes E \otimes U\otimes \underbrace{T\otimes \dots \otimes T}_{k-1 \text{ times}},\\
    \Phi(e_{2k})=\tau_{2k}E\otimes \dots \otimes E \otimes V\otimes \underbrace{T\otimes \dots \otimes T}_{k-1 \text{ times}};
\end{gather*}
if $n=2m+1$ is odd, we set
\[\Phi(e_n)= i^{(n^2+1)/2}\tau_{n}(-T)^{\otimes^m}.\]
With this choice of sign, \eqref{eqn:oddvolumeacts} is satisfied. Notice that for $n$ even, the volume element acts as $i^{n/2+q}(-T)^{\otimes^m}$, and $(-T)^{\otimes m}$ acts as multiplication by $\prod\sigma_t$ on $u(\sigma_m,\dotsc, \sigma_1)$; using~\eqref{eqn:evenvolumeacts}, we see that  $\Sigma_+$ is spanned by the $u(\sigma_m,\dotsc,\sigma_1)$ such that $\prod\sigma_t=(-1)^{n/2}$.

We can summarize these formulae by writing
\[\Phi(e_j)u(\sigma_m,\dotsm, \sigma_1)=\alpha_j(-1)^{[\frac{j-1}2]}\tau_j \prod_{a\leq j/2} \sigma_a u(\tilde\sigma_m, \dotsc, \tilde\sigma_1), \]
where
\[\alpha_j=\begin{cases}                    1 & j \text{ even}\\
-i& j=n \text{ and $n\equiv 3\mod 4$}\\
i  & \text { otherwise}
\end{cases}, \quad
\tilde\sigma_a=\begin{cases}-\sigma_a & a=[(j+1)/2]\\
                    \sigma_a & \text{otherwise}.
                   \end{cases}
                   \]
In order to identify elements of the basis $\{ u(\sigma_m,\dots,\sigma_1)\}$ by a single integer index $0\leq h<2^m$, we  write
\begin{equation*}
\label{eqn:orderofspinors}
u_h=u((-1)^{a_m},\dotsc, (-1)^{a_0}), \quad  \quad h=\sum a_k2^k,
\end{equation*}
so that e.g. for $m=2$
\[u_0=u(1,1), \quad u_1=u(1,-1), \quad u_2=u(-1,1), \quad u_3=u(-1,-1).\]
Specializing to $n=3,4$, these choices lead to~\eqref{eqn:cliff3dim}, \eqref{eqn:cliff40} and~\eqref{eqn:cliff31}.

\section{The Dirac current and $*\xi$}
\label{appendix:diraccurrent}
In this appendix we recall the definition of Dirac current and illustrate the relation with the vector $*\xi$ considered for Lorentzian four-dimensional Lie algebras.

Fixing the signature to $(3,1)$, we can define the hermitian product on spinors $\langle\,,\,\rangle$ by
\[
\langle\psi,\phi\rangle=(e_4\cdot\psi,\phi),\]
where $(\,,\,)$ is the positive definite hermitian product such that $u_0,\dots,u_3$ is an orthonormal basis (see~\cite{leistner2001}). It is a direct computation to verify that
\[\langle \psi,\phi\rangle = \langle \psi_+,\phi_-\rangle + \langle \psi_-,\phi_+\rangle,\qquad\langle X\cdot\psi,\phi\rangle = \langle \psi,X\cdot\phi\rangle.\]

The Dirac current of $\psi\in\Sigma_{3,1}$ is the vector $v_\psi\in \R^{3,1}$ determined by the relation
\begin{equation}\label{eqn:diraccurentDef}
	g(v_\psi,X)=-\langle X\cdot\psi,\psi\rangle
\end{equation}

\begin{example}
Let $\psi=u_0+\alpha u_1$, $\alpha\in\C$,  if $X=\sum_jx^je_j$ then
\[
\begin{split}
	\langle X\cdot\psi,\psi\rangle&=(e_4\cdot X\cdot(u_0+\alpha u_1),u_0+\alpha u_1)\\
	&=(-(x^1e_1+x^2e_2+x^3e_3)\cdot e_4\cdot(u_0+\alpha u_1),u_0+\alpha u_1)+x^4(1+\abs{\alpha}^2)\\
	&=-x^3(e_3\cdot e_4\cdot(u_0+\alpha u_1),u_0+\alpha u_1)+x^4(1+\abs{\alpha}^2)\\
	&=(1+\abs{\alpha}^2)(x^3+x^4),
\end{split}\]
by~\eqref{eqn:cliff31} and the fact that $\{u_j\}$ is an orthonormal basis for the hermitian product $(\,,\,)$. Thus, by~\eqref{eqn:diraccurentDef},
\[
v_\psi=-(1+\abs{\alpha}^2)(e_3-e_4).\]
This computation shows that the spinors appearing in Theorem~\ref{thm:killing31} have light-like Dirac current.
\end{example}

\begin{proposition}
If $\psi\in\Sigma_{3,1}$ is non-zero and satisfies $\xi\cdot\psi=\psi$ for some $\xi\in\Lambda^3(\R^{3,1})$, then
\[v_\psi=2g(*\xi,e_4)(\psi_+,\psi_+){*}\xi;\]
in particular, $v_\psi$ is non-zero.
\end{proposition}
\begin{proof}
The Dirac current of a spinor $\psi$ satisfies \[v_\psi=v_{\psi_+}+v_{\psi_-}.\]
Indeed, since $X\cdot\psi_\pm\in\Sigma^\mp$ and $\langle \psi,\phi\rangle = \langle \psi_+,\phi_-\rangle + \langle \psi_-,\phi_+\rangle$, we have
\[
\begin{split}
	g(v_\psi,X)&=-\langle X\cdot\psi,\psi\rangle=-\langle X\cdot(\psi_++\psi_-),\psi_++\psi_-\rangle\\
  &=-\langle X\cdot\psi_+,\psi_+\rangle-\langle X\cdot\psi_-,\psi_-\rangle\\
	&=g(v_{\psi_+},X)+g(v_{\psi_-},X).
\end{split}
\]
Furthermore, since~\eqref{eqn:nonlethal4} holds, we have
\[
\begin{split}
	g(v_{\psi_-},X)&=-\langle X\cdot\psi_-,\psi_-\rangle=-\langle X\cdot{*}\xi\cdot \psi_+,{*}\xi\cdot \psi_+\rangle\\
	&=2g(X,{*}\xi)\langle\psi_+,{*}\xi\cdot\psi_+\rangle+\langle {*}\xi\cdot X\cdot\psi_+,{*}\xi\cdot \psi_+\rangle\\
	&=2g(X,{*}\xi)\langle{*}\xi\cdot\psi_+,\psi_+\rangle-g({*}\xi,{*}\xi)\langle X\cdot\psi_+,\psi_+\rangle\\
	&=-2g(X,{*}\xi)g(v_{\psi_+},{*}\xi)-g(v_{\psi_+},X).
\end{split}\]
It follows, by substituting $X={*}\xi$, that $g(v_{\psi_-},{*}\xi)=g(v_{\psi_+},{*}\xi)$, thus
\begin{equation*}
	v_\psi=-g(v_\psi,{*}\xi){*}\xi.
\end{equation*}
Moreover, if~\eqref{eqn:nonlethal4} holds then $v_\psi$ is never zero. Indeed
\[
\begin{split}
	g(v_\psi,{*}\xi)&=-\langle *\xi\cdot\psi,\psi\rangle =-i\langle  \psi_--\psi_+,\psi_++\psi_-\rangle
	=-i(\langle \psi_-,\psi_+\rangle  - \overline{\langle \psi_-,\psi_+\rangle})\\
	&=2\imp\langle \psi_-,\psi_+\rangle =2\imp\langle i{*}\xi\cdot\psi_+,\psi_+\rangle =2\rep(e_4\cdot {*}\xi\cdot\psi_+,\psi_+)\\
	&=\rep(e_4\cdot {*}\xi\cdot\psi_+,\psi_+)-\rep(*\xi\cdot e_4\cdot\psi_+,\psi_+) - 2g(*\xi,e_4)\rep(\psi_+,\psi_+)\\
	&=\rep(e_4\cdot {*}\xi\cdot\psi_+,\psi_+)-\rep\overline{(e_4\cdot {*}\xi\cdot\psi_+,\psi_+)}-2g(*\xi,e_4)(\psi_+,\psi_+)\\
	&=-2g(*\xi,e_4)(\psi_+,\psi_+),
\end{split}\]
giving the first part of the statement. Furthermore,  $g(*\xi,e_4)\ne0$ because $*\xi$ is time-like, and $(\psi_+,\psi_+)\ne0$ since $(\,,\,)$ is positive definite and $\psi$ is non-chiral by~\eqref{eqn:nonlethal4}; hence, $v_\psi$ is non-zero.
\end{proof}

\section{Unimodular Lie algebras of dimension~$3$ and~$4$}
\label{appendix:liealgebras}
This appendix contains several lemmas aimed at comparing the Lie algebras obtained along the paper with the known classification of unimodular Lie algebras of dimensions $3$ and $4$ (see~\cite{Ova06}). Four-dimensional unimodular Lie algebras are listed in Table~\ref{table:4dimHarm}; the analogous list for three dimensions is given in Table~\ref{table:3dim}.
\begin{table}[th]
\centering
\caption{$3$-dimensional unimodular Lie algebras.}\label{table:3dim}
\begin{tabular}{L L C L}
\toprule
\text{Name} & \g\\
\midrule
\R^3 & (0,0,0,) \\
\lie{h}_3 & (0,0,e^{12}) \\
\lie{r}\lie{r}_{3,-1} & (0,-e^{12},e^{13})\\
\lie{r}\lie{r}'_{3,0} & (0,-e^{13},e^{12})  \\
\su(2)& (-e^{23},-e^{31},-e^{12}) \\
\Sl(2,\R)&(-e^{23},e^{31},e^{12})\\
\bottomrule
\end{tabular}
\end{table}

\begin{lemma}\label{lem:ClassificazioneBaseDim3}
Consider a Lie algebra $\g$ of the form
\[(-ae^{23},-be^{31},-ce^{12}),\]
for $a,b,c\in\R$. Then $\g$ is isomorphic to one of the following:
\begin{enumerate}
 \item $\su(2)$ if $a,b,c$ are positive, or all of them are negative;
 \item $\Sl(2,\R)$ if $a,b,c$ are non-zero and have mixed signs;
 \item $\lie{r}'_{3,0}$ if one parameter is zero and the other two have the same sign;
 \item $\lie{r}_{3,-1}$ if one parameter is zero and the other two have different sign;
 \item $\lie{h}_{3}$ if only one parameter is non zero and the others vanish;
 \item $\R^3$ if $a=b=c=0$.
\end{enumerate}
\end{lemma}
\begin{proof}
We consider the following cases.
\begin{enumerate}
\item If $a,b,c$ are non-zero, by setting
\[
\hat{e}^1=\sqrt{\abs{ b} } \sqrt{\abs{c} }e^1,\qquad \hat{e}^2=\sqrt{\abs{ a} } \sqrt{\abs{c} }e^2,\qquad
\hat{e}^3=\sqrt{\abs{ a} } \sqrt{\abs{b} }e^3;
\]
the Lie algebra has the form $\big(-\frac{a}{ \abs{a} }\hat{e}^{23},-\frac{b}{\abs{b}}\hat{e}^{31},-\frac{c}{\abs{c}}\hat{e}^{12}\big)$, thus it is $\su(2)$ or $\Sl(2,\R)$, depending on the signs of $a,b$ and $c$.

\item Assume now that one of the parameters is zero, say $a=0$, and $b$, $c$ have the same sign. Up to replacing $e_1$ with $-e_1$ we can assume $b,c>0$. Then, relative to the basis
\begin{gather*}
\hat{e}^1=-\sqrt{bc}\,e^1,\\
\hat{e}^2=-\frac{1}{2 b  \sqrt{c }}e^2-\frac{1}{2  c  \sqrt{b }}e^3,\qquad
\hat{e}^3=\frac{1}{2  b  \sqrt{{c} }}e^2-\frac{1}{2 c  \sqrt{ {b} }}e^3,
\end{gather*}
the Lie algebra takes the form
\[(0,\hat e^{31}, \hat e^{12}),\]
and so is isomorphic to $\lie{r}'_{3,0}$.
\item If one of the parameters is zero and the other two have different sign, we can assume $a=0$, $b>0$ and $c<0$; setting
\begin{gather*}
\hat{e}^1=-\sqrt{-bc}\,e^1,\\
\hat{e}^2=-\frac{1}{2 b  \sqrt{-c }}e^2+\frac{1}{2  c  \sqrt{b }}e^3,\qquad
\hat{e}^3=\frac{1}{2  b  \sqrt{{-c} }}e^2+\frac{1}{2 c  \sqrt{ {b} }}e^3,
\end{gather*}
we obtain
\[(0, -\hat e^{12}, \hat e^{13}),\]
hence $\lie r_{3,-1}$.
\item If two parameters are zero and one is non-zero, the Lie algebra is $2$-step nilpotent, hence isomorpic to the Heisenberg Lie algebra $\lie{h}_3$.
\item If $a=b=c=0$ the Lie algebra is abelian, hence isomorphic to $\R^3$.\qedhere
\end{enumerate}
\end{proof}

\begin{lemma}\label{lem:ClassificazioneBaseDim4}
Consider a Lie algebra $\g$ of the form
\[(-ae^{24},-be^{14},-ce^{12},0),\]
for $a,b,c\in\R$. Then $\g$ is isomorphic to one of the following:
\begin{enumerate}
 \item $\lie{d}_{4}$ if $c$ is non-zero and $ab>0$;
 \item $\lie{d}'_{4,0}$ if $c$ is non-zero and $ab<0$;
 \item $\lie{r}\lie{r}'_{3,0}$ if $c=0$ and $ab>0$;
 \item $\lie{r}\lie{r}_{3,-1}$ if $c=0$ and $ab<0$;
 \item $\lie{n}_{4}$ if $a=0$ (or $b=0$) and the other parameters are nonzero;
 \item $\lie{r}\lie{h}_{3}$ if only one parameter is nonzero;
 \item $\R^4$ if $a=b=c=0$.
\end{enumerate}
\end{lemma}
\begin{proof}
\begin{enumerate}
\item If $c$ is non-zero and $ab>0$ we can assume that $a,b,c$ are positive up to replacing $e_3$, $e_4$ with their opposites. Relative to the basis
\begin{gather*}
\hat{e}^1=-\frac{1}{2 a \sqrt{ b }}e^1+\frac{1}{2  b  \sqrt{a }}e^2,\qquad
\hat{e}^2=\frac{1}{2 a  \sqrt{ b }}e^1+\frac{1}{2  b  \sqrt{ a }}e^2,\\
\hat{e}^3=-\frac{1}{2 c a^{3/2} b^{3/2}  }e^3,\qquad \hat{e}^4=\sqrt{ ab}  e^4,
\end{gather*}
the Lie algebra can be written as
\[(\hat e^{14}, -\hat e^{24}, -\hat e^{12},0),\]
and so is isomorphic to $\lie{d}_{4}$.
\item If $c$ is non-zero and $ab<0$ we can assume that $b,c>0$ and $a<0$. Setting
\begin{gather*}
\hat{e}^1=\frac{1}{2 a \sqrt{ b }}e^1+\frac{1}{2  b  \sqrt{-a }}e^2,\qquad
\hat{e}^2=-\frac{1}{2 a  \sqrt{ b }}e^1+\frac{1}{2  b  \sqrt{ -a }}e^2,\\
\hat{e}^3=-\frac{1}{2 c (-a)^{3/2} b^{3/2}  }e^3,\qquad \hat{e}^4=-\sqrt{ -ab}  e^4,
\end{gather*}
we obtain
\[(-\hat e^{24},\hat e^{14}, -\hat e^{12}, 0),\]
and the Lie algebra is isomorphic to $\lie{d}'_{4_0}$.
\item If $c=0$ the Lie algebra is reducible, and we can apply Lemma~\ref{lem:ClassificazioneBaseDim3} to obtain $\lie{r}\lie{r}'_{3,0}$ if $ab>0$ and $\lie{r}\lie{r}_{3,-1}$ if $ab<0$.
\item If $a=0$ (and similarly if $b=0$), the Lie algebra is  $3$-step nilpotent and irreducible, i.e. it is the nilpotent Lie algebra $\lie{n}_4$.
\item If two parameters are zero and one is non-zero, the Lie algebra is $2$-step nilpotent and reducible, hence isomorphic to the Heisenberg Lie algebra $\lie{r}\lie{h}_3$.
\item If $a=b=c=0$ the Lie algebra is abelian, hence isomorphic to $\R^4$.\qedhere
\end{enumerate}
\end{proof}

\begin{lemma}\label{lemma:parallelFamilywithAbelian}
Consider a Lie algebra of the form
\begin{align*}
	&\bigl(a(e^{13}+e^{14})+b(e^{23}+e^{24}),b(e^{13}+e^{14})-a(e^{23}+e^{24}),\\
  &-c(e^{13}+e^{14})-d(e^{23}+e^{24}),c(e^{13}+e^{14})+d(e^{23}+e^{24})\bigr)
\end{align*}
for $a,b,c,d\in\R$. Then $\g$ is isomorphic to one of the following:
\begin{enumerate}
	\item $\lie{rr}_{3,-1}$ for $a^2+b^2\neq0$;
	\item $\lie{rh}_3$ for $a=b=0$ and $c^2+d^2\neq0$;
	\item $\R^4$ for $a=b=c=d=0$.
\end{enumerate}
\end{lemma}
\begin{proof}
We prove each case separately.
\begin{enumerate}
	\item If $b=0$ and $a\neq0$, the Lie algebra is isomorphic to $\lie{r}\lie{r}_{3,-1}$ by choosing the following basis of $\g^*$:
	\begin{gather*}
		\hat{e}^1=-a(e^3+e^4),\qquad \hat{e}^2=e^2,\qquad
		\hat{e}^3=e^1,\qquad \hat{e}^4=ce^1-de^2+ae^3.
	\end{gather*}

	If $b\neq0$, again we obtain
	$\lie{r}\lie{r}_{3,-1}$ by considering the basis
	\begin{gather*}
		\hat{e}^1=-\sqrt{a^2+b^2} (e^3+e^4),\qquad
		\hat{e}^2=\frac{-a+\sqrt{a^2+b^2}}{2b}e^1 -\frac{1}{2} e^2,\\
\hat{e}^3=-\frac{a+ \sqrt{a^2+b^2}}{2b}e^1 -\frac{1}{2}e^2,\qquad
		\hat{e}^4=
		(ac+bd)e^1+(bc-ad)e^2+(a^2+b^2) e^3.\qedhere
	\end{gather*}
	\item If $a=b=0$ and $c^2+d^2\neq0$, the Lie algebra has the form of an orthogonal semidirect product $\Span{e_3,e_4}\rtimes\Span{e_1,e_2}$, where $\ad e_1$, $\ad e_2$ are multiples of $(e^3+e^4)\otimes (e_3-e_4)$. It is clear that the resulting Lie algebra is isomorphic to $\lie{r}\lie{h}_3$.
	\item If $a=b=c=d=0$, the Lie algebra is clearly abelian.
\end{enumerate}
\end{proof}

\begin{lemma}\label{lemma:parallelFamilyNoAbelian}
Consider a Lie algebra of the form
\begin{align*}
	&\bigl(a(e^{13}+e^{14})-2e^{24}-2e^{23},-a(e^{24}+e^{23})+2e^{13}+2e^{14},\\
  &-b(e^{13}+e^{14})-c(e^{23}+e^{24})-4e^{12},b(e^{13}+e^{14})+c(e^{24}+e^{23})+4e^{12}\bigr)
\end{align*}
for $a,b,c\in\R$. Then $\g$ is isomorphic to one of the following:
\begin{enumerate}
	\item $\lie{d}_4$ for $a^2-4>0$;
	\item $\lie{d}'_{4,0}$ for $a^2-4<0$;
	\item $\lie{n}_4$ for $a^2-4=0$.
	\end{enumerate}
\end{lemma}
\begin{proof}
As in previous lemmas, we discuss each case separately.
\begin{enumerate}
	\item If $a^2-4>0$, then the Lie algebra is isomorphic to $\lie{d}_{4}$ via
	\begin{gather*}
		\hat{e}^1=(-a-\sqrt{a^2-4})e^1 + 2e^2,\qquad
		\hat{e}^2=(-a+\sqrt{a^2-4})e^1 + 2e^2,\\
		\hat{e}^3=-\frac{2c+ab}{\sqrt{a^2-4}}e^1+\frac{ac+2b}{\sqrt{a^2-4}}e^2-\sqrt{a^2-4}e^3,\qquad
		\hat{e}^4=\sqrt{a^2-4} (e^3+e^4).
	\end{gather*}
	\item The case $a^2-4<0$ results in $\lie{d}'_{4,0}$ by choosing
	\begin{gather*}
		\hat{e}^1=-\sqrt{4-a^2}e^1,\qquad
		\hat{e}^2=-ae^1 +2e^2,\\
		\hat{e}^3=\frac{2c+ab}{2\sqrt{4-a^2}} e^1 -\frac{ac+2b}{2\sqrt{4-a^2}}e^2-\frac{1}{2} \sqrt{4-a^2}e^3,\qquad
		\hat{e}^4=\sqrt{4-a^2} (e^3+e^4).\qedhere
	\end{gather*}
	\item If $a^2-4=0$, an isomorphism with $\lie{n}_{4}$ is obtained by setting
	\begin{gather*}
		\hat{e}^1=e^3+e^4,\qquad\hat{e}^2=e^1\\
		\hat{e}^3=\frac{c}{4}e^1 -\frac{1}{2}e^3,\qquad\hat{e}^4=-ae^1+2e^2+\frac{1}{4}(ac+2b)(e^3+e^4).
	\end{gather*}
\end{enumerate}
\end{proof}

\begin{lemma}\label{lemma:algebreHarmful40lamdanonzero}
Consider a Lie algebra $\g$ of the form
\[(ae^{34},be^{34},ce^{24}+de^{14},-ce^{23}-de^{13})
\]
for $a,b,c,d\in\R$. Then $\g$ is isomorphic to one of the following:
\begin{enumerate}
	\item $\su(2)\times\R$ for $c^2+d^2\ne0$ and $ad+bc<0$;
	\item $\Sl(2,\R)\times\R$ for $c^2+d^2\ne0$ and $ad+bc>0$;
	\item $\lie{d}'_{4,0}$ for $c^2+d^2\ne0$ and $ad+bc=0\ne a^2+b^2$;
	\item $\lie{rr}_{3,-1}$ for $c^2+d^2\ne0$ and $ad+bc=a^2+b^2=0$;
	\item $\lie{rh}_3$ for $c^2+d^2=0\neq a^2+b^2$;
    \item $\R^4$ if $a=b=c=d=0$.
\end{enumerate}
\end{lemma}
\begin{proof}
We will show a basis $\{\hat{e}^1,\dots,\hat{e}^4\}$ of $\g^*$ that gives the claimed isomorphism.
\begin{enumerate}
	\item If $c^2+d^2\neq0$ and $ad+bc\neq0$, we choose the following basis of
	$\g^*$:
		\begin{gather*}
			\hat{e}^1=e^4,\qquad
			\hat{e}^2=\frac{1}{ad+bc}(de^1+ce^2),\\
			\hat{e}^3=e^3,\qquad
			\hat{e}^4=\frac{1}{ad+bc}(be^1-ae^2).
		\end{gather*}
		With respect to this basis, the structure equations become
		\[\left( - (ad+bc)\hat e^{23},\hat e^{31},-(ad+bc)\hat e^{12},0\right).\]
		Proceeding as in Lemma~\ref{lem:ClassificazioneBaseDim3} we obtain $\Sl(2,\R)\times\R$ if $ad+bc>0$ and $\su(2)\times\R$ if $ad+bc<0$.

 	\item If $c^2+d^2\neq0$ and $ad+bc=0$, we set
	\begin{gather*}
			\hat{e}^1=e^3,\qquad
			\hat{e}^2=e^4,\\
			\hat{e}^3=ce^1-de^2,\qquad
			\hat{e}^4=de^1+ce^2.
		\end{gather*}
		The structure equations become
		\[\left(-\hat e^{24},\hat e^{14},(ac-bd)\hat e^{12},0\right).\]
	Observe that the conditions $ad+bc=0$ and $c^2+d^2\neq0$ force $(a,b)$ to be a multiple of $(c,-d)$, so that $ac-bd$ vanishes if and only if $a^2+b^2$ does. If $ac-bd\ne0$, since the structure coefficient of the new basis satisfies $\hat{c}_{241}=-\hat{c}_{142}$, applying Lemma~\ref{lem:ClassificazioneBaseDim4} we obtain $\lie{d}'_{4,0}$. On the other hand, if $ac-bd=0$, again by applying Lemma~\ref{lem:ClassificazioneBaseDim4}, we get $\lie{rr}_{3,-1}$.
	\item Finally, if $c^2+d^2=0$ then  the structure equations become
	\[(ae^{34},be^{34},0,0).\]
	If $a^2+b^2\ne0$ then $\g\cong\lie r\lie h_3$ by considering
    \[\hat e^1=e^3,\quad \hat e^2=e^4, \quad \hat e^3=\frac1{a^2+b^2}(ae^1+be^2),\quad\hat{e}^4=be^1-ae^2;\]
    otherwise, it is abelian.
\end{enumerate}
\end{proof}

\printbibliography
\medskip
\small\noindent D. Conti: Dipartimento di Matematica, Università di Pisa, largo Bruno Pontecorvo 6, 56127 Pisa, Italy.\\
\small\noindent F. A. Rossi: Dipartimento di Matematica e Informatica, Universit\`a degli studi di Perugia, via Vanvitelli 1, 06123 Perugia, Italy.\\
\small\noindent R.~Segnan Dalmasso: Dipartimento di Matematica e Applicazioni, Universit\`a di Milano Bicocca, via Cozzi 55, 20125 Milano, Italy.\\

\end{document}